\documentclass[oneside,english]{amsart}
\usepackage[latin9]{inputenc}
\usepackage{geometry}
\usepackage{amstext}
\usepackage{amsthm}
\usepackage{amssymb}
\usepackage{hyperref}

\usepackage{tikz}
\usepackage{subfigure}

\newtheorem{theor}{Theorem}[section]
\newtheorem*{theorem*}{Theorem}

\newtheorem{lemma}[theor]{Lemma}
\newtheorem{cor}[theor]{Corollary}
\newtheorem{defi}[theor]{Definition}
\newtheorem{prop}[theor]{Proposition}

\theoremstyle{definition}
\newtheorem{rem}[theor]{Remark}

\theoremstyle{plain}

\newtheorem{Problem}{Problem}[section]

\newcommand{\N}{\mathbb{N}}

\newcommand{\Event}{\mathcal{E}}
\def\Prob{{\mathbb P}}

\usepackage{xcolor}
\definecolor{b}{HTML}{4472c4}
\definecolor{o}{HTML}{ED7D31}
\definecolor{g}{HTML}{70ad47}
\definecolor{t}{RGB}{40,154,150}
\newcommand\HL[1]{#1}
\newcommand\co[1]{{#1}}

\newcommand\snb[1]{ \stackrel{#1}{\sim}}

\def\Cgl{{\mathcal C}}

\title{Shotgun assembly of unlabeled Erd\H os--R\'enyi graphs}

\author{Han Huang
\email{han.huang@math.gatech.edu}
}
\author{Konstantin Tikhomirov
\email{ktikhomi@andrew.cmu.edu}
}
\thanks{K.T. is partially supported by the Sloan Research Fellowship and NSF grant DMS-2054666}

\def\N{{\mathbb N}}

\def\Prob{{\mathbb P}}
\def\Exp{{\mathbb E}}
\def\Event{{\mathcal E}}

\def\deg{{\rm deg}}

\date{\today}

\def\cal{\mathcal}

\def\maj{vast majority}
\def\S_2{{[n] \choose 2}}
\begin{document}

\maketitle

\begin{abstract}
Given a positive integer
$n$, an unlabeled graph $G$ on $n$ vertices, \co{and a vertex
$v$ of $G$, let  $N_G(v)$ be the subgraph of $G$ induced by vertices of $G$ of distance at most one from $v$.} We show that there are universal constants $C,c>0$ with the following property. 
Let the sequence $(p_n)_{n=1}^\infty$ satisfy $n^{-1/2}\log^C n\leq p_n\leq c$.
For each $n$, let $\Gamma_n$ be an unlabeled $G(n,p_n)$ Erd\H os--R\'enyi graph.
Then with probability $1-o_n(1)$, any unlabeled graph $\tilde \Gamma_n$ on $n$ vertices with
$\{N_{\tilde \Gamma_n}(v)\}_{v}=\{N_{\Gamma_n}(v)\}_{v}$ must coincide with $\Gamma_n$.
This establishes $\tilde \Theta(n^{-1/2})$
as the transition range for the density parameter $p_n$ between reconstructability and non-reconstructability
of Erd\H os--R\'enyi graphs from their $1$--neighborhoods, and
resolves a problem of Gaudio and Mossel from \cite{GM}.
\end{abstract}

\section{Introduction}

The problem of reconstructing a labeled or unlabeled random graph from vertex neighborhoods of a given radius $r$
\co{was} introduced by Mossel and Ross in \cite{MosselRoss}
and has been a subject of very active research since then.
This problem can be viewed as an abstract form of the well studied {\it shotgun assembly} of DNA, \co{which is the problem of reconstructing DNA sequences by observing
a family of short subsequences possibly corrupted by noise}
(see, in particular, \cite{AMRW,MBT} for information theoretical aspects of DNA reconstruction).
The problem of graph assembly can be viewed as a question whether
the local structure of a graph contains all the information about its global structure.
Since its introduction in \cite{MosselRoss}, some specific instances of random graph assembly have been studied
in works \cite{BBN, BFM, GM, M, MS, NPS, PRS, DJM, DL, GRS, JKRS}.
The problem is closely related to the famous Graph Reconstruction Conjecture\co{, which states that for any graph $G$ with at least 3 vertices, $G$
can be uniquely reconstructed (up to an isomorphism)
from the multiset of induced
subgraphs of $G$ obtained by removing a single vertex from $G$ }\cite{Kelly,Bondy,Ulam}.

Let us give a precise formulation of the problem we are considering here.
Given two graphs $G_1$ and $G_2$ with the vertex set $[n]:=\{1,2,\dots,n\}$, we say
that $G_1$ and $G_2$ are isomorphic if there is a bijective mapping $f:[n]\to[n]$ such that for every pair of indices $i\neq j$
from $[n]$, $\{i,j\}$ is an edge of $G_1$ if and only if $\{f(i),f(j)\}$ is an edge of $G_2$.
The isomorphism is an equivalence relation which splits the set of graphs on $[n]$ into equivalence classes.
We will refer to an equivalence class with respect to this relation as an {\it unlabeled graph on $n$ vertices}.

Let $r>0$ be a fixed integer. 
Given an unlabeled graph $G$ on $n$ vertices, we define the $n$--multiset of vertex $r$--neighborhoods of $G$
as follows. Let $G'\in G$ be any labeled representative of the equivalence class $G$. For each vertex $v$ of $G'$,
we define {\it the $r$--neighborhood $N_{G',r}(v)$ of $v$}
as the subgraph induced by vertices of $G'$ at distance at most $r$ from $v$. We then define the $n$--multiset
$\{N_{G,r}(v)\}_v$ of vertex $r$--neighborhoods of $G$ as the multiset of the graph equivalence classes represented by $N_{G',r}(v)$, $v=1,2,\dots,n$. Note that the definition is consistent in the sense that it does not depend on the choice of the
representative $G'$.
We want to emphasize here that everywhere in this note the term ``$r$--neighborhood'' refers to an induced subgraph
rather than just a subset of vertices.

Given an unlabeled graph $G$ on $n$ vertices, we say that $G$ is {\it reconstructable} from its $r$--neighborhoods
if the following holds: whenever $\tilde G$ is an unlabeled graph on $n$ vertices such that the $n$--multisets
of vertex $r$--neighborhoods of $G$ and $\tilde G$ coincide, we necessarily have $G=\tilde G$.
Otherwise, we will say that the graph $G$ is {\it non-reconstructable}.

The main goal of this note is to study reconstructability of unlabeled Erd\H os--R\'enyi $G(n,p)$ graphs.
We recall that, given parameters $n$ and $p$, the $G(n,p)$ graph on the vertex set $[n]$ is defined
by drawing an edge between any given pair of vertices $v,w$ with probability $p$, independently from other edges.
By the unlabeled $G(n,p)$ graph we understand the (random) equivalence class of the random labeled
$G(n,p)$ graph on $[n]$.
The major question on graph reconstructability is

\medskip

\begin{center}
    \bf for which values of $p$ and $r$  is the unlabeled $G(n,p)$ graph reconstructable with high probability?
\end{center}

\medskip

Below, we focus on the setting $r=1$ ($1$--neighborhoods).
In work \cite{GM}, Gaudio and Mossel showed that for any constant $\varepsilon>0$
and for any positive sequence $(p_n)_{n=1}^\infty$ satisfying $n^{-1+\varepsilon}\leq p_n\leq n^{-1/2-\varepsilon}$,
the unlabeled $G(n,p_n)$ graph is asymptotically almost surely (a.a.s) non-reconstructable from its $1$--neighborhoods
(that is, the probability that the graph is non-reconstructable tends to one as $n\to\infty$).
In fact, by using a slightly more refined argument, it is possible to show that 
$G(n,p_n)$ is a.a.s non-reconstructable whenever $\omega(n^{-1}\log n)= p_n=o(n^{-1/2})$,
where we use the standard notation $\omega(g)$ to denote a non-negative quantity such that $\frac{\omega(g)}{g}\to \infty$
as $n$ tends to infinity.
\begin{theor}[{A slight refinement of \cite[Theorem~3]{GM}}]\label{3u14o21u4oi}
Let the sequence $(p_n)_{n=1}^\infty$ satisfy $p_n=\omega(n^{-1}\log n)$ and $p_n=o(n^{-1/2})$.
Then the unlabeled $G(n,p_n)$ graph is asymptotically almost surely non-reconstructable from its $1$--neighborhoods.
\end{theor}
For completeness, we give a proof of Theorem~\ref{3u14o21u4oi} in Section~\ref{s:entropy}. We remark here that after this manuscript was posted, a further improvement of \cite[Theorem~3]{GM} was obtained in \cite{JKRS}.

\medskip

On the other hand, using the fingerprinting technique, 
which relates
an edge $\{u,v\}$ of a graph to
the induced subgraph on the set of common neighbors of $u$ and $v$ (so the subgraph 
appears in both $1$-neighbors of $v$ and $u$), the authors of \cite{GM} showed that for $n^{-1/3+\varepsilon}\leq p_n\leq n^{-\varepsilon}$,
the $G(n,p_n)$ graph is a.a.s reconstructable from $1$--neighborhoods.
Note that the lower bound $n^{-1/3+o(1)}$ does not match
the aforementioned range for non-reconstrucability established in
\cite[Theorem~3]{GM}, and identifying a sharp threshold
has remained an open problem.
The main result of our paper is the following theorem:
\begin{theor}\label{19471048710498}
There are universal constants $C,c>0$ with the following property.
Let positive sequence $(p_n)_{n=1}^\infty$ satisfy $n^{-1/2}\log^{C}
n\leq p_n\leq c$ for large $n$.
Then the unlabeled $G(n,p_n)$ graph is reconstructable from its $1$--neighborhoods
with probability $1-n^{-\omega(1)}$.
\end{theor}
Thus, together with the aforementioned refinement of \cite[Theorem~3]{GM},
the theorem establishes $p=\tilde\Theta(n^{-1/2})$ as the transition point between 
reconstructability and non-reconstructability of unlabeled Erd\H os--R\'enyi graphs
from their $1$--neighborhoods. Here, we use the notation $\tilde \Theta(\cdot)$ to denote an
asymptotically equivalent quantity up to polylog multiples.

\medskip

Before considering the main challenges and ideas of the proof of Theorem~\ref{19471048710498},
let us make a simple ``reduction'' step
which will make the discussion more transparent.
For every labeled graph $G$, denote by $V(G)$ its vertex set and by $E(G)$ its edge set.
Given two graphs $G,G'$ with $i\in V(G)\cap V(G')$,
we say that $G$ and $G'$ are {\it isomorphic with fixed point $i$} if
\begin{itemize}
\item The vertex sets $V(G)$ and $V(G')$ have the same size;
\item There is a bijection $f:V(G)\to V(G')$ such that $\{f(v),f(w)\}\in E(G')$ if and only if $\{v,w\}\in E(G)$,
and such that $f(i)=i$.
\end{itemize}
\HL{It is not difficult to construct
an example of two non-isomorphic labeled graphs on a same vertex set
such that their respective $1$--neighborhoods are isomorphic with fixed point
(see Figure~\ref{f:f1}).}

\HL{Given two labeled graphs $G,G'$
on a same vertex set, we will write $G=G'$ if the identity map on the vertex set
is an isomorphism between $G$ and $G'$.}
We prove the following
\begin{theor}\label{309847109847}
There are universal constants $C,c>0$ with the following property.
Let positive sequence $(p_n)_{n=1}^\infty$ satisfy $n^{-1/2}\log^C n\leq p_n\leq c$ for large $n$, and
for each $n$, let $\Gamma_n$ be the (labeled) Erd\H os--R\'enyi $G(n,p_n)$ graph on $\{1,2,\dots,n\}$.
\co{Then with probability $1-n^{-\omega(1)}$,
for every labeled graph $\tilde \Gamma_n$ on $\{1,2,\dots,n\}$ 
such that $1$-neighborhoods $N_{\Gamma_n,1}(i)$ and $N_{\tilde \Gamma_n,1}(i)$ are isomorphic with fixed point $i$ for all $i\in\{1,2,\dots,n\}$, we have $\Gamma_n = \tilde \Gamma_n$.}
\end{theor}
It can be easily verified that Theorem~\ref{309847109847} implies
Theorem~\ref{19471048710498}; for completeness, we provide an argument
at the end of Section~\ref{s:step4}.
In fact, in the opposite direction,
Theorem~\ref{19471048710498} implies Theorem~\ref{309847109847}
via a simple union bound estimate.
We only highlight the proof idea.
An application of Theorem~\ref{19471048710498}
yields that with probability $1-n^{-\omega(1)}$,
every graph $\tilde \Gamma_n$
satisfying the conditions on the $1$--neighborhoods from Theorem~\ref{309847109847},
must be isomorphic to $\Gamma_n$. On the other hand,
it can be verified via the union bound argument that with probability at least $1-n^{-\omega(1)}$, any pair of $1$-neighborhoods of $\Gamma_n$ are not isomorphic.
Thus, necessarily
$\tilde \Gamma_n=\Gamma_n$ with probability $1-n^{-\omega(1)}$.

\bigskip

The proof of the reconstructability of the $G(n,p_n)$  Erd\H os--R\'enyi graph in the regime $n^{-\varepsilon}\geq p_n\geq n^{-1/3+\varepsilon}$
in the work \cite{GM} is based on the observation that with a high probability, given two adjacent vertices $v,w$ of the graph,
the intersection of their $1$--neighborhoods is a subgraph with a ``rich enough'' edge set. This richness allows
to distinguish adjacent pairs from non-adjacent pairs using a polynomial time algorithm\footnote{\co{Although the computational complexity of the reconstruction procedure is not explicitly analyzed in \cite{GM}, the polynomial time complexity can be verified immediately by applying the main result from \cite{Czajka}. Indeed, it is easy to check that the induced subgraph $G_{v,w}$
of $G(n,p_n)$ on the set of common neighbors of any fixed pair of vertices $v,w$
has a vertex set $V(G_{v,w})$
of size $\Theta(p_n^2 n)$ with probability $1-n^{-\omega(1)}$. On the other hand,
conditioned on $|V(G_{v,w})|=m$ for any $m=\Theta(p_n^2 n)$,
$G_{v,w}$ is [conditionally] Erd\H os--R\'enyi with parameters $m,p_n$ which,
in the regime $p_n\geq n^{-1/3+\varepsilon}$, satisfy $p_n=\omega(\frac{\log^4 m}{m})$. According to the main result of \cite{Czajka}, this implies that with probability $1-n^{-\omega(1)}$ one can define a {\it canonical labeling} of vertices of $G_{v,w}$ in polynomial time which, in turn, enables a polynomial time construction of an isomorphism
between two unlabeled copies of $G_{v,w}$. By ``injecting'' this argument into \cite{GM}, the polynomial time reconstruction procedure is obtained in the considered range of $p_n$.}}.
Under a weaker assumption $p_n\geq n^{-1/2+\varepsilon}$, the $1$--neighborhoods of any two adjacent vertices
still contain many common vertices with a high probability.
It is not difficult to check, however, that in the regime when $p_n$ is much smaller than $n^{-1/3}$,
the edge set of the induced graph on those common vertices is typically empty for most adjacent pairs,
and thus does not give a good test for adjacency.
The approach we present here is completely different from \cite{GM} and is ultimately based
on a careful {\it global} analysis of $3$--cycles in graphs. The argument is a multistep process,
where each next step is designed to strengthen the information available about the graphs
$\tilde\Gamma_n$ satisfying the assumptions on $1$--neighborhoods from Theorem~\ref{309847109847}.

Next, we give an outline of the argument. Since we are focused on reconstruction from $1$--neighborhoods,
in what follows we will use notation $N_{G}(i)$ in place of $N_{G,1}(i)$.
\HL{Unless explicitly specified otherwise, we will work with graphs on $n$ vertices labeled $\{1,2,\dots,n\}$.
Let $\Gamma$ be a $G(n,p)$
Erd\H os--R\'enyi graph on $\{1,2,\dots,n\}$, where we assume that $n^{-1/2}\log^C n
\leq p\leq \log^{-C}n$ for a large universal constant $C>0$.
Next, we construct an auxiliary random graph $\tilde\Gamma$ on the vertex set $\{1,2,\dots,n\}$ as follows.
For every realization of $\Gamma$, we define $\tilde\Gamma$
so that $N_{\Gamma}(i)$ and $N_{\tilde\Gamma}(i)$ are isomorphic with fixed point $i$
for every $1\leq i\leq n$, and, whenever possible,
take $\tilde\Gamma$ not equal to $\Gamma$.
Thus, $\tilde\Gamma$, viewed as a graph-valued random variable,
is a function of $\Gamma$.
Further, with this construction, $\tilde\Gamma$
represents a ``most distant'' from $\Gamma$ random graph
subject to the condition that the respective $1$--neighborhoods are isomorphic with fixed point everywhere on the probability space.
Note that with this definition of $\tilde\Gamma$,
proving Theorem~\ref{309847109847} amounts to checking that
$$
\Prob\big\{\Gamma=\tilde\Gamma\big\}=1-n^{-\omega(1)}.
$$
For the rest of the paper, the notation ``$\tilde\Gamma$'' refers to
the random graph constructed above}\footnote{In our view, working with the single ``representative'' $\tilde\Gamma$
rather than considering the entire family of possible assemblies
of $1$--neighborhoods of $\Gamma$ into graphs, makes the proofs
somewhat lighter.}.
Further, denote by $f_1,f_2,\dots,f_n$ the [random] center-preserving isomorphisms between the respective neighborhoods of $\Gamma$ and $\tilde\Gamma$ (which, viewed as random variables, are assumed to be measurable with respect to $\Gamma$).

We will mostly work with the graph $\Gamma$ conditioned on certain event of probability close to one
which encapsulates ``typical'' properties of  Erd\H os--R\'enyi graphs of the given density, such as upper and lower bounds
on the vertex degrees, on the number of common neighbors for a pair of vertices; as well as certain
graph expansion properties. Those conditions are stated and proved in Section~\ref{s:prelim} of the paper.

\bigskip

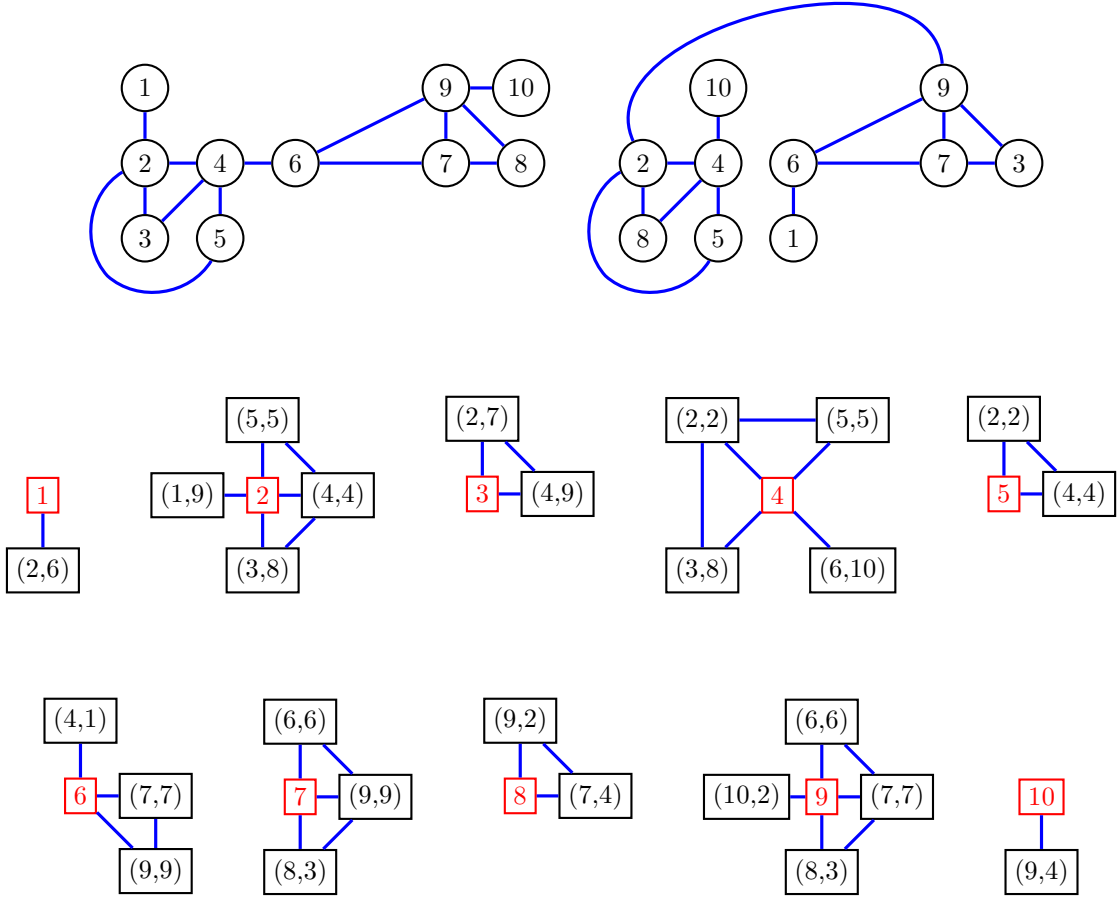
\begin{figure}[h]\label{f:f1}
\caption{An example of two non-isomorphic graphs -- $\Gamma$ (on the top left)
and $\tilde \Gamma$ (top right) whose respective $1$--neighborhoods are isomorphic.
We represent each pair of neighborhoods $N_{\Gamma}(i),N_{\tilde\Gamma}(i)$, $1\leq i\leq 10$,
by placing one on top of the other and putting {\it double labels} on vertices in the neighborhood.
The first number in each label is the vertex' index in graph $\Gamma$, and the second --- the index in $\tilde\Gamma$.
Centers of the neighborhoods are marked in red.}

\centering  

\subfigure
{
\begin{tikzpicture}[every node/.style={circle,thick,draw}]
    \node (1) at (-2,0) {1};
    \node (2) at (-2,-1) {2};
    \node (3) at (-2,-2) {3};
    \node (4) at (-1,-1) {4};
    \node (5) at (-1,-2) {5};
    \node (6) at (0,-1) {6};
    \node (7) at (2,-1) {7};
    \node (8) at (3,-1) {8};
    \node (9) at (2,0) {9};
    \node (10) at (3,0) {10};
    \draw[blue, very thick] [-] (1) -- (2);
    \draw[blue, very thick] [-] (2) -- (3);
    \draw[blue, very thick] [-] (3) -- (4);
    \draw[blue, very thick] [-] (2) -- (4);
    \draw[blue, very thick] [-] (4) -- (5);
    \draw[blue, very thick] [-] (4) -- (6);
    \draw[blue, very thick] [-] (6) -- (7);
    \draw[blue, very thick] [-] (6) -- (9);
    \draw[blue, very thick] [-] (7) -- (9);
    \draw[blue, very thick] [-] (7) -- (8);
    \draw[blue, very thick] [-] (8) -- (9);
    \draw[blue, very thick] [-] (9) -- (10);
    \draw[blue, very thick] (5) to [bend left=50] (-2.5,-2.5) to [bend left=50] (2);
\end{tikzpicture}
}
\subfigure
{  
\begin{tikzpicture}
\begin{scope}[every node/.style={circle,thick,draw}]
    \node (1) at (0,-2) {1};
    \node (2) at (-2,-1) {2};
    \node (8) at (-2,-2) {8};
    \node (4) at (-1,-1) {4};
    \node (5) at (-1,-2) {5};
    \node (6) at (0,-1) {6};
    \node (7) at (2,-1) {7};
    \node (3) at (3,-1) {3};
    \node (9) at (2,0) {9};
    \node (10) at (-1,0) {10};
    \draw[blue, very thick] [-] (1) -- (6);
    \draw[blue, very thick] [-] (2) -- (8);
    \draw[blue, very thick] [-] (8) -- (4);
    \draw[blue, very thick] [-] (2) -- (4);
    \draw[blue, very thick] [-] (4) -- (5);
    \draw[blue, very thick] [-] (6) -- (7);
    \draw[blue, very thick] [-] (6) -- (9);
    \draw[blue, very thick] [-] (7) -- (9);
    \draw[blue, very thick] [-] (7) -- (3);
    \draw[blue, very thick] [-] (3) -- (9);
    \draw[blue, very thick] [-] (4) -- (10);
    \draw[blue, very thick] (2) to [bend left=100] (9);
    \draw[blue, very thick] (5) to [bend left=50] (-2.5,-2.5) to [bend left=50] (2);
\end{scope}
\end{tikzpicture}
}

\vspace{1cm}

\subfigure
{
\begin{tikzpicture}
\begin{scope}[every node/.style={rectangle,thick,draw}]
    \node[red] (1) at (0,0) {1};
    \node (2) at (0,-1) {(2,6)};
    \draw[blue, very thick] [-] (1) -- (2);
\end{scope}
\end{tikzpicture}
}\qquad
\subfigure
{
\begin{tikzpicture}
\begin{scope}[every node/.style={rectangle,thick,draw}]
    \node[red] (2) at (0,0) {2};
    \node (1) at (-1,0) {(1,9)};
    \node (3) at (0,-1) {(3,8)};
    \node (4) at (1,0) {(4,4)};
    \node (5) at (0,1) {(5,5)};
    \draw[blue, very thick] [-] (4) -- (3);
    \draw[blue, very thick] [-] (5) -- (4);
    \draw[blue, very thick] [-] (2) -- (1);
    \draw[blue, very thick] [-] (2) -- (3);
    \draw[blue, very thick] [-] (2) -- (4);
    \draw[blue, very thick] [-] (2) -- (5);
\end{scope}
\end{tikzpicture}
}\qquad
\subfigure
{
\begin{tikzpicture}
\begin{scope}[every node/.style={rectangle,thick,draw}]
    \draw[white, very thick] [-] (0,0) -- (0,-1.3);
    \node[red] (3) at (0,0) {3};
    \node (2) at (0,1) {(2,7)};
    \node (4) at (1,0) {(4,9)};
    \draw[blue, very thick] [-] (2) -- (4);
    \draw[blue, very thick] [-] (3) -- (2);
    \draw[blue, very thick] [-] (3) -- (4);
\end{scope}
\end{tikzpicture}
}\qquad
\subfigure
{
\begin{tikzpicture}
\begin{scope}[every node/.style={rectangle,thick,draw}]
    \node[red] (4) at (0,0) {4};
    \node (2) at (-1,1) {(2,2)};
    \node (3) at (-1,-1) {(3,8)};
    \node (5) at (1,1) {(5,5)};
    \node (6) at (1,-1) {(6,10)};
    \draw[blue, very thick] [-] (4) -- (2);
    \draw[blue, very thick] [-] (4) -- (3);
    \draw[blue, very thick] [-] (4) -- (5);
    \draw[blue, very thick] [-] (4) -- (6);
    \draw[blue, very thick] [-] (2) -- (3);
    \draw[blue, very thick] [-] (2) -- (5);
\end{scope}
\end{tikzpicture}
}\qquad
\subfigure
{
\begin{tikzpicture}
\begin{scope}[every node/.style={rectangle,thick,draw}]
    \draw[white, very thick] [-] (0,0) -- (0,-1.3);
    \node[red] (5) at (0,0) {5};
    \node (2) at (0,1) {(2,2)};
    \node (4) at (1,0) {(4,4)};
    \draw[blue, very thick] [-] (2) -- (4);
    \draw[blue, very thick] [-] (5) -- (2);
    \draw[blue, very thick] [-] (5) -- (4);
\end{scope}
\end{tikzpicture}
}

\vspace{1cm}

\subfigure
{
\begin{tikzpicture}
\begin{scope}[every node/.style={rectangle,thick,draw}]
    \draw[white, very thick] [-] (0,0) -- (0,-1.3);
    \node[red] (6) at (0,0) {6};
    \node (4) at (0,1) {(4,1)};
    \node (7) at (1,0) {(7,7)};
    \node (9) at (1,-1) {(9,9)};
    \draw[blue, very thick] [-] (6) -- (4);
    \draw[blue, very thick] [-] (6) -- (7);
    \draw[blue, very thick] [-] (6) -- (9);
    \draw[blue, very thick] [-] (7) -- (9);
\end{scope}
\end{tikzpicture}
}\qquad
\subfigure
{
\begin{tikzpicture}
\begin{scope}[every node/.style={rectangle,thick,draw}]
    \node[red] (7) at (0,0) {7};
    \node (6) at (0,1) {(6,6)};
    \node (9) at (1,0) {(9,9)};
    \node (8) at (0,-1) {(8,3)};
    \draw[blue, very thick] [-] (7) -- (6);
    \draw[blue, very thick] [-] (7) -- (9);
    \draw[blue, very thick] [-] (7) -- (8);
    \draw[blue, very thick] [-] (9) -- (8);
    \draw[blue, very thick] [-] (9) -- (6);
\end{scope}
\end{tikzpicture}
}\qquad
\subfigure
{
\begin{tikzpicture}
\begin{scope}[every node/.style={rectangle,thick,draw}]
    \draw[white, very thick] [-] (0,0) -- (0,-1.3);
    \node[red] (8) at (0,0) {8};
    \node (9) at (0,1) {(9,2)};
    \node (7) at (1,0) {(7,4)};
    \draw[blue, very thick] [-] (8) -- (9);
    \draw[blue, very thick] [-] (8) -- (7);
    \draw[blue, very thick] [-] (9) -- (7);
\end{scope}
\end{tikzpicture}
}\qquad
\subfigure
{
\begin{tikzpicture}
\begin{scope}[every node/.style={rectangle,thick,draw}]
    \node[red] (9) at (0,0) {9};
    \node (6) at (0,1) {(6,6)};
    \node (7) at (1,0) {(7,7)};
    \node (8) at (0,-1) {(8,3)};
    \node (10) at (-1,0) {(10,2)};
    \draw[blue, very thick] [-] (9) -- (6);
    \draw[blue, very thick] [-] (9) -- (7);
    \draw[blue, very thick] [-] (9) -- (8);
    \draw[blue, very thick] [-] (9) -- (10);
    \draw[blue, very thick] [-] (7) -- (6);
    \draw[blue, very thick] [-] (7) -- (8);
\end{scope}
\end{tikzpicture}
}\qquad
\subfigure
{
\begin{tikzpicture}
\begin{scope}[every node/.style={rectangle,thick,draw}]
    \node[red] (10) at (0,0) {10};
    \node (9) at (0,-1) {(9,4)};
    \draw[blue, very thick] [-] (10) -- (9);
\end{scope}
\end{tikzpicture}
}

\end{figure}

\HL{
The high-level idea of the proof is to first show that with high probability the identity map $V(\Gamma)\to V(\tilde \Gamma)$ is ``almost'' a graph isomorphism between $\Gamma$ and $\tilde \Gamma$ in the sense that $f_i(j)=j$ for the \maj{} of pairs $(i,j)$ with $\{i,j\}\in E(\Gamma)$ (here, ``\maj'' refers to at least
$(1-c)$--fraction or more, for some small constant $c>0$), and then,
assuming the two graphs are ``almost isomorphic'', consider a bootstrap argument to show that $\Gamma=\tilde \Gamma$ with high probability. 

Steps I-III of the proof are devoted to that first part of the argument and Step IV deals with the bootstrap procedure.
\begin{itemize}
\item In Step I, we show that with high probability for a \maj{} of unordered pairs of vertices $\{v,w\}$, we have $\{v,w\}=\{f_i(v),f_i(w)\}$ for most choices
of $i$ in the common neighborhood of $v,w$. To illustrate our motivation for working with images of pairs of vertices as the initial step, suppose that $\{f_i(v),f_i(w)\}=\{f_{i'}(v'),f_{i'}(w')\}$ for two triples of vertices $\{i,v,w\}$
and $\{i',v',w'\}$ of $\Gamma$ with $\{v,w\}\neq \{v',w'\}$. Then necessarily
$\{v,w\}$ and $\{v',w'\}$ must be either both edges or both non-edges of $\Gamma$, which induces a constraint on the graph structure. Many vertex triples
of this type would induce many constraints, and that can be shown to occur with only a small probability.
\item In Step II, we use the result of Step I to show the existence of a permutation $\pi$ of $[n]$ such that for a \maj{} of vertices $v$, $f_i(v)=\pi(v)$ for most vertices $i$ in the neighborhood of $v$. 
\item We show that the permutation $\pi$ from Step II can be chosen to be the identity in Step III.
\end{itemize}
Now, we give a more detailed outline of structure of the four steps of the proof.
We will not specify the choice and the interrelation between the constants
to clarify the exposition.}

\bigskip

{\bf Step I.} We show that with \HL{high} probability most vertex pairs $\{v,w\}$ of $\Gamma$
are mapped into a same pair $\{\tilde v,\tilde w\}$ of vertices of $\tilde\Gamma$.
More precisely, conditioned on a typical realization of $\Gamma$,
we will say that a pair of vertices $\{v,w\}$ of $\Gamma$
is {\it focused} if there is a pair of vertices $\{\tilde v,\tilde w\}$ of $\tilde\Gamma$
such that for a {\it vast majority} of $1$--neighborhoods $N_\Gamma(i)$ of $\Gamma$ containing $\{v,w\}$,
we have $\{f_i(v),f_i(w)\}=\{\tilde v,\tilde w\}$
(we note here that for a typical realization of $\Gamma$, every pair of vertices of $\Gamma$
is contained in $(1\pm o(1))p^2n$ neighborhoods).
We then show that with probability close to one
most of the pairs of vertices of $\Gamma$ are focused.

The idea of the proof can be described as follows.
Consider a realization of $\Gamma$ such that many of pairs of vertices of $\Gamma$ are not focused.
For every non-focused pair $\{v,w\}$,
there are many pairs of indices $i_1,i_2\in[n]$ such that
$v,w$ are adjacent to both $i_1$ and $i_2$ in $\Gamma$, and
$\{f_{i_1}(v),f_{i_1}(w)\}
\neq \{f_{i_2}(v),f_{i_2}(w)\}$. A counting argument then implies that there are many $2$--tuples
of pairs of vertices $(\{v_1,w_1\},\{v_2,w_2\})$, such that
$\{v_1,w_1\}\neq \{v_2,w_2\}$, and for some $i_1,i_2\in[n]$, $\{f_{i_1}(v_1),f_{i_1}(w_1)\}
=\{f_{i_2}(v_2),f_{i_2}(w_2)\}$. But the last condition necessarily means that either both
pairs $\{v_1,w_1\},\{v_2,w_2\}$ are edges in $\Gamma$ or neither of them is.
That way, the condition that many vertex pairs are not focused introduces multiple
constraints (``dependencies'') on the edges of $\Gamma$. We want to claim that fulfilling those constraints
simultaneously is very unlikely. The obvious problem is that the collection of $2$--tuples
$(\{v_1,w_1\},\{v_2,w_2\})$ representing those constraints depends on the realization of $\Gamma$,
and we need to devise a decoupling argument to be able to use the randomness of $\Gamma$.
The assumption $p\geq n^{-1/2}\log^C n$ turns out crucial at this stage. This can be illustrated as follows.
Consider the set of typical realizations of the isomorphism $f_1$
(where by a ``realization'' we mean
both the domain of the mapping i.e the set of vertices adjacent to $1$ in $\Gamma$,
and the mapping itself). Since we are working with a typical graph,
we will assume that the number of vertices adjacent to $1$ is of order $\Theta(np)$.
Then the set of typical realizations of $f_1$ has size of order at most $\exp(O(np\log n))$.
On the other hand, on the event when a significant proportion of pairs of vertices
of $\Gamma$ are not focused, it turns out that there are on average $\Omega(n^2p^2/({\rm polylog}(n)))$ constraints
(``dependencies'') for pairs of vertices per neighborhood. For every two pairs $\{v_1,w_1\}$
and $\{v_2,w_2\}$ of distinct vertices of $\Gamma$, the probability of the event
$$
\big\{\mbox{Both $\{v_1,w_1\}$ and $\{v_2,w_2\}$ are edges of $\Gamma$}\big\}
\cup \big\{\mbox{Neither of $\{v_1,w_1\}$ and $\{v_2,w_2\}$ is an edge of $\Gamma$}\big\}
$$
is of order $1-\Theta(p)$.
Therefore, $\Omega(n^2p^2/({\rm polylog}(n)))$ distinct constraints can be satisfied with 
probability at most $\exp(-\Omega(n^2p^3/({\rm polylog}(n))))$.
It can be easily checked that for $p\geq n^{-1/2}\log^C n$ for a sufficiently large constant $C$, we have
$$
\exp(-\Omega(n^2p^3/({\rm polylog}(n))))\cdot \exp(O(np\log n))=o(1).
$$
That is, the probability that the constraints are satisfied simultaneously beats
the number of typical realizations of $f_1$. This observation forms the basis of our decoupling argument
which is rigorously carried out in Section~\ref{s:step1}.

\bigskip

{\bf Step II.} At the first step of the proof, we have shown that almost every pair of vertices
$\{v,w\}$ of $\Gamma$ is ``essentially'' mapped to a same pair of vertices of $\tilde\Gamma$.
The goal of the second step is to extend this property to mappings of vertices rather than vertex pairs.
More precisely, we want to show that with probability close to one there is a permutation $\pi$
such that for a vast majority of vertices $v$ of $\Gamma$, we have $f_i(v)=\pi(v)$ for almost all
neighborhoods $N_\Gamma(i)$ containing $v$.
The idea can be described as follows.
Assume that there is a triple
of vertices $\{v,w,z\}$ of $\Gamma$ and two indices $i_1\neq i_2$ such that $v,w,z$
are adjacent to both $i_1$ and $i_2$, and both pairs $\{v,z\}$ and $\{w,z\}$
are mapped to their focuses by both mappings $f_{i_1}$ and $f_{i_2}$.
Assume that the focus of $\{v,z\}$ is $\{\tilde a,\tilde b\}$,
and the focus of $\{w,z\}$ is $\{\tilde b,\tilde c\}$ (note that since the pairs $\{v,z\}$ and  $\{w,z\}$ 
share a common vertex in $\Gamma$, their focuses must share a common vertex in $\tilde\Gamma$).
But then, obviously, $f_{i_1}(z)=f_{i_2}(z)=\tilde b$, i.e $z$ must be mapped to a same vertex both by $f_{i_1}$
and $f_{i_2}$. An important part of the second step is to actually show that
the existence of many focused pairs implies existence of many triples with the aforementioned
properties.

\bigskip

{\bf Step III.}
We want to show that the permutation $\pi$ found at Step II is actually
close to the identity (in the Hamming metric) with a very large probability.
More precisely, we show that for certain (small) parameter $\varepsilon>0$, the event
\begin{align*}
\Event_3(\varepsilon):=
\bigg\{\sum_{i\in[n]}|\{z\snb{\Gamma}i\,
:\;f_i(z)=z
\}|\geq (1-\varepsilon)n^2p\bigg\}.
\end{align*}
has probability $1-n^{-\omega(1)}$.
Here, the idea is similar to the one used in Step I, although the details are quite different.
The crucial point is that permutations far from the identity introduce a relatively large number of constraints
on the edges of $\Gamma$, which can only be satisfied simultaneously with a very small probability.
To see how the argument actually works, consider a pair $\{v,w\}$ of vertices of $\Gamma$.
Assume that for a realization of $\Gamma$ that we have fixed, the vertices $v$ and $w$
are adjacent in $\Gamma$. Note that necessarily $\{v,\pi(w)\}$ is an edge in $\tilde\Gamma$,
whence $v$ is a vertex in the neighborhood $N_{\tilde\Gamma}(\pi(w))$.
If we also suppose that $f_{\pi(w)}^{-1}(v)=\pi^{-1}(v)$ (after all, Step II asserts that
on most neighborhoods the isomorphisms $f_i$ essentially agree with $\pi(\cdot)$, up to a small number of exceptional points),
then necessarily $\{\pi^{-1}(v),\pi(w)\}$ is an edge in $\Gamma$.
Thus, for a large number of pairs $\{v,w\}$ the condition that $\{v,w\}$ is an edge in $\Gamma$
implies that $\{\pi^{-1}(v),\pi(w)\}$ is an edge as well. When $\pi$ acts as the identity on $\{v,w\}$,
this condition does not add any additional constraints on the edges of $\Gamma$ since in that case
$\{v,w\}=\{\pi^{-1}(v),\pi(w)\}$. On the other hand, if $\pi$ is far from the identity
then for a significant number of pairs we will have $\{\pi^{-1}(v),\pi(w)\}\neq \{v,w\}$,
and so the condition ``if $\{v,w\}$ is an edge then $\{\pi^{-1}(v),\pi(w)\}$ is an edge''
becomes a constraint. A lower bound on the Hamming distance between $\pi$ and the identity
then allows us to collect a sufficient number of constraints on $\Gamma$ to be able to carry out
an argument similar to the one in Step I.

\bigskip

{\bf Step IV.}
In this last step, we show that the graphs $\Gamma$ and $\tilde \Gamma$ actually coincide with high probability,
using the information obtained in Step III.
Let  $ \cal{P} := \{ (v,w) \in [n]^{\times 2} \,:\, \{v,w\} \in E(\Gamma)\}$ and 
$ \tilde{\cal{P}} := \{ (\tilde{v},\tilde{w}) \in [n]^{\times 2} \,:\, \{\tilde{v},\tilde{w}\} \in E(\tilde{\Gamma})\}$
be the ``ordered'' edges of $\Gamma$ and $\tilde\Gamma$.
Further, let
$ \cal{M}\subset \cal{P}$ be the subset of {\it matched ordered edges} defined in the following way:
$$\cal{M}: = \{ (v,w) \in \cal{P} \,:\, f_v(w)=w\}.$$
Notice that whenever $ \cal{M} = \cal{P}$, we necessarily have $\Gamma = \tilde{\Gamma}$.
Further, using Step III, the set
${\cal{M}}^c = \cal{P}\backslash \cal{M}$ has cardinality $O(\varepsilon n^2p)=O(\varepsilon |\cal{P}|)$
with high probability. 
In this final step of the proof, we want to argue that the condition that $|\cal{M}^c| = o(|\cal P|)$ w.h.p,
essentially obtained in Step III, forces the equality $\cal{M}=\cal P$ with high probability (this
can be viewed as a bootstrapping argument). 

Let us examine a ``typical'' pair $(v,w) \in \cal P$, with $\tilde w:=f_v(w)$ (so that $(v, \tilde w) \in \tilde{\cal P}$),
and the common neighbors of $v$ and $w$ in $N_{\Gamma}(v)$
and of $v$ and $\tilde w$ in $N_{\tilde\Gamma}(v)$.
Since $f_v: N_{\Gamma}(v) \mapsto N_{\tilde \Gamma}(v)$ is a graph isomorphism,
$f_v$
induces a bijection between the common neighbors of $v$ and $w$ in $N_{\Gamma}(v)$ and of $v$ and $\tilde w$ in $N_{\tilde\Gamma}(v)$.
Now, with the assumption that $|\cal{M}^c| = o(| \cal P|)$, we expect that a vast majority
of the common neighbors of $v$ and $w$ in $\Gamma$ are also common neighbors of $v$ and $\tilde{w}$ in $\tilde \Gamma$
(that is, $f_v(\mbox{``common neigh. of $v,w$ in $\Gamma$''})\approx \mbox{``common neigh. of $v,w$ in $\Gamma$''}$). 
Without giving the precise definition at this moment, we introduce a subset $\cal{V}$ of the pairs $(v,w)$ in $\cal{P}$ having such property
(so that a vast majority of pairs from $\cal{P}$ are in $\cal{V}$).
Further, if we consider $( \tilde w ,\, f_{\tilde w}^{-1}(v)) \in \cal P$, the same heuristic argument implies that there is a significant overlap between the sets of the common neighbors of $\tilde w$ and $f_{\tilde w}^{-1}(v)$
in $\Gamma$ and of $\tilde{w}$ and $v$ in $\tilde \Gamma$.
The consequence is that now we have a nontrivial overlap 
between the sets of common neighbors of $\{v,w\}$ and of $\{\tilde{w}, f_{\tilde w}^{-1}(v)\}$ {\it in} $\Gamma$,
say the size is proportional to the expected number of common neighbors of 2 points in an Erd\H os--R\'enyi graph: $\Theta(p^2n)$. If $p^2n = \omega(1)$, then $\Theta(p^2n)$ is much larger than what is expected for the size of common neighbors of 3 or 4 distinct points in $\Gamma$ (which is of order $\Theta(np^3)$ or $\Theta(np^4)$, respectively).
Thus, the only reasonable explanation is that the unordered pairs $\{v,w\}$ and $\{\tilde{w}, f_{\tilde w}^{-1}(v)\}$
are actually the same pair in $\Gamma$, i.e $w=\tilde w$ and $v = f^{-1}_{\tilde w}(v)$.
Representing it as a relation between $\cal M$ and $\cal V$, we can write that, conditioned
on certain ``typical'' event of a high probability, there is an implication
\begin{align} \label{eq: IVVandW}
(v,w), (\tilde{w}, f_{\tilde w}^{-1}(v)) \in \cal{V} \Rightarrow (v,w),(w,v) \in \cal M .\end{align}
We want to use \eqref{eq: IVVandW} to derive a self bounding inequality on the cardinality of $\cal{M}$
which would force $\cal M = \cal P$ with high probability.
To that end, we split $[n]$ into two subsets $J_{\cal M}$ and $J_{\cal M}^c$, where
$J_{\cal M}$ is defined as a collection of those vertices $v$ such that a majority of pairs $(v,w) \in \cal P$
is contained in $\cal M$. 
From Step III, we know with high probability that $|J_{\cal M}^c| = o(n)$ (assuming right choices of parameters). 
Further, we show that for $v \in J_{\cal M}$, the proportion of the pairs $(v,w)$ in $\cal V^c$
is insignificant compared to the proportion of the pairs in $\cal M^c$. This allows us to use \eqref{eq: IVVandW}
to conclude that for a majority of pairs
$(v,w) \in \cal{M}^c$ with $v \in J_{\cal M}$, we must have $f_v(w)\in J_{\cal M}^c$. Observe that a partial consequence is that
$J_{\cal M}^c \neq \emptyset$ whenever $\cal M \neq \cal P$. 

Relying on these observations, we derive some structural information for $\Gamma$ and $\tilde \Gamma$.
The two crucial statements are (a) by removing $o( |J_{\cal M}^c|np)$ edges,
$\tilde{\Gamma}_{J_{\cal M}}$ becomes a subgraph of $\Gamma_{J_{\cal M}}$ via the identity map, and (b) the number of edges connecting $J_{\cal M}$ and $J_{\cal M}^c$ in $\tilde \Gamma$ is similar to that in $\Gamma$,
which implies that for a typical point in $J^c_{\cal M}$, majority of its neighbors in $\tilde \Gamma$ are in $J_{\cal M}$.

The conditions (a) and (b), together with the assumption that $J^c_{\cal M} \neq \emptyset$,
allow us to conclude that there exists a vertex $v \in J^c_{\cal M}$ such that a constant proportion
of its neighbors are not preserved by the map $f_v$ (i.e. mapped to different vertices of $\tilde\Gamma$),
and their image is in $J_{\cal M}$. And if any two such neighbors $w,w'$ form an edge $\{w,w'\}$ in $\Gamma$ (there will be about $\Omega( np \cdot np \cdot p)$ such edges), then typically
$$ \{f_v(w), f_v(w')\}\neq \{w,w'\} \mbox{ and }\{ f_v(w), f_v(w')\} \in E(\Gamma)\cap E(\tilde\Gamma)
.$$
But the latter is unlikely to be observed in a typical realization $\Gamma$ as
can be shown using an argument similar to that in Step I. 
Hence, it is likely that $J^c_{\cal M}=\emptyset$ and thus $\Gamma = \tilde \Gamma$ w.h.p.

\medskip

\co{At a high level, our argument resembles the strategy
in the earlier work \cite{PRS}, where shotgun assembly of a randomly colored hypercube
from $2$--neighborhoods of its vertices was considered.
One of the key ideas in \cite{PRS} is to verify the reconstructability in two steps,
first, by showing that any bijective mapping on the vertices
of the hypercube which preserves
the local structure, ``roughly maps neighbourhoods to neighbourhoods'',
and second, showing that any such mapping is with a high probability
an automorphism of the cube. In our work, we implement a relative of this
strategy via our four--step procedure. We note here that the graph
models considered in \cite{PRS} and the present paper are completely different,
and technical aspects of the two works cannot be matched.}

\bigskip

The paper is organized as follows.
In Section~\ref{s:prelim}, we revise the notation and state some typical properties of Erd\H os--R\'enyi graphs
important for our argument.
The four steps of the argument are successively carried out in Sections~\ref{s:step1},~\ref{s:step2},~\ref{s:step3},~\ref{s:step4}.
In Section~\ref{s:entropy} we derive Theorem~\ref{3u14o21u4oi}.
Finally, in Section~\ref{s:further} we discuss some further problems related to the reconstruction of random graphs.

\section{Notation and preliminaries}\label{s:prelim}

\HL{We start this section with a set definitions used throughout the paper. Although
some of the notions were already given in the introduction, we prefer
to recall them here as well, for the convenience of referencing.}

\begin{defi}
For every graph $G$, denote by $V(G)$ its vertex set and by $E(G)$ its edge set.
\end{defi}

\begin{defi}
For a graph $G$ on a vertex set $V$ and a subset $J \subset V$, let $G_J$ be the induced 
subgraph of $G$ with the vertex set $J$.
\end{defi}

\begin{defi}
Let $i\in\N$.
Given two graphs $G,G'$ with $i\in V(G)\cap V(G')$,
we say that $G$ and $G'$ are 
\co{isomorphic with fixed point $i$} if
\begin{itemize}
\item The vertex sets $V(G)$ and $V(G')$ have the same size;
\item There is a bijection $f:V(G)\to V(G')$ such that $\{f(v),f(w)\}\in E(G')$ if and only if $\{v,w\}\in E(G)$,
and such that $f(i)=i$.
\end{itemize}
\end{defi}

\begin{defi}
Let $G$ be a graph on $[n]$, and let $i\in [n]$.
Denote by $N_G(i)$ the $1$--neighborhood of $i$ in $G$ which we view as the {\bf subgraph} of $G$ induced by the set of vertices
at distance at most one from $i$.
\end{defi}

\begin{defi}
\co{Let $G$ be a graph on $[n]$, and let $v \in [n]$.
We write $w\stackrel{G}{\sim}v$ for any vertex $w\in V(N_G(v))\setminus\{v\}$,
and $w\stackrel{G}{\nsim} v$ for any vertex $w\notin V(N_G(v))\setminus\{v\}$.}
\end{defi}

\begin{defi}\label{aspfoiafpoifnqpfinqpfinapi}
For any graph $G$ on $[n]$ and each $i \in [n]$, let $S_G(i)$ be the collection of unordered 
pairs of vertices in $V(N_G(i))\backslash \{i\}$.
Further, let $\S_2$ denote the unordered pairs \co{of distinct elements} of $[n]$.
\end{defi}

In what follows, $\Gamma$ is a $G(n,p)$ random graph,
and for every realization of $\Gamma$, $\tilde\Gamma$ is a graph on $[n]$
such that $N_\Gamma(i)$ and $N_{\tilde\Gamma}(i)$ are
isomorphic with fixed point $i$
for every $1\leq i\leq n$. Our main goal is to show that, under appropriate conditions on $p$,
necessarily $\Gamma=\tilde\Gamma$ with probability $1-n^{-\omega(1)}$.
For each $i$, we let $f_i:\,V(N_\Gamma(i))\to V(N_{\tilde\Gamma}(i))$ be an
\co{isomorphism with fixed point $i$}
of the respective neighborhoods, and assume that $f_i$'s, viewed as random variables,
are measurable with respect to $\Gamma$.
Note that each $f_i$ naturally extends to a mapping from $S_\Gamma(i)$ to $S_{\tilde \Gamma}(i)$.

\begin{defi}\label{aifjnapifjnpfqijwnfpiqjnfpi}
Let $\varepsilon>0$.
An unordered pair $\{v,w\} \in \S_2$ is {\it $(1-\varepsilon)$--focused}     
if there is a subset $I \subset [n]$ with $ |I| \ge (1-\varepsilon) p^2n$
such that $\forall \;i \in I$,  $\{v,w\} \in S_\Gamma(i)$, and 
$f_i(\{v,w\}) = f_j(\{v,w\})$ for all $i,j\in I$.
We will further say that the pair $f_i(\{v,w\})$ (for $i\in I$) is a {\it focus} of $\{v,w\}$.
\end{defi}

In the next proposition, we estimate probabilities of certain events encapsulating ``typical''
properties of the Erd\H os--R\'enyi graph $\Gamma$ which we will need later. Clearly, most of those
properties regularly appear in some form in the random graph literature.
Although the proofs are quite standard, we prefer
to give them for completeness.
\begin{prop} \label{prop: GammaTypical}
Let, as before, $\Gamma$ be a $G(n,p)$ random graph and let\footnote{We shall use the constant $\Cgl$ throughout the paper.}
$$
\Cgl:=\frac{1}{100}.
$$
Define event $\Event_{typ}= \Event_{typ}(n,p)$ as the intersection of the following events.

[Events characterizing the number of (common) neighbors of $1,\, 2$, or $3$ vertices in the graph:]
\begin{align}
\label{eq: EtypDegree} &\Big\{ \big| |V(N_\Gamma(i))| - pn \big| \leq  \frac{\log n}{\sqrt{pn}} pn \mbox{ for all $i\in[n]$}\Big\},\\
& \label{eq: EtypCommonNeighbors2}\Big\{
\mbox{For every pair of vertices $\{v,w\}$,}  \\
&\hspace{1cm}|\{i\in[n]:\;\{v,w\}\subset V(N_\Gamma(i))\setminus\{i\}\}|\in[(1-\log^{-\Cgl}n)p^2n,(1+\log^{-\Cgl}n)p^2n]
\Big\},\nonumber  \\
& \label{eq: EtypCommonNeighbors3pts}\bigg\{ \forall \mbox{ distinct vertices } v,w,u \in [n],\,    
|\{i\in[n]:\;\{v,w,u\}\subset V(N_\Gamma(i))\setminus\{i\}\}|\le \frac{1}{10}p^2n \bigg\}. 
\end{align}
[Typical properties of subgraphs in a neighborhood of a vertex:]
\begin{align}
&  \label{eq: EtypSubgraphsEdgeCounts} \bigg\{\forall\; v \in [n] \mbox{ and }\forall\; J \subset V(N_\Gamma(v))\backslash \{v\},\;
            \bigg| |E(\Gamma_J)| - p{ |J| \choose 2} \bigg| \leq 8 n^{3/2}p^2\bigg\}, \\
& \label{eq: EtypEdgetoBigPart} \bigg\{ \forall\;  v \in [n],\,  J \subset V(N_{\Gamma}(v)) \backslash \{v\} \mbox{ with } |J| \le \frac{1}{3}pn, \mbox{ we have} \\
& \hspace{0.3cm}\Big| \Big\{ w\snb{\Gamma}v\,:\,
    \big| \big\{ u\snb{\Gamma}v\,:\, 
    u\notin J,\,
    \{w,u\} \in E(\Gamma) \big\}\big|
     \le 0.999(pn-|J|)p \Big\}\Big| 
                \le \frac{ \log^2 n }{ p^2n } |J|   \bigg\}, \nonumber \\
& \label{eq: EtypLocalSubgraphMap}\bigg\{  \forall\; u\in [n] \mbox{ and } \forall \mbox{ disjoint } I,J 
    \subset V(N_\Gamma(u)) \backslash\{u\} \mbox{ with }  |I|=|J| \ge  \frac{pn}{\log n} 
    \mbox{ and every bijection } g: I \mapsto J, \\
    & \hspace{0.3cm} |\{ \{v,w\} \subset I \,:\, \{v,w\}, \{g(v), g(w)\} \in E(\Gamma) \}| 
    \le  0.001\cdot \frac{|I|^2}{2}p \bigg\}. \nonumber
\end{align}
[The number of edges between a subset of $[n]$ and its complement:]
\begin{align}
& \label{eq: EtypEdgeToComplement}\bigg\{ \forall\; J \subset [n], \big| |E(\Gamma_{J,J^c})| - |J|(n-|J|)p \big| \le \frac{\log(n)}{\sqrt{pn}} |J|pn \bigg\},
\end{align}
where $\Gamma_{J,J^c}$ is the bipartite subgraph of $\Gamma$ on the vertex set $J\sqcup J^c$
(where we keep only edges connecting $J$ to $J^c$).

Then, assuming that $p^2n=\omega(\log^{2}n)$ and $p\leq 0.0001$, we have $\Prob(\Event_{typ})=1-n^{-\omega(1)}$.
\end{prop}

The proof of the proposition is accomplished as a combination of Lemmas~\ref{1-0948-109481-09481},~\ref{03247109487049870},~\ref{lem: edgeFromBigPart2},~\ref{314u1-49u21-94u2}, and~\ref{lem: ERCrossingedgesCounts} below.
First, we recall the classical Bernstein's inequality
\cite{Ber46} for \co{Binomial} random variables.
\begin{lemma} 
Suppose $Y_1,\dots, Y_m$ are i.i.d Bernoulli random variables with parameter $q\in(0,1)$. Then for every $t>0$ we have
\begin{align} \label{eq: BernStein}
    \Prob\Big\{ \Big| \sum_{i \in [m]} Y_i - mq \Big| \ge  t
    \Big\}  \le 2 \exp\Big( - \frac{t^2/2 }{ mq + t}\Big).
\end{align}
\end{lemma}

\begin{lemma}\label{1-0948-109481-09481}
   Assume that $p\leq \frac{1}{20}$ and that $p^2n=\omega(\log^{1+2\Cgl} n)$.
   Then the events \eqref{eq: EtypDegree}, \eqref{eq: EtypCommonNeighbors2}, and \eqref{eq: EtypCommonNeighbors3pts} hold with probability at least $1- n^{-\omega(1)}$.
\end{lemma}
\begin{proof}
    The arguments to prove probability estimates for these three events are essentially the same. 
    Notice that for any fixed $k$ points in $[n]$, the number of common neighbors of these $k$ points in $\Gamma$ is a Binomial random variable of parameter $n-k$ and $p^k$. We can then apply the Bernstein inequality with an appropriate choice of the parameter, together with the union bound argument to get the result. 
    
    Let $v \in [n]$. By Bernstein's inequality \eqref{eq: BernStein}, 
    \begin{align*}
    \Prob\bigg\{ \bigg| 
    |V(N_{\Gamma}(v)) \backslash \{v\}| 
 - (n-1)p \bigg|\ge t \bigg\} \le 2 \exp \bigg( - \frac{t^2/2}{ (n-1)p + t}
   \bigg).
    \end{align*}
    Taking $t = \log(n) \sqrt{pn} $ and by the union bound argument over all $v \in [n]$, we obtain that \eqref{eq: EtypDegree} has probability $1- n^{-\omega(1)}$. 
    
    For \eqref{eq: EtypCommonNeighbors2}, we have for any fixed pair $\{v,w\}$
    $$
    \Prob\Big\{\big||\{i\in[n]:\;\{v,w\}\subset V(N_\Gamma(i))\setminus\{i\}\}|-(n-2)p^2\big|\geq t\Big\}
    \leq 2\bigg( - \frac{t^2/2}{ (n-2)p^2 + t}\bigg),
    $$
    and taking $t=\frac{1}{2}p^2n\log^{-\Cgl}n$ and by the union bound argument, we get the desired conclusion.
    
    As for \eqref{eq: EtypCommonNeighbors3pts},
by Bernstein's inequality \eqref{eq: BernStein}, for any 3 distinct points $\{v,w,u\}$,
$$\Prob\bigg\{ \bigg| 
    |V(N_{\Gamma}(v)) \cap V(N_{\Gamma}(w)) \cap V(N_{\Gamma}(u)) \backslash \{v,w,u\}| 
 - (n-3)p^3 \bigg|\ge t \bigg\} \le 2 \exp \bigg( - \frac{t^2/2}{ (n-3)p^3 + t}
   \bigg).$$
   Let $ t = \frac{1}{20} p^2n$. With the assumption that $p^2n = \omega(\log n) $,
 we get the bound 
$$ 2 \exp \bigg( - \frac{t^2/2}{ (n-3)p^3 + t} \bigg) = 
 2\exp \big( - O(p^2n) \big) = 
n^{-\omega(1)}.$$
Furthermore, assuming $p \le \frac{1}{20}$, we have $(n-3)p^3 \le \frac{1}{20}p^2n$. And thus, 
$$\Prob\bigg\{  
    |V(N_{\Gamma}(v)) \cap V(N_{\Gamma}(w)) \cap V(N_{\Gamma}(u)) \backslash \{v,w,u\}| 
 \ge \frac{1}{10}p^2n \bigg\} = n^{-\omega(1)}.$$
Together with the union bound argument, we get the desired result. 
\end{proof}

Note that given a fixed subset $J$ of $[n]$,
the size of the edge set of the graph $\Gamma_J$
is concentrated around $p|J|(|J|-1)/2$. The next lemma provides a relative of that statement,
in which the set $J$ is chosen within a neighborhood of a vertex.
\begin{lemma}\label{03247109487049870}
Assume that $p^2n=\omega(1)$. Then the event \eqref{eq: EtypSubgraphsEdgeCounts}
holds with probability $1-n^{-\omega(1)}$.
\end{lemma}
\begin{proof}
For $v \in [n]$, condition on any realization of $V(N_{\Gamma}(v))$ with $|V(N_\Gamma(v))|\leq 2pn$
(we will denote the conditional probability measure by $\tilde\Prob$).
Let $J$ be a subset of \co{the set of neighbors of $v$ in $\Gamma$}. 
Then $|E(\Gamma_J)|$ is a Binomial random variable 
with parameters ${ |J| \choose 2}$ and $p$.
By Bernstein's inequality for Binomial random variables,
\begin{align*}
    \tilde \Prob \bigg\{ 
        \bigg||E(\Gamma_J)| - p{ |J| \choose 2} \bigg|
        \ge t
    \bigg\}  
    \le 2\exp \bigg( - \frac{t^2/2}{ p {|J| \choose 2} + t } \bigg),\quad t>0.
\end{align*}
By the union bound argument we get
\begin{align*}
    & \tilde\Prob \bigg\{  
        \exists J \subset \{y:\,\{v,y\}\in E(\Gamma)\}  
        \mbox{ such that }
        \bigg||E(\Gamma_J)| - \frac{p|J|(|J|-1)}{2} \bigg|
        \ge t
    \bigg\} \\
    &\hspace{2cm}\le  2^{2 pn} \cdot 
        2\exp \bigg( - \frac{t^2/2}{ 2n^2p^3 + t } \bigg)
    \le 2 \exp\bigg( \log(2)\cdot 2 np
    -\frac{ t^2/2}{ 2n^2p^3 + t } \bigg),\quad t>0.
\end{align*}
By taking $t \ge  8n^{3/2}p^{2}$, the last term can be bounded by 
$n^{-\omega(1)}$. Together with the union bound argument over all $v \in [n]$, 
and since $|V(N_\Gamma(v))|\leq 2pn$, $1\leq v\leq n$, with probability $1-n^{-\omega(1)}$,
we obtain the statement of the lemma. 
\end{proof}
\begin{rem}
Note that since the above statement is ``local'' (is about $1$--neighborhoods of $\Gamma$), it is
immediately translated to the graph $\tilde\Gamma$ with no changes.
\end{rem}

The next lemma is a crucial part of the self-bounding argument in Step IV mentioned in the introduction.
Given a fixed subset $J$ of the vertices of a labeled, not very sparse $G(m,q)$
Erd\H os--R\'enyi graph, standard concentration inequalities imply that vertices of the graph 
will be connected with about $|J^c|q$ vertices in $J^c$. When the set $J$ is allowed to depend on the graph,
this no longer holds for {\it all} vertices but remains true for a large proportion of vertices of the graph.
\begin{lemma} \label{lem: edgeFromBigPart}
    Let $m$ be a large integer.
    Let $G$ be an  Erd\H os--R\'enyi graph with parameters $m$ and $q$ such that $ mq \ge (\log m)^2$. 
    Let $\delta = ( \log( \log m))^{-1} $. Then,
    \begin{align*}
        &\Prob \bigg\{  \mbox{There is non-empty } J \subset [m] \mbox{ with } |J| \le \frac{1}{2}m \mbox{ such that } \\ &\hspace{2cm}
         \Big| \Big\{ w \in [m]\, :\,  
          \big| \big\{ u \in J^c \,:\, \{w,u\} \in E(G) \big\} \big|
          \le (1-\delta)|J^c|q\Big\}\Big| \ge \frac{ (\log m)^2 }{2 mq} |J| 
        \bigg\} = m^{-\omega(1)}.
    \end{align*}
\end{lemma}
\begin{proof}
    Fix a positive integer $ r \le \frac{m}{2}$ and a subset $J \subset [m]$ with $|J|=r$. 
    We partition $J^c$ into $l$ subsets $J_1,\dots, J_l$ with $l \le \frac{8}{\delta}$ and
    $ |J_i| \le \frac{\delta}{4}m$ for every $ i \in [l]$. 

    For each $w \in J_1$, let $D_w:= |\{ u \in J^{c}\backslash J_1 \,:\, \{w,u\} \in E(G)\}|$
    and $Z_w$ be the indicator of the event $\{D_w \le |J^{c} \backslash J_1 |q-\delta mq/4\}$. 
    By Bernstein's inequality \eqref{eq: BernStein}, 
    \begin{align*}
         \Prob \{ Z_w = 1 \} \le 
        2\exp \bigg( - \frac{ (\delta mq/4)^2/2}{ mq+ \delta mq/4 } \bigg)
    \le \exp \big( - \Omega(\delta^2 mq) \big)
    \end{align*}
    Hence, $\sum_{w\in J_1} Z_w$ is a Binomial($|J_1|,q'$) random variable with parameter
    $q' \le \exp \big( - \Omega(\delta^2 mq) \big)$. Note that
    \begin{align}
        |J_1| q' \le  m q' \le \exp \big( \log(m) - \Omega(\delta^2 mq)  \big)
        \le  \exp( - \Omega(\delta^2 mq) ),
    \end{align}
    where the last inequality follows from our assumption on $mq $ and $\delta$. 
    Thus, for any positive integer \co{$k \le |J_1|$}, 
    \begin{align} \label{eq: lem21300}
        \Prob\Big\{ \sum_{w\in J_1} Z_w \ge k \Big\} = \sum_{s=k}^{|J_1|} {|J_1| \choose s} (q')^s (1-q')^{\co{|J_1|}-s}
        \le \sum_{s=k}^{|J_1|} (|J_1|q')^s \le 2(|J_1|q')^k \le  \exp( - \Omega(\delta^2mqk)).
    \end{align}
    \co{
    Further, the right hand side estimate in \eqref{eq: lem21300} trivially holds for $k>|J_1|$.} We want to choose a suitable $k$ so that the probability is small enough to beat the union bound 
    \co{over all} possible choices of $J$ and $r$.
    By taking $k = \lceil \delta^3 \frac{(\log m)^2}{mq}r \rceil $, we get in view of the above
    \begin{align*}
    \Prob \bigg\{
            \big| \big\{ w \in J_1 \,:\,  
              \co{ D_w } \le (|J^c \backslash J_1| - \delta m/4)q \big\}\big| 
            \ge \delta^3 \frac{ (\log m)^2 }{mq} r
        \bigg\}
    &\le \exp \big( -  \Omega \big( \delta^5 (\log m)^2 r \big) \big)\\
    &= \exp( -\omega(\log m)r).
    \end{align*}
    Next, due to the assumption $|J_1|\le \frac{\delta}{4}m$ and $r \le \frac{m}{2}$,   
    $$ 
        (|J^c\backslash J_1| -\delta m/4)q \ge  \Big( m-r-\frac{1}{2}\delta m \Big) q  
        \ge (m-r -\delta(m-r)) q
        = (1-\delta)(m-r)q=(1-\delta)|J^c|q.
    $$ 
    \co{Together with the relation $D_w \le  |\{ u \in J^c \,:\, \{w,u\} \in E(G)\}| $, this gives}
    \begin{align} \label{eq: EdgeFromGoodSet1}
        \Prob \bigg\{
            \big| \big\{ w \in J_1 \,:\,  |\{ u \in J^c \,:\, \{w,u\} \in E(G) \}| 
            \le (1-\delta)|J^c|q \big\}\big| 
            \ge \delta^3 \frac{ (\log m)^2 }{mq} r
        \bigg\} 
    = \exp( -\omega(\log m)r).
    \end{align}
    With the same argument, \eqref{eq: EdgeFromGoodSet1} holds when we replace $J_1$ in the above inequality by $J_i$ for every $i \in [l]$.
    It remains to treat vertices in $J$. 
    The argument is essentially the same. For each $w \in J$,  
    we set $D_w:= |\{ u \in J^c \,: \{w,u\} \in E(G)\}|$ and let $Z_w$ be the indicator of the event $\{D_w \le (1-\delta)|J^c|q\}$. Since the expected value of $D_w$ is $|J^c|q$, by Bernstein's inequality \eqref{eq: BernStein} we have 
    $\Prob \{ Z_w =1\} \le  \exp( -\Omega(\delta^2 mq))$. Repeating the same argument for $\sum_{ w\in J} Z_w$ as that for $\sum_{ w\in J_1}Z_w$ we get 
    \eqref{eq: EdgeFromGoodSet1} with $J$ in the place of $J_1$. 

    By taking the union bound, 
    we get 
    \begin{align*} 
        \Prob \bigg\{
            &\big| \big\{ w \in [m] \,:\,  |\{ u \in J^c \,:\, \{w,u\} \in E(G) \}| 
            \le (1-\delta)|J^c|q \big\}\big| 
            \ge \frac{ (\log m)^2 }{ 2mq} r
        \bigg\} \\
    \le&  \Prob \bigg\{ \exists J' \in \{ J_1,\dots, J_l,\, J\} \mbox{ s.t. } 
            \big| \big\{ w \in J' \,:\,  |\{ u \in J^c \,:\, \{ w,u \} \in E(G) \}| 
            \le (1-\delta)|J^c|q \big\}\big| 
            \ge \frac{\delta}{l+1} \frac{ (\log m)^2 }{2mq} r
        \bigg\} \\
    \le& (l+1)\exp( -\omega(\log m)r) = \exp(-\omega(\log m)r).
    \end{align*}
    Notice that the number of $J \subset [m]$ with $|J|=r$ is 
    $ { m \choose r} \le \exp( \log(m)r)$. 
    Applying the union bound over all $J\subset [m]$ with $|J|=r$ and all positive integer 
    $r \le \frac{m}{2}$, we get the desired bound. 
\end{proof}

\begin{lemma} \label{lem: edgeFromBigPart2}
    Assuming $ p^2n \ge 2\log^2 n$ and $p\leq \frac{1}{20}$,  the event \eqref{eq: EtypEdgetoBigPart} happens with probability $ 1- n^{-\omega(1)}$.
\end{lemma} 
\begin{proof}
    Fix $v \in [n]$, consider the event $O_v$ that $ \big| | 
    V(N_{\Gamma}(v))\backslash \{v\} 
    | - pn \big| \le 
    \frac{\log n}{\sqrt{pn}}pn$. 
    Applying Bernstein's inequality \eqref{eq: BernStein}, we get $\Prob(O_v^c) = n^{-\omega(1)}$. 
    Fix any subset $V \subset [n] \backslash \{v\}$ with $ \big| |V| - pn \big|  \le \frac{\log n}{\sqrt{pn}} pn$.
    Note that, conditioned on the event $\{V(N_{\Gamma}(v))\backslash \{v\}= V\}$, 
    the induced graph $\Gamma_V$ is an Erd\H os--R\'enyi graph with parameters $|V|$ and $p$.
    In what follows, by $\tilde\Prob$ we denote the conditional probability given the event
    $\{V(N_{\Gamma}(v))\backslash \{v\}= V\}$.
    
    From the assumption of $n$ and $p$, we have $|V| = (1+o_n(1))pn$, and  
    $ \frac{1}{2} \log n \le \log(|V|) \le \log n$,
     $|V|p \ge (\log(|V|))^2$, and 
     $ \frac{ ( \log|V|)^2}{\log(\log(|V|)) |V|p} \le \frac{ \log^2 n}{p^2n}$.
    Applying Lemma \ref{lem: edgeFromBigPart}, we obtain
     \begin{align*}
        \tilde\Prob \bigg\{ 
            &\mbox{There is non-empty }  J \subset V \mbox{ with } |J| \le \frac{1}{2}\Big(1- \frac{ \log n}{\sqrt{pn}}\Big)pn \mbox{ such that } \\
            & \hspace{0.1cm} \Big| \Big\{ w \in V \, :\,  
                |\{ u \in V \cap J^c\,:\, \{w,u\}\in E(\Gamma) \}| \le 
                0.999(pn-|J|)p \Big\}\Big| 
               \ge \frac{ \log^2 n }{ p^2n } |J|   
        \bigg\} = n^{-\omega(1)}.
    \end{align*}
    Together with the estimate
    $\Prob(O^c_v)=n^{-\omega(1)}$, this implies that the event 
    \begin{align*}
        \bigg\{ 
            &\mbox{There is non-empty } J \subset V(N_{\Gamma}(v)) \backslash \{v\} 
            \mbox{ with } |J| \le \frac{1}{3}pn \mbox{ such that } \\
            &\Big| \Big\{ w\snb{\Gamma}v\,:\,
                \co{\big| \big\{ u\snb{\Gamma}v\,:\,
                u\notin J,
                \, \{w,u\}\in E(\Gamma) \big\} \big|}
                 \le 
                    0.999(pn - |J|)p \Big\}\Big| 
                \ge \frac{ \log^2 n }{ p^2n } |J| 
        \bigg\} 
    \end{align*}
    happens with probability $n^{-\omega(1)}$. When the set $J$ is empty, the desired probability bound has
    been verified in Lemma~\ref{1-0948-109481-09481}.
    
    Finally, we can apply the union bound argument over all $v \in [n]$ to get the result. 
\end{proof}

\begin{lemma}
Let $G$ be a labeled $G(m,q)$ Erd\H os--R\'enyi graph with $m$ and $q$ satisfying $mq = \omega((\log m)^2)$
and $q\leq 0.0001$. Then
\begin{align*}
    \Prob \bigg\{  &\exists \mbox{ disjoint sets } I,J \subset [m] \mbox{ with }  |I|=|J| \ge  \frac{m}{10\log m} 
    \mbox{ and a bijection } g: I \mapsto J \mbox{ such that } \\
    &  |\{ v,w \in I \,:\, \{v,w\}, \{g(v), g(w)\} \in E(G) \}| \ge 0.001 {|I| \choose 2}q
        \bigg\} = m^{-\omega(1)}.
\end{align*}
\end{lemma}
\begin{proof}
First, we fix a positive integer $r \ge  \frac{ m}{10 \log m}$ and fix $I$, $J$, and $g$ satisfying the description stated in the lemma with $|I|=|J|=r$. 
Since $I$ and $J$ are disjoint,
$$
|\{ \{v,w\} \,:\, v,w \in I, \{v,w\}, \{g(v), g(w)\} \in E(G) \}|
$$
is a Binomial random variable with parameters ${ |I| \choose 2}$ and $q^2$, whose expectation is ${ |I| \choose 2}q^2$, which is 
much less than $ 0.001{ |I| \choose 2}q$. 
By Bernstein's inequality for Binomial random variables \eqref{eq: BernStein}, we have 
$$
    \Prob \bigg\{  |\{ \{v,w\} \,:\, v,w \in I, \{v,w\}, \{g(v), g(w)\} 
    \in E(G) \}| \ge 0.001{ |I| \choose 2}q \bigg\}
    \le \exp( - \Omega(r^2q)).
$$
Next, the number of choices for the triple $(I, J, g)$ is bounded above by $\exp( 2\log(m)r)$. 
With the assumption that $ rq \ge  \frac{ mq }{10\log(m)} = \omega(\log m)$, the union bound argument implies the statement of the lemma.
\end{proof}

As a consequence, we have 
\begin{lemma}\label{314u1-49u21-94u2}
    Assuming $ p^2n  = \omega(( \log n)^2)$ and $p\leq 0.0001$,
    the event \eqref{eq: EtypLocalSubgraphMap} 
    happens with probability $1-n^{-\omega(1)}$.
\end{lemma}
We skip the proof since it is similar to the proof of Lemma \ref{lem: edgeFromBigPart2} via Lemma \ref{lem: edgeFromBigPart}.

\begin{lemma}\label{lem: ERCrossingedgesCounts}
Assume that $pn=\omega(\log^2 n)$. Then the event \eqref{eq: EtypEdgeToComplement}
happens with probability at least $1 - n^{-\omega(1)}$. 
\end{lemma}
\begin{proof}
    Fix an integer $1 \le r \le \lceil n/2\rceil$. Note that
    the number of subsets $J\subset [n]$ with $|J| = r$ is ${ n \choose r } \le \exp( \log(n)r)$.
    Fix for a moment any $J\subset [n]$ with $|J| = r$.
    Let $\Gamma_{J,J^c}$ denote the subgraph of $\Gamma$ containing only edges connecting
    $J$ and $J^c$. 
    Then, $|E(\Gamma_{J,J^c})|$ is a Binomial random variable with parameters $(n-r)r$ and $p$. By Bernstein's inequality for Binomial random variables \eqref{eq: BernStein}, 
    we have 
    \begin{align*}
        \Prob \Big\{
            \big| |E(G_{J,J^c})| - r(n-r)p \big| \ge t 
        \Big\}
    \le
        2\exp \bigg( - \frac{ t^2/2 }{r(n-r)p + t } \bigg).
    \end{align*}
    Choosing $ t = \sqrt{pn}\log(n)r$ and under the assumption that $  \frac{ \log^2 n }{pn} = o(1)$, we get
    $$
    \Prob \Big\{
            \big| |E(G_{J,J^c})| - r(n-r)p \big| \ge t 
        \Big\}\leq 2\exp \bigg( - \frac{ pnr^2\log^2 n }{2r(n-r)p + 2\sqrt{pn}\log(n)r} \bigg)
        =2\exp(- r\Omega(\log^2 n)),
    $$
    whence
    \begin{align*}
        \Prob \Big\{ 
            \exists J \subset [n] \mbox{ with }
            \big| |E(\Gamma_{J,J^c})| - |J|(n-|J|)p \big| \ge \sqrt{pn}\log(n)|J|  
        \Big\}
    \le 2
        \sum_{r=1}^{\lceil n/2\rceil} \exp\Big( r\log(n) - r\Omega(\log^2 n) \Big)
    = n^{-\omega(1)}.
    \end{align*}
\end{proof}

\section{Step I}\label{s:step1}

\co{
For each $\{u,v\} \in \S_2$, the number of common neighbors of $u$ and $v$ in $\Gamma$ is typically of size $(1\pm o_n(1))p^2n$.
We further recall that the notion of $(1-\varepsilon)$--focused pairs
was introduced in Definition~\ref{aifjnapifjnpfqijwnfpiqjnfpi}.}
For any $\varepsilon >0$, let $\Event_1(\varepsilon)$ be defined as
$$ \Event_1(\varepsilon):=
\bigg\{\Big|\Big\{\{u,v\}\in \S_2:\;\mbox{the pair $\{u,v\}$ is $(1-\varepsilon)$--focused}\Big\}\Big|\geq (1-\varepsilon) {n \choose 2}
\bigg\}. $$
The goal of this section is to prove
\begin{prop}\label{19741094720987098}
Let $\varepsilon=\omega(\log^{-\Cgl}n)$, with $\varepsilon\leq 1/2$, and assume $p^2n=\omega(\log^{3+4\Cgl} n)$
and $p\leq 0.0001$.
Then \co{the} probability of $\Event_1(\varepsilon)$ is at least $ 1- n^{-\omega(1)}$.
\end{prop}

\bigskip

Everywhere in this section, we will implicitly assume the above conditions on $n$ and $p$.
We will show that the event $\Event_1^c(\varepsilon)$ is small under the appropriate assumptions on the parameter $\varepsilon$. 
Define a random set $$A:=\{ (i,\{v,w\})\,:\, i\in [n], \{v,w\} \in S_\Gamma(i)\}$$
(note that conditioned on $\Event_{typ}$ the set has size $(1\pm o_n(1))
{n \choose 2}
p^2n$). Further, we define a subset
$A_1\subset A$  \HL{ as a collection of elements $(i,\{v,w\})$ of $A$ such that there are {\it many}
neighborhoods $N_{\Gamma}(i')$ and pairs $\{v',w'\}$ in those neighborhoods, distinct from $\{v,w\}$
but mapped to a same pair of vertices of $\tilde\Gamma$.} Specifically, 
\begin{align*}
    A_1:= \Big\{ (i,\{v,w\})& \in A \,:\;
        \big|\big\{ (i', \{v',w'\}) \in A\,:\, \{v',w'\}\neq \{v,w\} \mbox{ and }
        f_{i'}(\{v',w'\}) = f_i(\{v,w\}) \big\}\big| \ge \frac{1}{2} \varepsilon^2p^2n 
    \Big\}.
\end{align*}
In the next lemma we show that, conditioned on a ``bad'' event $\Event_{typ} \cap \Event_{1}^c(\varepsilon)$,
the set $A_1$ must have a large cardinality.
Roughly speaking, the lemma asserts that if there are many not focused pairs in $\Gamma$ (those pairs
which are mapped to different pairs of $\tilde \Gamma$, depending on a neighborhood) then necessarily
the are many pairs in $\Gamma$ mapped to a same pair in $\tilde\Gamma$.
\begin{lemma} \label{lem: IO1}
Let $\varepsilon=\omega(\log^{-\Cgl}n)$, and 
condition on any realization of $\Gamma$ from $\Event_{typ} \cap \Event_{1}^c(\varepsilon)$. Then
    $$|A_1|\ge (1 - o_n(1))\frac{1}{2}\varepsilon^2p^2n
    {n \choose 2}
    .$$
\end{lemma}
\begin{proof}
\HL{We start by defining a map $h: \S_2 \mapsto \S_2$ which, 
for each pair of vertices $x,y$ in $\tilde\Gamma$, assigns a ``most frequent''
preimage of that pair within the collection of neighborhoods $\big\{N_{\Gamma}(i):\;S_{\tilde\Gamma}(i)\ni x,y\big\}$.
Formally, for every $\{x,y\}\in \S_2$, $h(\{x,y\})$ is chosen so that }
\begin{align*}
    &\max_{ \{v'',w''\} \in \S_2 } \big| \big\{ i\in [n]\,:\,  
     S_\Gamma(i) \ni \{v'',w''\} \mbox{ and } f_i(\{v'',w''\})=\{x,y\} \big\}\big|\\
    &\hspace{1cm}= \big| \big\{ i\in [n]\,:\, 
      S_\Gamma(i) \ni h(\{x,y\})\mbox{ and }
      f_i(h(\{x,y\}))=\{x,y\} \big\}\big|.
\end{align*} 
The map $h$ does not have to be uniquely defined; we fix any choice of $h$ satisfying the above condition.
For each $\{x,y\} \in \S_2$, we define the set 
$$A_2(\{x,y\}):= 
    \{ (i,\{v,w\}) \in A \, :\,  f_i(\{v,w\})=\{x,y\} 
        \mbox{ and } h(\{x,y\}) \neq \{v,w\}
    \}. $$
Thus, $A_2(\{x,y\})$ records indices of the neighborhoods and the pairs of vertices in those neighborhoods
which are mapped to $\{x,y\}$ but at the same time are not the ``most frequent preimage'' of $\{x,y\}$.
Note that for every $\{x,y\}$ and $\{v,w\}$ in $\S_2$, 
\co{ 
\begin{align*}
& \big|\big\{ (i, \{v',w'\}) \in A\, :\, \{v',w'\} \neq  \{v,w\}, f_{i}(\{v',w'\}) = \{x,y\} \big\}\big| \\
=& \big| \big \{ (i,\{v',w'\}) \in A\, :\, f_{i}(\{v',w'\}) = \{x, y \} \big \} \big| - 
\big| \big \{ (i, \{v,w\}) \in A\, :\, f_{i}(\{v,w\}) = \{x, y\} \big \} \big| \\
\ge & 
\big| \big \{ (i,\{v',w'\}) \in A\, :\, f_{i}(\{v',w'\}) = \{x, y \} \big \} \big| - 
\big| \big \{ (i, h(\{x,y\})) \in A\, :\, f_{i}(h(\{x,y\})) = \{x, y\} \big \} \big| \\
=& |A_2(\{x,y\})|
\end{align*}
}
(where in case $ h(\{x,y\})=\{v,w\}$ the equality holds). Consequently, 
$$
    \mbox{whenever }\{x,y\}\mbox{ satisfies }|A_2(\{x,y\})| \ge \frac{1}{2} \varepsilon^2p^2n, \mbox{ we have }  
    \{ (i,\{v,w\}) \in A\,:\, f_i(\{v,w\})=\{x,y\} \} \subset A_1. 
$$
Using the trivial bound 
$|\{ (i,\{v,w\}) \in A\,:\, f_i(v,w)=\{x,y\} \}| \ge |A_2(\{x,y\})|$, 
we can estimate $|A_1|$ in the following way: 
$$    |A_1| \ge \sum_{ \{x,y\} \,:\, |A_2(\{x,y\})| >
    \frac{\varepsilon^2}{2}p^2n } |A_2(\{x,y\})|. $$
It remains to bound the sum from below. For convenience, we define
\begin{align*}
    A_2:=&\{ (i,\{v,w\}) \in A\,:\,  \{v,w\} \neq h(f_i(\co{\{v,w\}}))\},
\end{align*}
which is the disjoint union of $\{A_2(\{x,y\})\}_{\{x,y\}\in \S_2}$.

Let $S_{\rm{NF}}$ be the collection of not $(1-\varepsilon)$--focused 
pairs of $\Gamma$. Within the event $\Event_{1}^c$, we clearly have $|S_{\rm{NF}}| \ge \varepsilon {n \choose 2}
$ . 
We claim that since we have also conditioned on the event $\Event_{typ}$,
for each $\{v,w\}\in S_{\rm{NF}}$, at least one of the following must hold:
\begin{enumerate}
    \item Either
    $|\{ i \in [n] \,:\, S_{\Gamma}(i)\ni \{v,w\}, \{v,w\} \neq h(f_i(\{v,w\})) \}| \ge (\varepsilon - \log^{-\Cgl}(n)) p^2n$,
    \item Or there exist two distinct pairs $\{x,y\}, \{x',y'\} \in \S_2$ such that $h(\{x,y\})=h(\{x',y'\})=\{v,w\}$.
\end{enumerate}
\co{Indeed, to see that, suppose $\{v,w\} \in S_{\rm{NF}}$ does not satisfy
the first condition. By \eqref{eq: EtypCommonNeighbors2}, from the definition of $\Event_{typ}$
we have $| \{ i \in [n]\,:\, S_{\Gamma}(i) \ni \{v,w\} \}| \ge (1- \log^{-\Cgl} n) p^2n$, which implies 
$$|\{ i \in [n] \,:\, S_{\Gamma}(i)\ni \{v,w\}, \{v,w\} = h(f_i(\{v,w\})) \}| \ge (1- \varepsilon) p^2n. 
$$
In particular, there exists $\{x,y\}\in \S_2$ such that $h(\{x,y\})=\{v,w\}$. Now, if $\{v,w\}$ fails to satisfy the second condition, then 
$$
| \{ i \in [n]\, :\, S_{\Gamma}(i) \ni \{v,w\}, f_i(\{v,w\}) =\{x,y\} \}| =
| \{ i \in [n]\,:\, S_{\Gamma}(i) \ni \{v,w\}
, \{v,w\} = h(f_i(\{v,w\})) \}| \ge (1- \varepsilon) p^2n,
$$ which contradicts the inclusion $\{v,w\} \in S_{\rm{NF}}$, and the claim follows.
}

Let $S_{\rm{NF},1} \subset S_{\rm{NF}}$ be the subset in which the first condition holds.  \co{ For $\{v,w\} \in S_{\rm NF,1}$, we have 
$$ |\{ (i,\{v,w\})\in A_2 \}| \ge ( \varepsilon - \log^{-\Cgl}(n) ) p^2n .$$
Further, for $ \{v,w\} \in \S_2\backslash {\rm Im}(h) $, 
$$ |\{ (i,\{v,w\})\in A_2 \}| \ge 
( 1 - \log^{-\Cgl}(n) ) p^2n, $$ 
by \eqref{eq: EtypCommonNeighbors2}. Therefore,
the following bound holds for every $\{v,w\} \in \S_2$:
$$ |\{ (i,\{v,w\})\in A_2 \}| \ge 
{\bf 1}_{ S_{\rm NF,1}}(\{v,w\}) ( \varepsilon - \log^{-\Cgl}(n) ) p^2n 
+ {\bf 1}_{ \S_2 \backslash {\rm Im}(h)}(\{v,w\}) ( 1 - \varepsilon) p^2n ,
$$
and hence, 
\begin{align*}
    |A_2| \ge &
    |S_{\rm{NF},1}| \Big(\Big(\varepsilon - \log^{-\Cgl}(n)  \Big) p^2n\Big) + 
    \bigg(
    {n \choose 2}
    - |{\rm Im}(h)|\bigg) \Big(1- \varepsilon  \Big) p^2n.
\end{align*}
}
Further, we claim that ${n \choose 2}
- |{\rm Im}(h)| \ge |S_{\rm{NF}}|- |S_{\rm{NF,1}}|$ since
for every $\{v,w\} \in S_{\rm{NF}}\backslash S_{\rm{NF,1}}$ we have $|h^{-1}(\{v,w\})| \ge 2$.

Hence, 
\begin{align*}
    |A_2| \ge& |S_{\rm{NF},1}| \Big(\Big(\varepsilon - \log^{-\Cgl}(n)  \Big) p^2n\Big) + 
    (|S_{\rm{NF}}|- |S_{\rm{NF,1}}|)
    \Big(1- \co{ \varepsilon} \Big) p^2n\\
    \ge&
    |S_{\rm{NF}}|\Big(\Big(\varepsilon - \log^{-\Cgl}(n)  \Big) p^2n\Big)
    \ge 
    (1-o_n(1))\varepsilon^2p^2n {n \choose 2}
    .
\end{align*}
Since $\{A_2(\{x,y\})\}_{\{x,y\}\in \S_2}$ is a partition of 
$A_2$,
\begin{align*} 
    |A_1| \ge& \sum_{ \{x,y\} \,:\, |A_2(\{x,y\})| >
    \frac{\varepsilon^2}{2}p^2n } |A_2(\{x,y\})| 
    = |A_2| - \sum_{ \{x,y\} \,:\, |A_2(\{x,y\})| \le
    \frac{\varepsilon^2}{2}p^2n } |A_2(\{x,y\})| \\
    \ge& (1-o_n(1))\varepsilon^2p^2n{n \choose 2}
    -{n \choose 2} 
    \frac{\varepsilon^2}{2}p^2n
    \ge  \Big( \frac{1}{2} - o_n(1)\Big) \varepsilon^2p^2n{n \choose 2}
    .
\end{align*}
\end{proof}

Define a random set
$$S_{A1}:= \bigg\{\{v,w\} \in \S_2:  |\{i\in[n]:\; (i,\{v,w\})\in A_1\}| \ge \frac{\varepsilon^2}{4}p^2n  \bigg\}.$$
The set $S_{A1}$ can be viewed as a collection of pairs of vertices $\{v,w\}$ of $\Gamma$
such that for {\it many} neighborhoods $N_{\Gamma}(i)$ in which this pair is present,
the pair of vertices $f_i(\{v,w\})$ of $\tilde\Gamma$ has a {\it significant number} of preimages distinct from $\{v,w\}$.
As a corollary of the last lemma, we obtain
\begin{cor}\label{-93871-98471098798}
Let $\varepsilon=\omega(\log^{-\Cgl}n)$.
Then, conditioned on any realization of $\Gamma$ from $\Event_{typ} \cap \Event_1^c$,
\begin{align} \label{eq: ISgood}
 \frac{| S_{A1}|}{{n \choose 2}
 } \ge (1-o_n(1))\frac{\varepsilon^2}{4}.
\end{align}
\end{cor}
\begin{proof}
By Lemma~\ref{lem: IO1}
and in view of the condition
$$\HL{\forall\; \{v,w\} \in \S_2,\, |\{ i \in [n]\, :\, S_{\Gamma}(i) \ni \{v,w\}\}| \le (1+o_n(1))p^2n,}$$
we have 
\begin{align*}
(1 - o_n(1))\frac{1}{2}\varepsilon^2p^2n{n \choose 2}
\le 
    |A_1| \le  |S_{A1}| p^2n(1+o_n(1)) + \bigg({n \choose 2}
    - |S_{A1}|\bigg) \frac{\varepsilon^2}{4}p^2n,
\end{align*}
which implies \eqref{eq: ISgood}.
\end{proof}


\HL{Whenever the set $S_{A1}$ is large, one can prove by the probabilistic method that there are many pairs $\{v,w\}$, $\{v',w'\}$ with $\{v,w\}\neq \{v',w'\}$ such that $f_i(\{v,w\})= f_{i'}(\{v',w'\})$ for some $i$ and $i'$. 
In fact, for every pair $\{v,w\}\in S_{A1}$, by sampling $\omega(\varepsilon^{-2}p^{-2}\log n)$
indices from $[n]$ uniformly and independently from each other, with a high probability
we will hit a pair of neighborhoods $N_{\tilde\Gamma}(i)$ and $N_{\tilde \Gamma}(i')$ both containing a pair 
 $\{\tilde v,\tilde w\}$
 with $f_{i}^{-1}(\{\tilde v,\tilde w\})=\{v,w\}$ and $f_{i'}^{-1}(\{\tilde v,\tilde w\})\neq\{v,w\}$.
 That is, by sampling $\omega(\varepsilon^{-2}p^{-2}\log n)$ indices
 we are very likely to find a constraint on the edges of the graph $\Gamma$, involving the pair $\{v,w\}$. Returning to the proof overview from the introduction, large $S_{A1}$ means there are many constraints on the edges of the graph $\Gamma$.}
We refer to Figure~\ref{f:fig2} for a visualization of the idea.
The next lemma makes the remark about sampling the indices rigorous.

\begin{figure}[h]
\caption{An illustration of Step I of the proof. The left column represents pairs of vertices of the graph $\Gamma$,
and the right column --- pairs of vertices of $\tilde\Gamma$.
Every line connecting a pair $\{v,w\}$ on the left with a pair $\{\tilde v,\tilde w\}$ on the right represents the action of an isomorphism $f_i$,
for some neighborhood $N_\Gamma(i)$ containing $\{v,w\}$.
Since we are always condition on a ``typical'' realization of $\Gamma$, the number of lines emanating
from each of the pairs of $\Gamma$ is $(1\pm o(1))p^2n$. The main observation is that existence
of many non-focused pairs of $\Gamma$ (such as pair $\{1,2\}$ or pair $\{2,4\}$ in this illustration)
necessarily implies existence of many right ``meeting points'' where many lines emanating from distinct
left pairs meet. In this illustration, the right pair $\{2,3\}$ represents one of the ``meeting points''.
Such a configuration of lines necessarily corresponds to a constraint on the edges of $\Gamma$.
In this example, we have the constraint on the pairs $\{1,4\}$, $\{2,3\}$, $\{2,4\}$, namely,
either all three pairs must be edges of $\Gamma$ or none of them is.
The main principle is
$$
\hspace{-1cm}\mbox{Many non-focused pairs $=$ many constraints},
$$
whereas many constraints can be simultaneously satisfied only with a very small probability.
Thus, $\Prob\{\Event_1^c(\varepsilon)\}$ must be small.}

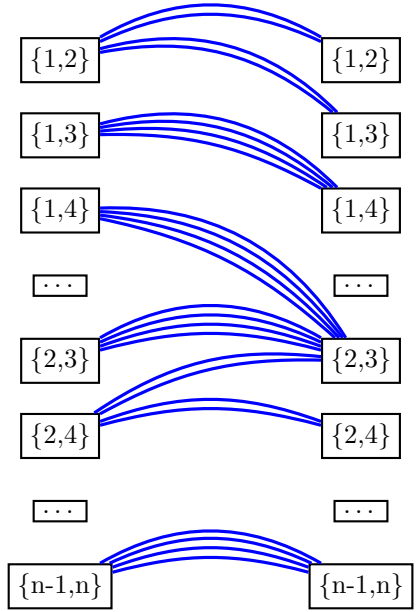
\centering  

\subfigure
{
\begin{tikzpicture}[every node/.style={rectangle,thick,draw}]
    \node (1) at (-2,4) {\{1,2\}};
    \node (2) at (-2,3) {\{1,3\}};
    \node (3) at (-2,2) {\{1,4\}};
    \node (4) at (-2,1) {\dots};
    \node (5) at (-2,0) {\{2,3\}};
    \node (6) at (-2,-1) {\{2,4\}};
    \node (7) at (-2,-2) {\dots};
    \node (8) at (-2,-3) {\{n-1,n\}};
    \node (9) at (2,4) {\{1,2\}};
    \node (10) at (2,3) {\{1,3\}};
    \node (11) at (2,2) {\{1,4\}};
    \node (12) at (2,1) {\dots};
    \node (13) at (2,0) {\{2,3\}};
    \node (14) at (2,-1) {\{2,4\}};
    \node (15) at (2,-2) {\dots};
    \node (16) at (2,-3) {\{n-1,n\}};
    \draw[blue, very thick] (1) to [bend left=30] (9);
    \draw[blue, very thick] (1) to [bend left=25] (9);
    \draw[blue, very thick] (1) to [bend left=30] (10);
    \draw[blue, very thick] (1) to [bend left=25] (10);
    \draw[blue, very thick] (2) to [bend left=30] (11);
    \draw[blue, very thick] (2) to [bend left=25] (11);
    \draw[blue, very thick] (2) to [bend left=20] (11);
    \draw[blue, very thick] (2) to [bend left=15] (11);
    \draw[blue, very thick] (3) to [bend left=30] (13);
    \draw[blue, very thick] (3) to [bend left=25] (13);
    \draw[blue, very thick] (3) to [bend left=20] (13);
    \draw[blue, very thick] (3) to [bend left=15] (13);
    \draw[blue, very thick] (5) to [bend left=30] (13);
    \draw[blue, very thick] (5) to [bend left=25] (13);
    \draw[blue, very thick] (5) to [bend left=20] (13);
    \draw[blue, very thick] (5) to [bend left=15] (13);
    \draw[blue, very thick] (6) to [bend left=20] (13);
    \draw[blue, very thick] (6) to [bend left=15] (13);
    \draw[blue, very thick] (6) to [bend left=20] (14);
    \draw[blue, very thick] (6) to [bend left=15] (14);
    \draw[blue, very thick] (8) to [bend left=25] (16);
    \draw[blue, very thick] (8) to [bend left=20] (16);
    \draw[blue, very thick] (8) to [bend left=15] (16);
    \draw[blue, very thick] (8) to [bend left=30] (16);
\end{tikzpicture}
}
\label{f:fig2}
\end{figure}

\begin{lemma} \label{lem: ISgood2}
Let $\varepsilon=\omega(\log^{-\Cgl}n)$, and condition on any realization of $\Gamma$ from $\Event_{typ} \cap \Event_1^c$.
Let $m =\omega(\varepsilon^{-2}p^{-2}\log n) $ be an integer, and  
    let $X_1,\dots, X_m$ be i.i.d
    random variables uniformly distributed on $[n]$.
    Then, 
    \begin{align*} 
       \Prob_{X_1,\dots, X_m} \Big\{ 
            \forall\; \{v,w\} \in S_{A1}\;\;\; &\mbox{there are distinct indices } i, j \in [m] \mbox{ and a pair } \{v',w'\} \neq \{v,w\}\\
            &\mbox{such that } f_{X_i}(\{v,w\}) = f_{X_j}(\{v',w'\})
            \Big\} 
       = 1- n^{-\omega(1)}.
    \end{align*} 
\end{lemma}
\begin{proof}
Let $m':=\lfloor m/2\rfloor$.
Let $X_{1},X_{2},\dots, X_{m'}$ and $Y_1,Y_2,\dots, Y_{m'}$
be two independent sets of $m'$ i.i.d random variables uniformly distributed on $[n]$. 
It is sufficient to show that the event 
$$
    \Event_{I,X,Y}:=\{
        \forall\; \{v,w\} \in S_{A1} \;\;\; 
        \exists\; i,j \in [m'] \mbox{ and }  \{v',w'\} \neq \{v,w\}
         \mbox{ such that } f_{X_i}(\{v,w\}) = f_{Y_j}(\{v',w'\})
    \}
$$
satisfies 
$$
    \Prob( \Event_{I,X,Y}) = 1 - n^{-\omega(1)}.
$$

By definition of $S_{A1}$, for any $\{v,w\}\in S_{A1}$
$$ \Prob_{X'}\{ (X',\{v,w\}) \in A_1\} \ge \frac{\varepsilon^2 p^2n}{4n} = \frac{\varepsilon^2p^2}{4},$$
where $X'$ is a random variable uniformly distributed on $[n]$. 

Let $\Event_{I,X}$ be the event
$$
    \Event_{I,X} := \{ \forall\; \{v,w\} \in S_{A1}\;\;\; \exists j\; \in [m'] \mbox{ such that } (X_j,\{v,w\})\in A_1 \}.
$$
Note that
\begin{align}
\begin{split}
 \Prob( \Event_{I,X}^c) =&  \Prob\{ \exists\; \{v,w\} \in S_{A1}\;\;\;
\forall \;i\in [m'],\;\; (X_i,\{v,w\}) \notin A_1\} \\ \le&
|S_{A1}|
\Big(1 -  \frac{\varepsilon^2p^2}{4}\Big)^{m'}  
\le \exp \Big( 2\log(n)- \frac{\varepsilon^2p^2 m'}{4}\Big). \label{eq: IEIX}
\end{split}
\end{align}

By assuming $m' = \omega(\log(n)\varepsilon^{-2}p^{-2})$,
the event $\Event_{I,X}^c$ holds with probability $n^{-\omega(1)}$.
Next, fix an $m'$--tuple $(x_1,\dots, x_{m'})\in[n]^{m'}$ (a realization of $X_1,\dots,X_{m'}$)
such that $\Event_{I,X}$ holds.
For each $ \{v,w\} \in S_{A1}$, we pick $x_{v,w} \in \{x_i\}_{i=1}^m $ 
so that $ ( x_{v,w}, \{v,w\}) \in A_1$ and let 
$$U( \{v,w\}):= \{ i \in [n] \,:\, \exists \{v',w'\} \in S_\Gamma(i) \backslash \{v,w\} \mbox{ so that } f_{i}(\{v',w'\}) = f_{x_{v,w}}(\{v,w\})\}.$$
Notice that $f_i$ is an injective map, and thus 
$$|U(\{v,w\})| = | \{ (i, \{v',w'\}) \in A\,:\, \{v',w'\} \neq \{v,w\} \mbox{ and } f_{i}(\{v',w'\}) = f_{x_{v,w}}(\{v,w\}) \}|
\ge \frac{1}{2} \varepsilon^2p^2n,$$
where the last inequality follows from $(x_{v,w},\{v,w\}) \in A_1$ and the definition of $A_1$. Thus, for every $\{v,w\} \in S_{A1}$, 
$\Prob\{ Y' \in U(\{v,w\})\} \ge \frac{\varepsilon^2}{2p^2}$, where $Y'$ is a random variable uniformly distributed on $[n]$.

Similar to the upper estimate of $\Prob(\Event_{I,X}^c)$, our choice of $m'$ implies 
\begin{align}
 \Prob ( \Event_{I,X,Y}^c | \Event_{I,X})
\le n^{-\omega(1)}. \label{eq: IEIXY}
\end{align}
Combining \eqref{eq: IEIX} and \eqref{eq: IEIXY}, we get 
$$
    \Prob(\Event_{I,X,Y}) = 1- n^{-\omega(1)}.
$$
\end{proof}

\bigskip

The last lemma essentially means that conditioned on a realization of the graph $\Gamma$ from
$\Event_{typ} \cap \Event_1^c$ and sampling uniformly at random from the set of indices,
we are able to find a large number (namely, of order $|S_{A1}|$, in turn estimated from below as $\Omega(\varepsilon^2 n^2)$
in Corollary~\ref{-93871-98471098798}) of constraints on the graph $\Gamma$.
Those \co{ constraints} are ``spread'' over just a little more than $\log(n)\varepsilon^{-2}p^{-2}$
neighborhoods, yielding just a little less than $\varepsilon^4 n^2p^2/\log n$ constraints per neighborhood.
Simultaneous occurrence of such a number of constraints in $\Gamma$ happens with probability
$\exp(-\Omega(\varepsilon^4 n^2p^2/{\rm polylog }(n)))$.
On the other hand, as it has been mentioned in the introduction, the total number of choices of the {\it sets} of neighbors
both in $\Gamma$ and $\tilde\Gamma$ for a given vertex is of order $\exp(\Theta(pn\log n))$. 
In the regime $p\geq n^{-1/2}{\rm polylog}(n)$ the \co{inverse}
probability of constraints' satisfaction
beats the number of choices of the adjacent vertices, implying the required statement.
In what follows, we make the argument rigorous.
\begin{proof}[Proof of Proposition~\ref{19741094720987098}]
Let $\cal{D}$ be the collection of data structures of the following type:
$$ (i, \{i_1,i_2,\dots, i_k\}, g_i:\{i_1,\dots, i_k\} \mapsto [n]),$$
where $k$ may take any value in the interval $[\frac{2}{3} pn,\frac{4}{3} pn]$;
where $i\in [n]$; $\{i_1,\dots, i_k\}$ is any $k$--subset of $[n]$,
and $g_i$ is any mapping from $\{i_1,\dots, i_k\}$ to $[n]$.
Similarly to the isomorphisms $f_i$, we will be working with the natural extensions
of the mappings $g_i$ to sets of pairs of vertices.
Elements of $\cal{D}$ are meant to provide a partial description of 
the structure of corresponding $1$--neighborhoods of $\Gamma$ and $\tilde\Gamma$.
Since the graphs are random, a given structure from $\cal{D}$
may or may not accurately describe the respective $1$--neighborhoods, depending
on a realization of $\Gamma$ and $\tilde\Gamma$.
Given an element $D=
(i, \{i_1,i_2,\dots, i_k\}, g_i:\{i_1,\dots, i_k\} \mapsto [n])
\in \cal{D}$, we say that $\Gamma$ is {\it $D$--compatible} if 
\begin{enumerate}
\item $\{i_1,\dots, i_k\}$ is the set of all vertices in $\Gamma$ adjacent to $i$, and
\item $\forall \ell \in [k],   f_i(i_\ell)=g_{i}(i_\ell)$.
\end{enumerate}
Set $m:=\lfloor\varepsilon^{-2}p^{-2}\log^2 n\rfloor$, and let $\cal{D}^{\times m}$ be the $m$--fold Cartesian
product of $\cal{D}$:
$$\cal{D}^{\times m}= \{(D(1),\dots, D(m))\,:\, \forall t \in [m],\, D(t) \in \cal{D}\}.$$
Given an $m$--tuple $D^m\in \cal{D}^{\times m}$, we say that
$\Gamma$ is {\it $D^m$--compatible} if $\Gamma$
is compatible with all $m$ elements of $D^m$.

The proof of the proposition is accomplished by identifying a
special subset $\cal{D}_*^{\times m}$
of elements of $\cal{D}^{\times m}$ such that, on the one hand,
conditioned on $\Event_1^c(\varepsilon) \cap \Event_{typ}$
the graph $\Gamma$ is compatible with one of the structures in $\cal{D}_*^{\times m}$
with probability close to one whereas, on the other hand,
the unconditional probability
$$
\Prob\{\mbox{$\Gamma$ is $D^m$--compatible for some structure }D^m\in \cal{D}_*^{\times m}\}$$
is close to zero. The collection $\cal{D}_*^{\times m}$ thus
would correspond to ``non-typical'' realizations of $\Gamma,\tilde\Gamma$.

For each $D^m = (D(1),\dots, D(m)) \in \cal{D}^{ \co \times m}$, let 
$S(D^m) \subset \S_2$ be the set 
\begin{align*}
    S(D^m):= \bigg\{ \{v,w\}\,:\, &\exists\; t_1,t_2 \in [m] \mbox{ and } \{v',w'\} \in \S_2 \backslash \{\{v,w\}\} \\ 
    &\mbox{ such that } g_{i(t_1)}(\{v,w\}) = g_{i(t_2)}(\{v',w'\}) \bigg\}
\end{align*}
(in the above definition, we implicitly require that
$ v,w \in \{i_1(t_1),\dots, i_{k(t_1)}(t_1)\}$ and $v',w' \in \{i_1(t_2),\dots, i_{k(t_2)}(t_2)\}$). Then we define $\cal{D}_*^{\times m}$ as
the subset of $\cal{D}^{\times m}$ containing those structures
$D^m$ with $|S(D^m)| \ge \frac{\varepsilon^2}{5}{n \choose 2}
$.

Next, we will estimate the probability (with respect to the randomness of $\Gamma$)
that $\Gamma$ is $D^m$--compatible for a given element $D^m$ of $\cal{D}_*^{\times m}$.
Consider an auxiliary graph $G(D^m)$ with the vertex set $S(D^m)$ and such that
$ ( \{v,w\} , \{v',w'\})$ is an edge iff there exist $t_1\neq t_2$
with $\co{g_{i(t_1)}(\{v,w\})= g_{i(t_2)}(\{v',w'\})}$.

\begin{figure}[h]\label{adksjfnapifunw4fpoiqwnf}
\co{ 
\caption{The auxiliary graph $G(D^{10})$ for the data
structure $D^{10}$ corresponding to the pair of graphs $\Gamma$, $\tilde\Gamma$
from Figure~\ref{f:f1}, such that $\Gamma$ is $D^{10}$--compatible. We assume that $D^{10}=(D(1),\dots, D(10))$,
where for each $t\leq 10$, $i(t):=t$, $g_t:=f_t$, and the domain of $g_t$
coincides with that of $f_t$.
In this example, the graph $G(D^{10})$ has only one connected component
consisting of two points $\{2,4\}$ and $\{7,9\}$.}

\subfigure{
\begin{tikzpicture}[every node/.style={rectangle,thick,draw}]
\node (1) at (-2,0) {$\{2,4\}$};
\node (2) at (2,0) {$\{7,9\}$};
 \draw[blue, very thick] (1) to [bend left=8] (2);
\end{tikzpicture}
}
}
\end{figure}
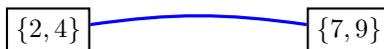

By definition of $S(D^m)$, none of the vertices of $G(D^m)$ are isolated. 
Let $G_1,\dots, G_r$ be the connected components of $G(D^m)$. 
The event that $\Gamma$ is $D^m$--compatible is contained in the event that
for every $G_s$ with $s \in [r]$,
either $ V(G_s) \subset E(\Gamma)$ or $V(G_s) \cap E(\Gamma) = \emptyset$
(see Figure~\ref{adksjfnapifunw4fpoiqwnf}).
The probability (with respect to the randomness of $\Gamma$) for a given index $s\in[r]$ can be estimated as
$$
    \Prob\{ V(G_s) \subset E(\Gamma)\mbox{ or } V(G_s) \cap E(\Gamma) = \emptyset\} = (1-p)^{|V(G_s)|} + p^{|V(G_s)|}
    \le \exp( - cp|V(G_s)| )
$$
for some universal constant $c\in (0,1)$.
Therefore, by independence,
$$ \Prob\{ \Gamma  \mbox{ is $D^m$--compatible}\} 
\le \prod_{s \in [r]}  \exp( - cp|V(G_s)| )
\le \exp( - cp|S(D^m)|) \le \exp \Big( -cp\frac{\varepsilon^2}{5} {n \choose 2}
\Big).$$
A rough estimate gives
\begin{align*}
   |\cal{D}^{\times m}|= |\cal{D}|^m 
   \le  (pn\cdot n^{4pn} )^m
   \le \exp( 5\log(n)pnm).
\end{align*}
Thus, by the union bound we have 
$$ \Prob\{\mbox{$\Gamma$ is $D^m$--compatible for some structure }D^m \in \cal{D}_*^{\times m} \}  \le 
\exp \Big( 5\log(n)pnm -c'p\varepsilon^2n^2  \Big),
$$
for some universal constant $c'>0$.

\co{In view of the conditions
$m \le \varepsilon^{-2}p^{-2}\log^2n$ and $\varepsilon = \omega( \log^{-\Cgl}(n))$, 
$$
\frac{ c' p \varepsilon^2 n^2 }{ 5 \log (n) pnm } 
=  \frac{c' \varepsilon^2 n}{5 \log (n)m }
\ge \frac{c' \varepsilon^4 p^2n}{5 \log^3 (n) }
= \omega( \log^{-4\Cgl-3}(n)p^2n) = \omega(1),
$$
where the last equality follows from our assumption that $p^2n = \omega( \log^{3+4\Cgl}(n))$. 
Hence,}
$$ \Prob\{\mbox{$\Gamma$ is $D^m$--compatible for some structure }D^m \in \cal{D}_*^{\times m} \}  \le n^{-\omega(1)}.
$$

On the other hand, by \eqref{eq: ISgood} and Lemma~\ref{lem: ISgood2}, 
$$ \Prob\{\mbox{$\Gamma$ is $D^m$--compatible for some structure }D^m\in \cal{D}_*^{\times m} \,\vert\, \Event_1^c(\varepsilon) \cap \Event_{typ}
\}  = 1- n^{-\omega(1)}, 
$$
which implies $\Prob(\Event_1^c(\varepsilon)\cap \Event_{typ}) = n^{-\omega(1)}$. 
The result follows.
\end{proof}

\section{Step II}\label{s:step2}
\co{In this step, we show that with high probability, there exists
a bijection $\pi :[n] \mapsto [n]$ such that for most vertices
$z$ of $\Gamma$, we have
$f_i(z)=\pi(z)$ for at least an
$(1-\varepsilon)$--fraction of neighbors $i$ of $z$ in $\Gamma$. The permutation
$\pi$ thus can be viewed as an ``approximate'' graph isomorphism between $\Gamma$
and $\tilde \Gamma$.}

Let $\varepsilon>0$.
We define \co{the} event
\begin{align*}
\Event_2(\varepsilon):=
\big\{&\mbox{There is a permutation $\pi:[n]\to[n]$ with }|\{i\in [n]:\, \{i,z\}\in E(\Gamma),\,
f_i(z)=\pi(z)
\}|\geq (1-\varepsilon)pn\\
&\mbox{for at least $(1-\varepsilon)n$ vertices $z\in[n]$}\big\}.
\end{align*}

The main statement of the section is
\begin{prop}\label{138413-4982174-29847}
There is a universal constant 
$C_2>0$ with the following property.
Let $n$ be large, let $\varepsilon\in[\log^{-\Cgl}n,1/2]$, and assume that $p^2n\geq \log^{\Cgl}n$.
Then the event $\Event_1(\varepsilon)\cap \Event_{typ}$
is contained in the event $\Event_2(C_2
\varepsilon^{1/3})$.
\end{prop}

First, we will restate the conditions described by $\Event_1(\varepsilon)$, in a more technical
yet more useful form. Namely, we will show that the condition that many vertex pairs of $\Gamma$
are $(1-\varepsilon)$--focused implies that there are many vertices $z$ which are elements of {\it many}
$(1-\varepsilon)$--focused
pairs in {\it many} neighborhoods \co{(see Definition~\ref{aifjnapifjnpfqijwnfpiqjnfpi}
for the notion of $(1-\varepsilon)$--focused
pairs).}
\begin{lemma}\label{1-83471-94871987}
There is a universal constant $C''>0$ with the following property.
Let $n$ be sufficiently large, and assume that $\varepsilon\in[\log^{-\Cgl}n,1/2]$.
Then, conditioned on any realization of $\Gamma$ from
$\Event_1(\varepsilon)\cap \Event_{typ}$,
there are at least $(1-C'' \varepsilon^{1/3})n$ vertices $z\in[n]$ such that
\begin{align*}
\big|\big\{i\snb{\Gamma}z:
\;&|\{v\in V(N_\Gamma(i))\setminus\{i,z\}:\;
\{v,z\}\mbox{ is $(1-\varepsilon)$--focused and is mapped to its focus \co{by $f_i$}}\}|\\
&\geq (1-C''\varepsilon^{1/3})pn\big\}\big|
\geq (1-C''\varepsilon^{1/3})pn.
\end{align*}
\end{lemma}
\begin{proof}
The proof is accomplished with a simple a counting argument combined with the definition of $\Event_{typ}$.
For every $z\in [n]$, let $I_z:=V(N_\Gamma(z))\setminus\{z\}$ be the set of neighbors of $z$ in $\Gamma$;
let $T\subset \S_2$ be the set of $(1-\varepsilon)$--focused pairs, and given an ordered triple $(z,i,v)$,
let ${\bf 1}_{(z,i,v)}$ be the indicator of the expression
$$
i\in I_z,\;v\in V(N_\Gamma(i))\setminus\{i,z\},\;\{z,v\}\mbox{ is $(1-\varepsilon)$--focused
and is mapped to its focus \co{by $f_i$}}.
$$
Using the definition of $\Event_1(\varepsilon)$, we get
\begin{equation}\label{1-39871-294821-984}
\sum_{z=1}^n\sum_{i\in I_z}\sum_{v\in V(N_\Gamma(i))\setminus\{i,z\}}{\bf 1}_{(z,i,v)}
\geq (1-\varepsilon)^2p^2n\cdot n(n-1).
\end{equation}
Let the constant $C''>0$ be chosen later.
Let $U$ be the set of \co{vertices} $z\in[n]$ satisfying the assertion of the lemma.
\co{For every $z\in U^c$, let $\tilde I_z$ be the subset of all indices $i\in I_z$
with
$$|\{v\in V(N_\Gamma(i))\setminus\{i,z\}:\;\{z,v\}\mbox{ is $(1-\varepsilon)$--focused
and is mapped to its focus \co{by $f_i$}}\}|\geq (1-C''\varepsilon^{1/3})pn.$$
Note that $|\tilde I_z|\leq (1-C''\varepsilon^{1/3})pn$, by the definition of $U^c$.}
On the other hand, in view of the definition of $\Event_{typ}$, we have
$|I_z|\leq (1+\log^{-\Cgl}n)pn$ and $|V(N_\Gamma(i))|\leq (1+\log^{-\Cgl}n)pn$ for all $z,i\in [n]$.
Consequently,
\begin{align*}
\sum_{z\in U^c}\sum_{i\in I_z}\sum_{v\in V(N_\Gamma(i))\setminus\{i,z\}}{\bf 1}_{(z,i,v)}
&\leq\co{ \sum_{z\in U^c}\sum_{i\in \tilde I_z}(1+\log^{-\Cgl}n)np
+\sum_{z\in U^c}\sum_{i\in I_z\setminus\tilde I_z}(1-C''\varepsilon^{1/3})pn}
\\
&=\co{|U^c|\Big(|\tilde I_z|(1+\log^{-\Cgl}n)pn+(|I_z|-|\tilde I_z|)(1-C''\varepsilon^{1/3})pn\Big)}\\
&\leq
|U^c|\bigg((1-C''\varepsilon^{1/3})pn(1+\log^{-\Cgl}n)np
+(C''\varepsilon^{1/3} pn+pn\log^{-\Cgl}n)(1-C''\varepsilon^{1/3})pn\bigg),
\end{align*}
whence
\begin{align*}
\sum_{z=1}^n\sum_{i\in I_z}\sum_{v\in V(N_\Gamma(i))\setminus\{i,z\}}{\bf 1}_{(z,i,v)}
&\leq \co{\sum_{z\in U}|I_z|\,\max\limits_{i\in[n]}|V(N_\Gamma(i)|+
\sum_{z\in U^c}\sum_{i\in I_z}\sum_{v\in V(N_\Gamma(i))\setminus\{i,z\}}{\bf 1}_{(z,i,v)}}\\
&\leq |U|\big((1+\log^{-\Cgl}n)pn\big)^2
+|U^c|\big(n^2p^2+2n^2p^2\log^{-\Cgl}n-(C'')^2\varepsilon^{2/3} n^2p^2 \big)\\
&\leq n^3p^2+5n^3p^2\log^{-\Cgl}n-(C'')^2\varepsilon^{2/3} n^2p^2|U^c|.
\end{align*}
It is easy to see that if $|U^c|\geq C''\varepsilon^{1/3}n$ and assuming $C''$ is sufficiently
large, the last inequality would contradict \eqref{1-39871-294821-984}.
The result follows.
\end{proof}

\bigskip

Before moving on to the proof of Proposition~\ref{138413-4982174-29847}, we consider the following lemma.
\begin{lemma}\label{193847194872109487}
Let $m\geq 1$, and let $U_1,U_2,\dots,U_\ell$ be subsets of $[m]$.
Then
$$
\sum_{i,j\in[\ell]}|U_i\cap U_j|\geq\frac{\co{( \sum_{i\in[\ell]}|U_i|)^2}}{m}.
$$
\end{lemma}
\begin{proof}
Let $\xi$ be a random variable uniformly distributed in $[m]$,
so that
$$
\sum_{i,j\in[\ell]}|U_i\cap U_j|=\sum_{i,j\in[\ell]}m\cdot \Exp\,{\bf 1}_{\{\xi\in U_i\cap U_j\}}
=m\,\Exp\,\bigg(\sum_{i,j\in[\ell]}{\bf 1}_{\{\xi\in U_i\}}\cdot {\bf 1}_{\{\xi\in U_j\}}\bigg)
=m\,\Exp\,\bigg(\sum_{i\in[\ell]}{\bf 1}_{\{\xi\in U_i\}}\bigg)^2.
$$
Applying Jensen's inequality, we then get
$$
\sum_{i,j\in[\ell]}|U_i\cap U_j|\geq m\,\bigg(\Exp\,\sum_{i\in[\ell]}{\bf 1}_{\{\xi\in U_i\}}\bigg)^2
=\frac{\co{( \sum_{i\in[\ell]}|U_i|)^2}}{m}.
$$
\end{proof}

\bigskip

\begin{proof}[Proof of Proposition~\ref{138413-4982174-29847}]
Let $\varepsilon$ satisfy the assumptions of Lemma~\ref{1-83471-94871987},
and let $C''$ be the constant from the lemma.
Condition on any realization of $\Gamma$ in $\Event_1(\varepsilon)\cap \Event_{typ}$, and 
fix any point $z\in [n]$ satisfying the assertion of Lemma~\ref{1-83471-94871987}.
Let $J_z$ be the set of all neighbors $i$ of $z$ in $\Gamma$ such that
$$
|\{v\in V(N_\Gamma(i))\setminus\{i,z\}:\;
\{v,z\}\mbox{ is $(1-\varepsilon)$--focused and is mapped to its focus \co{by $f_i$}}\}|\\
\geq (1-C''
\varepsilon^{1/3})pn.
$$
Note that $|J_z|\geq (1-C''
\varepsilon^{1/3})pn$.
Denote
$$
M_z(i):=\big\{v\in V(N_\Gamma(i))\setminus\{i,z\}:\;
\{v,z\}\mbox{ is $(1-\varepsilon)$--focused and is mapped to its focus \co{by $f_i$}}\big\},\quad i\in J_z.
$$
We construct an auxiliary multigraph $G_z$ with the vertex set $J_z$,
such that for every pair of distinct elements $i_1,i_2\in J_z$, the edge between $i_1$ and $i_2$
has multiplicity equal to the size of the intersection
$M_z(i_1)\cap M_z(i_2)$ (see Figure~\ref{aldkjfnafijnwefijwfijwnij}).

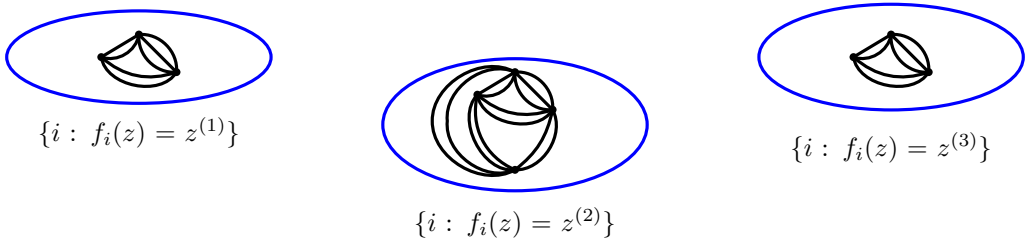
\begin{figure}[h]\label{aldkjfnafijnwefijwfijwnij}
\caption{An illustration of the auxiliary graph $G_z$ used
to estimate from above the size of the set of pairs of indices $i_1\neq i_2$ from $J_z$
with $f_{i_1}(z)\neq f_{i_2}(z)$ (denoted by $Q$ in the proof of Proposition~\ref{138413-4982174-29847}).
The graph $G_z$ splits into communities in accordance with the values of $f_i(z)$.
In this illustration, there are three communities $\{i:\,f_i(z)=z^{(\ell)}\}$, $\ell=1,2,3$.
Any pair of vertices from different communities are either not adjacent
or are connected by an edge of multiplicity one.
Further, any pair of vertices within a same community are connected by a multiedge
of multiplicity at most $(1+\log^{-\Cgl}n)p^2n$,
in view of conditioning on $\Event_{typ}$.}

\centering  

\subfigure
{
\begin{tikzpicture}
    \usetikzlibrary{arrows}
    \usetikzlibrary{shapes}
    \tikzstyle{every node}=[ellipse, minimum width=100pt,
    align=center]
    \draw  node[fill,circle,inner sep=0pt,minimum size=3pt] at (-5,0.5) {};
    \draw  node[fill,circle,inner sep=0pt,minimum size=3pt] at (-4.5,0) {};
    \draw  node[fill,circle,inner sep=0pt,minimum size=3pt] at (-5.5,0.2) {};
    \draw[black, very thick] (-5,0.5) to [bend left=50] (-4.5,0);
    \draw[black, very thick] (-5,0.5) to [bend right=50] (-4.5,0);
    \draw[black, very thick] (-5,0.5) to (-4.5,0);
    \draw[black, very thick] (-5.5,0.2) to (-5,0.5);
    \draw[black, very thick] (-5.5,0.2) to [bend right=50] (-5,0.5);
    \draw[black, very thick] (-5.5,0.2) to [bend right=30] (-4.5,0);
    \draw[black, very thick] (-5.5,0.2) to [bend right=70] (-4.5,0);
    \node[blue, draw=blue, very thick, minimum height=35pt] (1) at (-5,0.2) {};
    \node[text width=3cm] at (-5,-0.8) {$\{i:\,f_i(z)=z^{(1)}\}$};
    \draw  node[fill,circle,inner sep=0pt,minimum size=3pt] at (0,0) {};
    \draw  node[fill,circle,inner sep=0pt,minimum size=3pt] at (0.5,-0.5) {};
    \draw  node[fill,circle,inner sep=0pt,minimum size=3pt] at (-0.5,-0.3) {};
    \draw  node[fill,circle,inner sep=0pt,minimum size=3pt] at (0,-1.3) {};
    \draw[black, very thick] (0,0) to [bend left=50] (0.5,-0.5);
    \draw[black, very thick] (0,0) to [bend right=50] (0.5,-0.5);
    \draw[black, very thick] (0,0) to (0.5,-0.5);
    \draw[black, very thick] (-0.5,-0.3) to (0,0);
    \draw[black, very thick] (-0.5,-0.3) to [bend right=50] (0,0);
    \draw[black, very thick] (-0.5,-0.3) to [bend right=30] (0.5,-0.5);
    \draw[black, very thick] (-0.5,-0.3) to [bend right=70] (0.5,-0.5);
    \draw[black, very thick] (-0.5,-0.3) to [bend right=40] (0,-1.3);
    \draw[black, very thick] (-0.5,-0.3) to [bend right=70] (0,-1.3);
    \draw[black, very thick] (0,0) to  [bend right=60] (-1.1,-0.65) to [bend right=60] (0,-1.3);
    \draw[black, very thick] (0,0) to  [bend right=60] (-0.9,-0.65) to [bend right=60] (0,-1.3);
    \draw[black, very thick] (0.5,-0.5) to [bend left=20] (0,-1.3);
        \draw[black, very thick] (0.5,-0.5) to [bend left=60] (0,-1.3);
    \node[blue, draw=blue, very thick, minimum height=50pt] (2) at (0,-0.7) {};
    \node[text width=3cm] at (0,-2) {$\{i:\,f_i(z)=z^{(2)}\}$};
    \draw  node[fill,circle,inner sep=0pt,minimum size=3pt] at (5,0.5) {};
    \draw  node[fill,circle,inner sep=0pt,minimum size=3pt] at (5.5,0) {};
    \draw  node[fill,circle,inner sep=0pt,minimum size=3pt] at (4.5,0.2) {};
    \draw[black, very thick] (5,0.5) to [bend left=50] (5.5,0);
    \draw[black, very thick] (5,0.5) to [bend right=50] (5.5,0);
    \draw[black, very thick] (5,0.5) to (5.5,0);
    \draw[black, very thick] (4.5,0.2) to (5,0.5);
    \draw[black, very thick] (4.5,0.2) to [bend right=50] (5,0.5);
    \draw[black, very thick] (4.5,0.2) to [bend right=30] (5.5,0);
    \draw[black, very thick] (4.5,0.2) to [bend right=70] (5.5,0);
    \node[blue, draw=blue, very thick, minimum height=40pt] (3) at (5,0.2) {};
    \node[text width=3cm] at (5,-1) {$\{i:\,f_i(z)=z^{(3)}\}$};
\end{tikzpicture}
}

\end{figure}

Thus, the total multiplicity of the edges of $G_z$ is given by
$$
\frac{1}{2}\sum_{i_1\neq i_2,\;i_1,i_2\in J_z}|M_z(i_1)\cap M_z(i_2)|,
$$
where the sets $M_z(i)$ satisfy
$$
\sum_{i\in J_z}|M_z(i)|\geq |J_z|(1-C''
\varepsilon^{1/3})pn.
$$
Applying Lemma~\ref{193847194872109487}, we immediately get
that the total multiplicity of the edges of $G_z$ is at least
\begin{align}
\frac{1}{2}(1-C''
\varepsilon^{1/3})^2p^2n|J_z|^2-\frac{1}{2}\sum_{i\in J_z}|M_z(i)|
&\geq \frac{1}{2}(1-C''
\varepsilon^{1/3})^2p^2n|J_z|^2
-\frac{1}{2}|J_z|(1+\log^{-\Cgl}n)pn.\label{aojfnapeifunpfiwunfpqijfnap}
\end{align}

Next, we will get an upper bound on the total multiplicity.
Note that the size of the intersection $M_z(i_1)\cap M_z(i_2)$ (hence the multiplicity
of any multiedge in $G_z$) does not exceed $(1+\log^{-\Cgl}n)p^2n$,
in view of the definition of $\Event_{typ}$.
Further, whenever $f_{i_1}(z)\neq f_{i_2}(z)$, the multiedge between $i_1$ and $i_2$
can only have multiplicity zero or one. Indeed, otherwise we would be able to find two distinct
vertices $v,w\in M_z(i_1)\cap M_z(i_2)$, such that both $\{v,z\}$ and $\{w,z\}$ are mapped
to their respective focuses by both $f_{i_1}$ and $f_{i_2}$.
If $f_{i_1}(z)=\tilde z$ and $f_{i_2}(z)=\hat z$ for some $\tilde z\neq \hat z$ then necessarily $f_{i_1}(\{v,z\})
=f_{i_2}(\{v,z\})=\{\tilde z,\hat z\}$, and, similarly, $f_{i_1}(\{w,z\})
=f_{i_2}(\{w,z\})=\{\tilde z,\hat z\}$. But this is impossible since $v,w,z$ must be mapped to three distinct vertices
of $\tilde \Gamma$ (both by $f_{i_1}$ and $f_{i_2}$).

Thus, if we denote by $Q$
the subset of all $2$--subsets $\{i_1,i_2\}$ of $J_z$ such that $f_{i_1}(z)\neq f_{i_2}(z)$
then the total multiplicity of multiedges of $G_z$ can be estimated from above by
$$
|Q|+(1+\log^{-\Cgl}n)p^2n\bigg(\frac{1}{2}|J_z|(|J_z|-1)-|Q|\bigg)
\leq \frac{1}{2}|J_z|(|J_z|-1)(1+\log^{-\Cgl}n)p^2n-|Q|p^2n.
$$
Together with the lower bound \co{\eqref{aojfnapeifunpfiwunfpqijfnap}}, this implies the inequality
\begin{align*}
|Q|
&\leq \frac{1}{2}|J_z|^2(1+\log^{-\Cgl}n)+\frac{1}{2p}|J_z|(1+\log^{-\Cgl}n)-\frac{1}{2}(1-C''
\varepsilon^{1/3})^2|J_z|^2\\
&\leq 
\co{\frac{1}{2}\,n^2p^2(1+\log^{-\Cgl}n)^3
+\frac{1}{2}\,n(1+\log^{-\Cgl}n)^2-\frac{1}{2}(1-C''
\varepsilon^{1/3})^2(1+\log^{-\Cgl}n)^2n^2p^2,}
\end{align*}
which in turn gives, \co{as $p^2n\varepsilon^{1/3}=\Omega(1)$,}
$$
|Q|\leq C\varepsilon^{1/3}n^2p^2
$$
for some universal constant $C>0$.
Thus, there is a vertex $\pi'(z):=\tilde z$ of $\tilde \Gamma$ such that $f_i(z)=\pi'(z)$
for at least $(1-C'\varepsilon^{1/3})pn$ indices $i\in J_z$.

\medskip

The above argument produced a mapping $\pi':[n]\to[n]$ such that for at least $(1-C'\varepsilon^{1/3})n$ vertices
$z\in[n]$, we have $f_i(z)=\pi'(z)$ for at least $(1-C'\varepsilon^{1/3})pn$ indices $i\in J_z$.
The last step of the proof is to ``convert'' the mapping $\pi'$ to a permutation $\pi$ of $[n]$
with similar properties.
\co{We have, in view of the conditions on $\pi'$,
$$
\sum_{\tilde w\in \pi'([n])}\sum_{i=1}^n\sum_{v\snb{\Gamma}i
}{\bf 1}_{\{f_i(v)=\tilde w\}}
=\sum_{i=1}^n\sum_{v\snb{\Gamma}i
}{\bf 1}_{\{f_i(v)\in\pi'([n])\}}
=\sum_{z=1}^n\sum_{i\snb{\Gamma}z
}{\bf 1}_{\{f_i(z)\in\pi'([n])\}}
\geq (1-C'\varepsilon^{1/3})^2 n^2p.
$$
On the other hand, we know that each vertex of $\Gamma$, hence $\tilde\Gamma$, has degree at most $(1+\log^{-\Cgl}n)pn$
on $\Event_{typ}$, whence
$$
\sum_{i=1}^n\sum_{v\snb{\Gamma}i
}{\bf 1}_{\{f_i(v)=\tilde w\}}\leq (1+\log^{-\Cgl}n)pn
$$
for every $\tilde w\in \pi'([n])$.
Thus,
$$
(1-C'\varepsilon^{1/3})^2 n^2p\leq (1+\log^{-\Cgl}n)pn |\pi'([n])|,
$$
implying that $|\pi'([n])|\geq (1-\tilde C\varepsilon^{1/3})n$ for some constant $\tilde C>0$.
Let $W$ be a maximal subset of $[n]$ on which the mapping $\pi'$ is injective, that is $W$
is a subset such that for every $z\in W^c$, $\pi'(z)\in \pi'(W)$.
From the above, $|W|\geq (1-\tilde C\varepsilon^{1/3})n$.
Choose any permutation $\pi$ on $[n]$ such that $\pi(z)=\pi'(z)$ for all $z\in W$.
Then 
$\pi(z)=\pi'(z)$ for at least
$(1-\tilde C\varepsilon^{1/3})n$ points $z\in[n]$.}
It is easy to check that $\pi$ satisfies the required properties.
\end{proof}

\section{Step III}\label{s:step3}

The goal of this section is to show that with high probability the permutation from event $\Event_2(\varepsilon)$
\co{can be chosen to be} close to the identity.
Let $\varepsilon>0$.
We define the event
\begin{align*}
\Event_3(\varepsilon):=
\bigg\{\sum_{i\in[n]}|\{z\snb{\Gamma}i:\;f_i(z)=z
\}|\geq (1-\varepsilon)n^2p\bigg\}.
\end{align*}

\begin{prop}\label{0983701984721-94871-}
There is a universal constant $C_3
>0$ with the following property.
Let $n$ be large and let \co{$p =\omega\big( \frac{\log^{2\Cgl}n}{n}\big)$ and $\max\{2ep, \log^{-\Cgl} n\}\le \varepsilon \le \frac{1}{2}$. }
Then
$$
\Prob(\Event_3(C_3
\varepsilon)\cap \Event_{typ})
\geq \Prob(\Event_2(\varepsilon)\cap \Event_{typ})-n^{-\omega(1)}.$$
\end{prop}

The crucial part of the argument is encapsulated in the lemma below,
which establishes a connection between certain properties of the permutation $\pi$
from $\Event_2(\varepsilon)$ and the number of constraints imposed on the edges of $\Gamma$.
\begin{lemma}\label{10983740219487109487}
Let $\varepsilon\in[\log^{-\Cgl}n,1/2]$.
Condition on a realization of the graph $\Gamma$ from $\Event_2(\varepsilon)\cap \Event_{typ}$,
and let $\pi$ be the permutation from the definition of $\Event_2(\varepsilon)$.
Denote by $W$ the collection of all ordered pairs $(v,w)$, with $v\neq w$,
such that $\{v,w\}$ is an edge of $\Gamma$ and $f_v(w)=\pi(w)\neq w$.
Further, let $m:=\lceil 2n\log n\rceil$, and let $X_j,Y_j$, $j=1,2,\dots,m$,
be i.i.d variables uniformly distributed on $[n]$. Then with probability at least
$$
\max(0,|W|n^{-2}-3\varepsilon p-2n^{-1}-8m n^{-2})^m,
$$
all unordered pairs $\{X_j,Y_j\}, \{\pi^{-1}(X_j),\pi(Y_j)\}$, $j=1,2,\dots,m$, are distinct edges of $\Gamma$.
\end{lemma}
\begin{proof}
Condition on a realization of the graph $\Gamma$ from $\Event_2(\varepsilon)\cap \Event_{typ}$,
and let $\pi$ be the permutation from the definition of $\Event_2(\varepsilon)$.
For each $i\in[n]$, denote by $U_i$ the collection of all vertices $v\snb{\Gamma}i$ 
such that $f_i(v)=\pi(v)$.
Observe that, in view of the definition of $\Event_2(\varepsilon)$,
$$
\sum_{i=1}^n |U_i|\geq (1-\varepsilon)pn\cdot (1-\varepsilon)n\ge \co{( 1 - 2 \varepsilon) n^2p. }
$$

Let $X,Y$ be independent variables uniform on $[n]$.
Obviously then,
$(X,Y)$ belongs to $W$ with probability $|W|n^{-2}$, and, conditioned on the event $\{(X,Y)\in W\}$,
the pair $(X,Y)$ is uniformly distributed on $W$.
Further, let $\tilde T$ be the collection of all ordered pairs $(\tilde w,\tilde v)$ such that 
$\{\tilde w,\tilde v\}$ is an edge of $\tilde \Gamma$ and $f_{\tilde w}(\pi^{-1}(\tilde v))\neq\tilde v$.
Note that
$$
|\tilde T|=\sum_{i=1}^n |(V(N_\Gamma(i))\setminus\{i\})\setminus U_i|
=\sum_{i=1}^n |(V(N_\Gamma(i))\setminus\{i\})|-\sum_{i=1}^n |U_i|
\leq (1+\log^{-\Cgl}n)n^2p-(1-2\varepsilon)n^2p\leq 3\varepsilon n^2p.
$$
Thus, we get that with probability at least $|W|n^{-2}-3\varepsilon p$,
the following holds:
\begin{align} \label{eq: XYcondition}
(X,Y)\in W\quad \mbox{ and }\quad(\pi(Y),X)\notin \tilde T.
\end{align}
\co{Assume for a moment that condition \eqref{eq: XYcondition} holds. The inclusion $(X,Y) \in W$ implies that 
$\{f_X(X), f_X(Y)\} = \{ X, \pi(Y)\}$ is an edge of $\tilde \Gamma$. Then, the condition $( \pi(Y), X) \notin \tilde T $ is equivalent to $f_{\pi(Y)}(\pi^{-1}(X))=X$. Thus, 
$\{ f_{\pi(Y)}^{-1}(X) , f_{\pi(Y)}^{-1}(\pi(Y))\} 
= \{\pi^{-1}(X), \pi(Y)\} 
$ is an edge in $\Gamma$. To summarize,
$$
    \eqref{eq: XYcondition}\;\; \Rightarrow \{X,Y\}, \{ \pi^{-1}(X), \pi(Y)\}\in E(\Gamma) \mbox{ and } \pi(Y) \neq Y,
$$
where the condition $\pi(Y) \neq Y$ follows immediately as $(X,Y) \in W$.}
Since $X\notin\{Y,\pi(Y)\}$ with probability at least $1-2n^{-1}$, we get
that the event
$$
\mbox{$\{X,Y\}$ and $\{\pi^{-1}(X),\pi(Y)\}$ are distinct edges of $\Gamma$}
$$
holds with probability at least $|W|n^{-2}-3\varepsilon p-2n^{-1}$.

In view of the above observations, with probability at least
$$
\max(0,|W|n^{-2}-3\varepsilon p-2n^{-1}-\co{4}m n^{-2})^m,
$$
all unordered pairs $\{X_j,Y_j\}, \{\pi^{-1}(X_j),\pi(Y_j)\}$, $j=1,2,\dots,m$, are distinct edges of $\Gamma$.
\end{proof}

\begin{lemma}\label{103987410294870497}
Let $\varepsilon\in[\log^{-\Cgl}n,1/2]$.
Let event $\Event$ be defined as
\begin{align*}
\Event:=
\big\{&\mbox{There is a permutation $\pi:[n]\to[n]$ with $|\{j\in[n]:\,\pi(j)\neq j\}|\geq 14\varepsilon n$}\\
&\mbox{and such that }|\{i\in [n]:\,z\snb{\Gamma}i 
,\;f_i(z)=\pi(z)
\}|\geq (1-\varepsilon)pn\\
&\mbox{for at least $(1-\varepsilon)n$ vertices $z\in[n]$}\big\}.
\end{align*}
Condition on any realization of the graph $\Gamma$ from $\Event\cap \Event_{typ}$.
Let $W$ be defined as in Lemma~\ref{10983740219487109487}.
Then necessarily $|W|\geq 4\varepsilon n^2 p$.
\end{lemma}
\begin{proof}
Let $Q$ be the collection of all ordered pairs $(v,w)$
with $v\neq w$,
such that $\{v,w\}$ is an edge of $\Gamma$ and $f_v(w)=\pi(w)$.
By the definition of $\Event$, we get $|Q|\geq (1-\varepsilon)pn\cdot (1-\varepsilon)n$.
Let $J$ be the set of indices $j\in[n]$ with $\pi(j)\neq j$, and let $U$ be the collection
of all ordered pairs of the form $(v,j)$ where $j\in J$ and $\{v,j\}$ is an edge of $\Gamma$.
From the definition of $\Event_{typ}$ we get $|U|\geq |J|(1-\log^{-\Cgl}n)pn$. On the other hand,
the total number of pairs $(v,w)$ corresponding to edges of $\Gamma$ is at most $n\cdot (1+\log^{-\Cgl}n)pn$.
Thus,
$$
|Q|+|U|-|Q\cap U|\leq n\cdot (1+\log^{-\Cgl}n)pn,
$$
implying
$$
|W|=|Q\cap U|\geq (1-\varepsilon)pn\cdot (1-\varepsilon)n+|J|(1-\log^{-\Cgl}n)pn-n\cdot (1+\log^{-\Cgl}n)pn
\geq |J|pn/2-3\varepsilon n^2p.
$$
Since $|J|\geq 14\varepsilon n$, we get the result.
\end{proof}

\begin{proof}[Proof of Proposition~\ref{0983701984721-94871-}]
Let $\Event$ be defined as in Lemma~\ref{103987410294870497},
let $m:=\lceil 2n\log n\rceil$, and let $X_j,Y_j$, $j=1,2,\dots,m$,
be i.i.d variables uniformly distributed on $[n]$ mutually independent with $\Gamma$.
Let $M$ be the set of all permutations $\sigma$ with
$|\{j\in[n]:\,\sigma(j)\neq j\}|\geq 14\varepsilon n$, and for each $\sigma\in M$ denote by
$\Event_\sigma$ the event
\begin{align*}
\Event_\sigma:=
\big\{&|\{i\in [n]:\,z\snb{\Gamma}i 
,\;f_i(z)=\sigma(z)
\}|\geq (1-\varepsilon)pn\\
&\mbox{for at least $(1-\varepsilon)n$ vertices $z\in[n]$}\big\}.
\end{align*}
Obviously, $\Event=\bigcup_{\sigma\in M}\Event_\sigma$, and,
combining Lemmas~\ref{10983740219487109487} and~\ref{103987410294870497}, we get that
conditioned on any $\Gamma$ in $\Event_\sigma\cap \Event_{typ}$,
with conditional probability at least
$$
\max(0,4\varepsilon p-3\varepsilon p-2n^{-1}-8m n^{-2})^m\geq \bigg(\frac{\varepsilon p}{2}\bigg)^m
$$
all pairs $\{X_j,Y_j\}, \{\sigma^{-1}(X_j),\sigma(Y_j)\}$, $j=1,2,\dots,m$, are distinct edges of $\Gamma$.

On the other hand, for \co{an} arbitrary permutation $\sigma$ the (unconditional) probability that
the pairs $\{X_j,Y_j\}$ and $\{\sigma^{-1}(X_j),\sigma(Y_j)\}$, $j=1,2,\dots,m$, are distinct edges of $\Gamma$,
is at most $p^{2m}$. Hence, for every $\sigma\in M$ we get
\co{ 
\begin{align*}
\Prob(\Event_\sigma\cap \Event_{typ})&\leq
\frac{\Prob\big\{ \forall j, \{X_j, Y_j\} \in \Gamma, \{ \sigma^{-1}(X_j), \sigma(Y_j)\} \in \Gamma, \mbox{ and } 
\{X_j, Y_j\} \neq \{ \sigma^{-1}(X_j), \sigma(Y_j)\}\big\}}{\Prob\big\{ \forall j\in [m], \{X_j, Y_j\} \in \Gamma, \{ \sigma^{-1}(X_j), \sigma(Y_j)\} \in \Gamma, \mbox{ and } 
\{X_j, Y_j\} \neq \{ \sigma^{-1}(X_j), \sigma(Y_j)\}\, \big|\, 
\Event_\sigma\cap \Event_{typ} \big\} 
} \\
&\le  \frac{p^{2m}}{ ( \varepsilon p/ 2)^m} =\bigg(\frac{2 p}{\varepsilon}\bigg)^m
\leq \exp(-m),
\end{align*}
where the last inequality follows from the assumption
$ 2pe \le \varepsilon$. }

Taking the union bound over all $\sigma\in M$, we get that the event
$
\Event\cap \Event_{typ}$ has probability $n^{-\omega(1)}$,
so that
$$
\Prob((\Event_2(\varepsilon)\setminus\Event)\cap \Event_{typ})\geq \Prob((\Event_2(\varepsilon)\cap \Event_{typ})
-n^{-\omega(1)}.
$$

It remains to check that $(\Event_2(\varepsilon)\setminus\Event)\cap \Event_{typ}$
is contained in $\Event_3(C\varepsilon)\cap \Event_{typ}$ for a sufficiently large universal constant $C>0$.
\co{Fix any realization of $\Gamma$ within the event
$(\Event_2(\varepsilon)\backslash \Event) \backslash \Event_{typ}$, and let $\pi$ be the corresponding permutation from $\Event_2(\varepsilon)$. Then the set 
$\{z\in [n]\,:\, \pi(z)\neq z\}$ has cardinality at most $14 \varepsilon n$. 
Combined with \eqref{eq: EtypDegree}, this implies that the set
$ 
O_1 := \{ (i,z)\,:\, i \snb{\Gamma} z,\,\, \pi(z) \neq z \} 
$
has cardinality at most 
$
 (1+ \log^{-\Cgl}(n)) pn \cdot    14 \varepsilon n  \le 15 \varepsilon n^2p. 
$
Next, by the definition of $\Event_2(\varepsilon)$, the set 
$
O_2 := \{ (i,z) \,:\, i \snb{\Gamma} z,\,\, f_i(z) = \pi (z) \} 
$
has cardinality at least 
$
    (1 - \varepsilon) pn \cdot (1-\varepsilon) n \ge (1-2 \varepsilon)n^2p,
$
and thus $ |O_2\backslash O_1| \ge (1- 17 \varepsilon ) n^2p$. On the other hand, 
$$
O_2 \backslash O_1 \subset \{ (i,z) \,: \, i \snb{\Gamma} z,\, f_i(z)=z \}, 
$$
which implies that we are within the event $\Event_3(17 \varepsilon)$.}
\end{proof}

\section{Step IV}\label{s:step4}

In this section, we complete the proof of Theorem~\ref{309847109847}.
Our goal is to develop a bootstrapping argument which would allow to pass from the conclusion of Step III
to the desired statement, \co{specifically, we will show that
if the permutation $\pi$ from the definition of $\Event_2(\varepsilon)$ is ``close'' to the identity, then it {\it must} be the identity.}
Recall some definitions from the introduction:
\begin{align*}
 \cal{P} &:= \{ (v,w) \in [n]^{\times 2} \,:\, \{v,w\} \in E(\Gamma)\},\\
 \tilde{\cal{P}} &:= \{ (\tilde{v},\tilde{w}) \in [n]^{\times 2} \,:\, \{\tilde{v},\tilde{w}\} \in E(\tilde{\Gamma})\},\\ 
\cal{M}&: = \{ (v,w) \in \cal{P} \,:\, f_v(w)=w\}.
\end{align*}
For every pair of vertices $(v,w)\in\cal{P}$, let 
$$
    \cal{V}(v,w) := \{ z \,:\, (v,z) \in \cal{M} \mbox{ and } \{w,z\}\in E(\Gamma) \},
$$ 
and define
$$
    \cal{V}:= \Big\{ (v,w) \in \cal{P} \, :\,    |\cal{V}(v,w)|\ge \frac{6}{10} p^2n \Big\}. 
$$
\co{Everywhere in this section, by $\cal{V}^c$ and $\cal{M}^c$ we denote the complements of $\cal{V}$
and $\cal{M}$ in $\cal{P}$.}

\begin{lemma} \label{lem: whyV}
    Conditioned on $\Event_{typ}$, we have 
    \begin{align} \label{eq: VNotMImpliesNotV}
    (v,w) \in \cal{V} \cap \cal{M}^c \Rightarrow (f_v(w), f^{-1}_{f_v(w)}(v) ) \in \cal{V}^c. 
\end{align}
\end{lemma}
We refer to Figure~\ref{f:fig3} for a graphical illustration of the above lemma.

\medskip

\begin{figure}[h]
\caption{A graphical interpretation of the statement of
Lemma~\ref{lem: whyV}. Here, $(v,w)$ is an ordered pair from $\cal{V} \cap \cal{M}^c$, with $\tilde w:=f_v(w)$.
On the left hand side, we depict a part of the neighborhoods $N_\Gamma(v),N_{\tilde\Gamma}(v)$, showing $w$
and common neighbors $z_1,z_2,z_3,\dots$
of $v$ and $w$ in $\Gamma$ (which correspond to common neighbors $\tilde z_1,\tilde z_2,\tilde z_3,\dots$
of $v$ and $\tilde w$ in $\tilde\Gamma$). On the right hand side, we consider a part of the neighborhoods
$N_{\Gamma}(\tilde w),N_{\tilde\Gamma}(\tilde w)$. Note that the condition $(v,w)\in \cal{V}$
implies that $\tilde z_j=z_j$ for at least $\frac{6}{10}p^2n$ indices $j$.
If it were also true that $(\tilde w,f_{\tilde w}^{-1}(v))\in \cal{V}$ then we would necessarily have
$z_j'=\tilde z_j$ for at least $\frac{6}{10}p^2n$ indices $j$. But then $z_j$ would be a common neighbor
of $v,w,\tilde w$ in $\Gamma$ for at least $(1-o(1))\frac{2}{10}p^2n$ indices $j$, which is prohibited by $\Event_{typ}$.}

\centering  

\subfigure
{
\begin{tikzpicture}[every node/.style={rectangle,thick,draw}]
    \node (v) at (-5,0) {$v$};
    \node (w) at (-3,0) {$(w,\tilde w)$};
    \node (tw) at (5,0) {$\tilde w$};
    \node (pv) at (7,0) {$(f_{\tilde w}^{-1}(v),v)$};
    \node (z1) at (-6,2) {$(z_1,\tilde z_1)$};
    \node (z2) at (-4,2) {$(z_2,\tilde z_2)$};
    \node (z3) at (-2,2) {$(z_3,\tilde z_3)$};
    \node (dots1) at (0,2) {\dots};
    \node (z1t) at (3,2) {$(z_1',\tilde z_1)$};
    \node (z2t) at (5,2) {$(z_2',\tilde z_2)$};
    \node (z3t) at (7,2) {$(z_3',\tilde z_3)$};
    \node (dots2) at (9,2) {\dots};
    \draw[blue, very thick] (v) to (w);
    \draw[blue, very thick] (v) to (z1);
    \draw[blue, very thick] (v) to (z2);
    \draw[blue, very thick] (v) to (z3);
    \draw[blue, very thick] (v) to (dots1);
    \draw[blue, very thick] (w) to (z1);
    \draw[blue, very thick] (w) to (z2);
    \draw[blue, very thick] (w) to (z3);
    \draw[blue, very thick] (w) to (dots1);
    \draw[blue, very thick] (pv) to (tw);
    \draw[blue, very thick] (pv) to (z1t);
    \draw[blue, very thick] (pv) to (z2t);
    \draw[blue, very thick] (pv) to (z3t);
    \draw[blue, very thick] (pv) to (dots2);
    \draw[blue, very thick] (tw) to (z1t);
    \draw[blue, very thick] (tw) to (z2t);
    \draw[blue, very thick] (tw) to (z3t);
    \draw[blue, very thick] (tw) to (dots2);
\end{tikzpicture}
}
\label{f:fig3}
\end{figure}
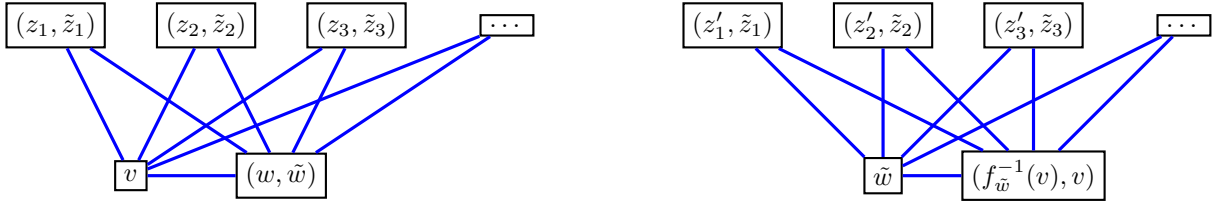

\begin{proof}
    For brevity, we denote $\tilde w:=f_v(w)$.
    It is suffice to prove  
    \begin{align} \label{eq: verifiableImplyMatch}
     (v,w), (\tilde{w}, f^{-1}_{\tilde w}(v)) \in \cal V \Rightarrow 
    (v,w), (w, v) \in \cal M,
    \end{align}
    because for any $(v,w) \in \mathcal{V} \cap \mathcal{M}^c$, $(\tilde w, f^{-1}_{\tilde w}(v) )$ 
    cannot be contained in $\mathcal{V}$ without contradicting \eqref{eq: verifiableImplyMatch}.
    
    Fix $(v,w), ( \tilde w, f_{\tilde w}^{-1}(v)) \in \mathcal{V}$.
    Since $(f_v(v),f_v(w))=(v,\tilde{w})$, the pair $\{v, \tilde{w} \} $ is an edge of $\tilde{\Gamma}$.
    Applying $f_{\tilde w}^{-1}$ to 
    $(v, \tilde{w})$ we get that $\{ f^{-1}_{\tilde{w}}(v), \tilde{w} \}$ is an edge of $\Gamma$.
    Further, observe that $f_v(\cal{V}(v,w)) = \cal{V}(v,w)$,
    and thus $\cal{V}(v,w)$ is contained in the set of common neighbors of $v$ and $\tilde{w}$ in $\tilde{\Gamma}$. Applying the same reasoning,  
    $ \cal{V}(\tilde{w}, f^{-1}_{\tilde{w}}(v))$  is contained in the set of common neighbors of $v$ and $\tilde{w}$ in $\tilde\Gamma$.
    Because the set of common neighbors of $\{v, \tilde{w}\}$ in $\tilde{\Gamma}$ is contained in $N_{\tilde \Gamma}(v)$ and 
    $f_v$ is a graph isomorphism from $N_{\Gamma}(v)$ to $N_{\tilde \Gamma}(v)$, 
    the number of common neighbors of $\{v, \tilde{w}\}$ in $\tilde{\Gamma}$
    is the same as the number of common neighbors of $\{v,w\}$ in $\Gamma$, 
    which is at most $( 1 + \log^{-\Cgl}n) p^2n$ (see condition \eqref{eq: EtypCommonNeighbors2}).
    
    Thus, if both $(v,w), (\tilde{w},f^{-1}_{\tilde{w}}(v)) \in \cal{V}$, then we get
    $$|\cal{V}(v,w) \cap \cal{V}(\tilde{w}, f^{-1}_{\tilde{w}}(v))| \ge 
    \frac{6}{10}p^2n + \frac{6}{10}p^2n - (1+\log^{-\Cgl}n)p^2n  \ge \frac{1.5}{10}p^2n.$$ 
    
    By \eqref{eq: EtypCommonNeighbors3pts}
    in $\Event_{typ}$, the number of common neighbors of any 3 distinct points in $\Gamma$ is at most $ \frac{p^2n}{10}$. 
    Thus, we conclude that 
    $ |\{v,w, \tilde{w}, f^{-1}_{\tilde{w}}(v)\}| \le 2$.  
    On the other hand, the set is at least of size $2$ since $w\neq v$. Further, $\tilde{w} \neq v$ since 
    $\{v,\tilde{w}\} \in E(\tilde{\Gamma})$, which leaves two options: either $v = f^{-1}_{\tilde w}(v)$ and $w=\tilde w$,
    or $w = \tilde{w} = f^{-1}_{\tilde w}(v)$. But
    the latter is not possible since
    $\{ \tilde{w}, f^{-1}_{\tilde{w}}(v)\} \in E(\Gamma)$. Therefore, $ v = f^{-1}_{\tilde w}(v)$ and $w=\tilde{w}$. 
\end{proof}

Next, we consider a partition $[n] = J_{\cal M} \sqcup J_{\cal M}^c$, where  
$$J_{\cal{M}} = \Big\{ v \in [n]\,:\,  |\{w:\; (v,w) \in {\cal{M}}^c \}| \le \frac{1}{10}pn \Big\}.$$
The set $J_{\cal M}$ can be viewed as a collection of ``good'' vertices of $\Gamma$ whose neighbors \co{are mostly mapped to vertices of $\tilde \Gamma$ with matching labels}. 
The next lemma asserts that, conditioned on $\Event_{typ}$ and assuming $p^2n=\omega(\log^2 n)$,
the set of points $\{w:\; (v, w) \in \cal{V}^c \}$ is much less than the set of points $\{w:\;  (v,w) \in \cal{M}^c \}$
for every $v\in J_{\cal M}$.

\begin{lemma} \label{lem: McToVc}
Conditioned on $\Event_{typ}$,
\begin{align*}
     \forall\; v \in J_{\cal M}
     \quad\mbox{we have}\;\;
          | \{w:\; (v, w) \in \cal{V}^c \} | \le \frac{ \log^2 n}{ p^2n }  |\{w:\;  (v,w) \in \cal{M}^c \}|.
\end{align*}
\end{lemma}
\begin{proof}
This is a consequence of \eqref{eq: EtypEdgetoBigPart} in $\Event_{typ}$. To see that, 
\co{we fix $v \in J_{\cal M}$ and}
let $J= \{ w\,:\, (v,w) \in \cal{M}^c \}$. Notice that $|J| \le \frac{1}{10}pn$ and 
\begin{align*}
\{ w\,:\, (v,w) \in \cal{V}^c \} 
&=
\Big\{ w\,:\;(v,w)\in\cal{P},\; |\{ z \,:\, (v,z) \in \cal{M} \mbox{ and } \{w,z\}\in E(\Gamma) \}|< \frac{6}{10}p^2n\Big\}\\
&\subset 
\Big\{ \co{w\snb{\Gamma}v} 
\, :\,  
    \big| \big\{ u\snb{\Gamma}v
    \,:\,u\notin J,\,
    \{ w,u \} \in E(\Gamma) \big\} \big| \le 0.999(pn-|J|)p \Big\}.
\end{align*}
\end{proof}

\begin{lemma} \label{lem: IVsturctureJ}
Assume that $p^2n = \omega\big( \log^2 n \big)$. 
Conditioned on $\Event_{typ}$ we have
\begin{align} \label{eq: likelyMatched}
|\{ (v,\tilde w) \in \tilde{\cal{P}} \,:\, v \in J_{\cal{M}},  \tilde w \in J_{\cal{M}}\, \mbox{ and } f_{v}^{-1}(\tilde w)\neq \tilde w \}| &\le 3\frac{\log^2 n}{p} |J_{\cal{M}}^c|.
\end{align}
In particular, if $J_{\cal M}^c = \emptyset$, then necessarily $\cal M = \cal P$, and thus $\Gamma = \tilde \Gamma$.

Moreover, for any parameter $\delta \in (0, 0.001]$, and assuming $n$  is sufficiently large, 
conditioned on the intersection $\Event_{typ}\cap \{|J_{\cal{M}}^c| \le \delta n\}$, 
\begin{align}
\label{eq: crossMatch}
|\{ (v,\tilde w) \in \tilde{\cal{P}}  \, :\, v \in J_{\cal{M}}, \tilde w \in J_{\cal{M}}^c \}|
=|\{ (v,\tilde w) \in \tilde{\cal{P}}  \, :\, v \in J_{\cal{M}}^c, \tilde w \in J_{\cal{M}} \}|
&\ge (1 - 2\delta )
|J_{\cal{M}}^c|pn.
\end{align}
\end{lemma}
Before presenting the proof, we would like to discuss the conditions of the lemma and why the statement {\it should}
be true.
The first part of the lemma deals with the statistics of the collection of ``bad'' pairs from $\tilde{\cal{P}}$
i.e pairs $(v,\tilde w)\in \tilde{\cal{P}}$ with $f_{v}^{-1}(\tilde w)\neq \tilde w$.
The first assertion of the lemma is essentially that, under the assumption $p^2n=\omega(\log^2 n)$, the number 
of ``bad'' pairs with both vertices in $J_{\cal{M}}$ is {\it much smaller}
than the number of ``bad'' pairs with at least one vertex in $J_{\cal{M}}^c$.
To see why this {\it should} be true, consider any vertex $v\in J_{\cal{M}}$.
At least $90-o(1)$ \co{percent} of pairs $(v,\tilde w)$ are ``good'', just by the definition of $J_{\cal{M}}$.
This, and condition \eqref{eq: EtypEdgetoBigPart} in the definition of $\Event_{typ}$,
implies that for a vast majority of $\tilde w\in V(N_{\tilde\Gamma}(v))\setminus\{v\}$ with $\tilde w\neq f_v^{-1}(\tilde w)$,
the pair $(v,f_v^{-1}(\tilde w))$ must belong to the set $\cal{V}$, that is, at least $\frac{6}{10}p^2n$
of the common neighbors of $v$ and $f_v^{-1}(\tilde w)$ in $N_{\Gamma}(v)$ must be mapped to the same vertices of $\tilde\Gamma$.
But, according to Lemma~\ref{lem: whyV}, for any such $\tilde w$
the pair $(\tilde w,f_{\tilde w}^{-1}(v))$ must belong to $\cal{V}^c$.
However, for any $\tilde w$ satisfying the additional assumption $\tilde w\in J_{\cal{M}}$,
the set $\{u:\; (\tilde w, u) \in \cal{V}^c \}$ has a small size, according to Lemma~\ref{lem: McToVc}.
Hence, the set $\{ (v,\tilde w) \in \tilde{\cal{P}} \,:\, v \in J_{\cal{M}},  \tilde w \in J_{\cal{M}}\, \mbox{ and } f_{v}^{-1}(\tilde w)\neq \tilde w \}$ must be small as well.

The second assertion of the lemma essentially tells that, assuming $J_{\cal{M}}^c$ is small,
for a ``typical'' vertex $v\in J_{\cal{M}}^c$ a vast
majority of the pairs $(v,\tilde w)\in\tilde{\cal{P}}$ are such that $\tilde w\in J_{\cal{M}}$.
Note that analogous statement for the graph $\Gamma$ is a simple consequence of standard concentration properties:
given any set $U$ (whether fixed or random, dependent on $\Gamma$), the number of edges in $\Gamma_{U^c}$
cannot be much larger than $\frac{1}{2}|U^c|^2p$ whereas the number of edges in $\Gamma_{U,U^c}$
should be roughly of order $|U|\,|U^c|p$. Hence, if $U^c$ is small then for most pairs $(v,w)\in\cal{P}$ with $v\in U^c$,
we have $w\in U$. For the graph $\tilde\Gamma$, we do not have apriori knowledge of such properties.
Instead, we obtain the required statement by relying on the definition of $J_{\cal{M}}^c$ and
by reusing statistics of common neighbors. A crucial observation here is that for most pairs $(v,w)\in\cal{P}$
with $v\in J_{\cal{M}}$ and $w\in J_{\cal{M}}^c$, we should also have $f_v(w)\in J_{\cal{M}}^c$,
which is ultimately a consequence of Lemma~\ref{lem: whyV}.
On the other hand, the number of pairs $\{(v,w)\in\cal{P}:\;v\in J_{\cal{M}},\;w\in J_{\cal{M}}^c\}$,
is roughly of order $|J_{\cal{M}}|\,|J_{\cal{M}}^c|\,p= (n-|J_{\cal{M}}^c|)\,|J_{\cal{M}}^c|\,p$.
\begin{proof}[Proof of Lemma~\ref{lem: IVsturctureJ}] 
For convenience, for any $I \subset [n]$ we define
$\cal{M}(I) := \{ (v,w) \in \cal{M}\,:\, v \in I \}$ and $ \cal{M}^c(I):= \{ (v,w) \in \cal{M}^c\,:\, v\in I\}$ (and similarly for $\cal{V}(I)$ and $\cal{V}^c(I)$).
By Lemma \ref{lem: McToVc}, conditioned on $\Event_{typ}$, 
\begin{align} \label{eq: IVVcMc}
|\cal{V}^c(J_{\cal{M}})| \le \frac{ \log^2 n}{ p^2n} |\cal{M}^c(J_{\cal{M}})|. 
\end{align}
Next, for each $(v,w) \in \cal{M}^c(J_{\cal{M}}) \cap \cal{V}(J_{\cal{M}})$ with $ f_v(w)=\tilde{w} \in J_{\cal{M}}$, applying \eqref{eq: VNotMImpliesNotV} we get 
$ (\tilde{w}, f^{-1}_{\tilde{w}}(v)) \in \cal{V}^c(J_{\cal{M}})$. 
\co{Further, we claim that
the map $(v,w) \mapsto (\tilde w, f_{\tilde w}^{-1}(v))$ where $\tilde w := f_v(w)$, is a bijection on $\cal P$. Indeed, applying the correspondence twice to the pair $(v,w)$, we get 
$$
(v,w) \rightarrow (\tilde w, f_{\tilde w}^{-1}(v)) \rightarrow ( \tilde v, f_{\tilde v}^{-1}( \tilde w)),
$$
where 
$$ 
\tilde v = f_{\tilde w}( f_{\tilde w}^{-1}(v))=v,
$$
which implies $( \tilde v, f_{\tilde v}^{-1}(\tilde w)) = (v,w)$. The claim follows.}
Now, using injectivity of this map and \eqref{eq: IVVcMc} we conclude that 
\begin{align*} 
|\{ (v,w) \in \cal{M}^c(J_{\cal{M}}) \cap \cal{V}(J_{\cal{M}}) \, :\, f_v(w) \in J_{\cal{M}} \}|
\le | \cal{V}^c(J_{\cal{M}}) | \le \frac{\log^2 n}{p^2n} 
| \cal{M}^c( J_{\cal{M}})|.
\end{align*}
Applying \eqref{eq: IVVcMc} again we get
\begin{align} \label{eq: notMuchBadEdgeFromJgdtoJgd}
|\{ (v,w) \in \cal{M}^c(J_{\cal{M}}) \, :\, f_v(w) \in J_{\cal{M}} \}|
\le 2\frac{\log^2 n}{p^2n} | \cal{M}^c( J_{\cal{M}})|.
\end{align}
Note that since we assumed that 
$p^2n = \omega\big(  \log^2 n \big)$,
\eqref{eq: notMuchBadEdgeFromJgdtoJgd} implies that if $| \cal{M}^c(J_{\cal{M}})| > 0$, then $J_{\cal{M}}^c\neq \emptyset$.

Next, we will establish a quantitative relation between $|\cal{M}^c(J_{\cal{M}})|$ and $|J^c_{\cal{M}}|$,
which can be done by examining the set $\{ (\tilde{v}, \tilde{w}) \in \tilde{\cal{P}} \,:\, \tilde{w} \in J_{\cal{M}}^c \} $. First of all, by \eqref{eq: EtypDegree} in $\Event_{typ}$, 
\begin{align} \label{eq: IVedgesWithEndPointsinJbad}
    |\{ (\tilde{v}, \tilde{w}) \in \tilde{\cal{P}} \,:\, \tilde{w} \in J_{\cal{M}}^c \}| = \sum_{ \tilde{w} \in J^c_{\cal{M}} }
    \deg_{\tilde{\Gamma}}(\tilde{w})
    = \sum_{ \tilde{w} \in J^c_{\cal{M}}} \deg_\Gamma (\tilde{w})
    = (1 +o(1))|J_{\cal{M}}^c|pn.
\end{align}
Notice that the map $(v,w) \mapsto (v,f_v(w))$ from $\cal{P}$ 
to $\tilde{\cal{P}}$ is a bijection, whose inverse is $(\tilde{v}, \tilde{w}) \mapsto (\tilde{v}, f_{\tilde{v}}^{-1}(\tilde{w}))$. Therefore, we have
\begin{align*}
    |\{ (\tilde{v}, \tilde{w}) \in \tilde{\cal{P}} \,:\, \tilde{w} \in J_{\cal{M}}^c \}| &=  |\{ (v,w) \in \cal{P}\,:\, f_v(w) \in J^c_{\cal M} \}| \\
    &\ge |\{ (v,w) \in \cal{M}^c(J_{\cal M}) \,:\, f_v(w) \in J^c_{\cal M} \}|
    \ge  \Big(1-2\frac{\log^2 n}{p^2n} \Big) | \cal{M}^c(J_{\cal M})|,
\end{align*}
by \eqref{eq: notMuchBadEdgeFromJgdtoJgd}.
Together with the condition
$p^2n = \omega\big(\log^2 n \big)$ and \eqref{eq: IVedgesWithEndPointsinJbad}, this gives
\begin{align} \label{eq: deltaLambdaRelation}
| \cal{M}^c(J_{\cal M})| \le (1 + o(1)) |J^c_{\cal{M}}|\,pn.
\end{align}
Note that the above bound implies that if $|J^c_{\cal{M}}|=0$,
then $ |\cal{M}^c|=0$, and hence $\Gamma = \tilde{\Gamma}$.
Applying \eqref{eq: deltaLambdaRelation} and \eqref{eq: notMuchBadEdgeFromJgdtoJgd} we obtain the first statement in the lemma. 

It remains to prove the second statement, where we rely on the fact that $|J^c_{\cal M}|\le \delta n $.
Recall that 
$\Gamma_{J_{\cal M},J^c_{\cal M} }$ is the bipartite subgraph of $\Gamma$ on the vertex set $J_{\cal M} \sqcup J^c_{\cal M}$ where we keep only edges connecting $J_{\cal M}$ to $J^c_{\cal M}$.
By \eqref{eq: EtypEdgeToComplement} in $\Event_{typ}$ 
and the assumption that $ |J_{\cal M}^c| \le \delta n$, 
$$
|E_{\Gamma}(J_{\cal{M}}, J_{\cal{M}}^c)|  \ge \Big(1 - \frac{ \log n }{\sqrt{pn}}\Big) p |J_{\cal{M}}| |J_{\cal{M}}^c|
\ge 
\Big(1 - \frac{ \log n }{\sqrt{pn}}\Big) (1- \delta ) pn|J_{\cal{M}}^c|. 
$$
Next, 
\begin{align*}
   |\{ (v, \tilde w) \in \tilde{\cal P} 
    \,:\, v \in J_{\cal M}\mbox{, and } \tilde w \in J_{\cal M}^c \}|
= & |\{ (v,w) \in \cal{P} \,:\, v\in J_{\cal{M}}\mbox{, and } f_v(w) \in J_{\cal{M}}^c\}|  \\
\ge & |\{ (v,w) \in \cal{P}  \, :\, v \in J_{\cal{M}}, w \in J_{\cal{M}}^c, \mbox{ and } f_v(w) \in J_{\cal{M}}^c \}| \\
= & |E_{\Gamma}(J_{\cal{M}}, J_{\cal{M}}^c)| - |\{ (v,w) \in \cal{P}  \, :\, v \in J_{\cal M}, w \in J_{\cal{M}}^c, \mbox{ and } f_v(w) \in J_{\cal M} \}| \\
= & |E_{\Gamma}(J_{\cal{M}}, J_{\cal{M}}^c)| - |\{ (v,w) \in \cal{M}^c(J_{\cal M})\, :\, w \in J_{\cal{M}}^c\, \mbox{ and } f_v(w) \in J_{\cal M} \}|   \\
\ge & |E_{\Gamma}(J_{\cal{M}}, J_{\cal{M}}^c)|  
- 2\frac{ \log^2 n}{p^2n} | \cal{M}^c(J_{\cal M})|, 
\end{align*}
\co{where the last inequality follows by \eqref{eq: notMuchBadEdgeFromJgdtoJgd}.}
Together with \eqref{eq: deltaLambdaRelation},  this implies
\begin{align*}
 |\{ (v, \tilde w) \in \tilde{\cal P} 
    \,:\, v \in J_{\cal M}\mbox{, and } \tilde w \in J_{\cal M}^c \}|
   \ge &
\bigg[\Big(1 - \frac{ \log n }{\sqrt{pn}}\Big) (1- \delta ) -  (1+o(1))\cdot 2\frac{\log^2 n}{p^2n} \bigg] pn|J_{\cal{M}}^c| \\
 \ge &
(1-  2\delta ) |J_{\cal{M}}^c|pn \co{,}
\end{align*}
\co{where the last inequality holds for $n$ large enough.}
The second statement follows. 
\end{proof}

Next, we want to develop a procedure similar to the one used in {Step I}: more precisely,
we want to show that the assumption that $0< |J_{\cal{M}}^c| \le 
    \delta n$ necessarily induces a large number of {\it constraints} on the edges of $\Gamma$,
    which can be satisfied only with a very small probability.
    In the next lemma, we show that under the mentioned assumption, there is a vertex $v \in J_{\cal{M}}^c$
    with many constraints coming from the neighborhood $N_{{\Gamma}}(v)$.
\begin{lemma} \label{lem: IVEgood}
    Let $\delta\le 0.001$ be any fixed positive constant. Assuming that $p^2n=\omega(\log^2 n)$, for every sufficiently large $n$,
    conditioned on $\Event_{typ}\cap\{0< |J_{\cal{M}}^c| \le 
    \delta n\}$, there is $v \in J_{\cal{M}}^c$ such that 
    the set
    $$ I_v:= \{ \co{ w \snb{\Gamma} v } 
    \,:\,  f_v(w)\neq w,\, f_v(w) \in J_{\cal{M}} \}$$
    satisfies $|I_v|\ge \frac{1}{11}pn$,
    and such that the set 
    \begin{align} \label{eq: Egood}
        E_v := \Big\{ \{w,w'\} \in E(\Gamma_{I_v}) \, :\, f_v(w) \neq w, \, f_v(w')\neq w', \mbox{ and } 
         \{f_v(w), f_v(w')\} \in E(\Gamma) \Big\}
    \end{align}
    satisfies 
    \begin{align*}
         | E_v | \ge 
        (1-o(1) ) |E( \Gamma_{I_v} )|
        = (1- o(1) )  \frac{ |I_v|^2 }{2}p  
        \ge  0.001 n^2p^3.
    \end{align*}
    \end{lemma}
    \begin{proof}
    Let $$I:=  \{ v \in J_{\cal{M}}^c\, :\, |\{ (v,\tilde{w}) \in \tilde{\cal{P}}\,: \, \tilde{w} \in  J_{\cal{M}}^c \}|
    \le 5\delta pn \}.$$ 
    By \eqref{eq: EtypDegree} from $\Event_{typ}$, every vertex of $\tilde{\Gamma}$ is at most $(1 + \log^{-\Cgl}n)pn$. 
    In view of this degree bound,  
    $$ |\{ (v, \tilde{w}) \in \tilde{\cal{P}} \, :\, v \in J_{\cal{M}}^c,\;\tilde w\in[n] \}| = \sum_{ v \in J_{\cal{M}}^c} \mbox{deg}_{\tilde{\Gamma}}(v) \le (1 + o(1))pn|J_{\cal{M}}^c|. $$
    Together with \eqref{eq: crossMatch}, this gives
    \begin{align*}
        |J_{\cal M}^c\setminus I| \cdot 5\delta pn \le 
        \sum_{v \in J_{\cal M}^c } |\{ (v,\tilde w) \in \tilde{ \cal P} \,: \, \tilde{w} \in J_{\cal M}^c \}|   
        \le (1 + o(1)) pn |J^c_{\cal M}| - (1-2\delta ) |J^c_{\cal M}|pn
        = (2\delta + o(1)) |J^c_{\cal M}|pn,
    \end{align*} 
    and therefore
    \begin{align} \label{eq: IComparableToJMc}
     |I| > \frac{1}{2} |J_{\cal{M}}^c|.
    \end{align}
    In particular, the set $I$ is non-empty. Take any $v\in I$.
    The set 
    $$ I_v= \{ \co{w \snb{\Gamma} v } \,:\,  f_v(w)\neq w,\, f_v(w) \in J_{\cal{M}} \}$$
    satisfies 
    \begin{align*}
        |I_v| \ge& |\{ w \snb{\Gamma} v 
        \,:\,  f_v(w)\neq w \}| - 
            |\{ w \snb{\Gamma} v
            \, :\, f_v(w)\in J_{\cal{M}}^c \}| \\
        \ge & \frac{1}{10}pn - 5 \delta pn \ge \frac{1}{11}pn, 
    \end{align*}
    with the assumption that $\delta \le 0.001$.
    By \eqref{eq: EtypSubgraphsEdgeCounts},
    \begin{equation}\label{1-398710981709}
        |E(\Gamma_{I_v})| \geq ( 1 - o(1)) { |I_v| \choose 2} p 
        = ( 1- o(1)) \frac{|I_v|^2}{2}p \ge (1 - o(1)) \frac{1}{242} n^2p^3.
    \end{equation}
    Next, we want to argue that there exists $v \in I$ so that 
    for a ``typical'' pair $w,w' \in I_v$ with $ \{w,w'\} \in E(\Gamma)$, we also have
    $ \{f_v(w),\, f_v(w') \} \in E(\Gamma)$. \co{Note that, if $\{f_v(w), f_v(w')\} \notin E(\Gamma)$, then 
    $\{f_v(w), f_v(w')\}$ is an unordered pair satisfying $f_v(w), f_v(w') \in J_{\cal M}$, \co{$f_v(w)\snb{\tilde \Gamma}
    f_v(w')$}, and $f_v(w) \stackrel{\Gamma}{\nsim}f_v(w')$. Now we consider the set 
    $$
        E' := \big\{ \{ \tilde{w}, \tilde{w}'\} \in E(\tilde{\Gamma})\, :\, 
        \tilde{w} , \tilde{w}' \in J_{\cal{M}} ,\, 
     \{ \tilde{w}, \tilde{w}'\} \notin E(\Gamma) \big\}. 
    $$
    Then, by the definition of $E'$ and $I_v$ we have 
    \begin{align} \label{eq: E'00}
        & \big\{ (v, \{w, w'\}) \,: \, v\in I,  w,w' \in I_v, w \snb{\Gamma} w', \mbox{ and } f_v(w) 
        \stackrel{\Gamma}{\nsim} f_v(w') \big\}   \nonumber \\ 
     = &  \big\{ (v, \{w, w'\}) \,: \, v\in I, w,w' \in I_v,  \mbox{ and }
      \{f_v(w), f_v(w')\} \in E' 
     \big\}  \nonumber \\
       \subset  &  \big\{ (v, \{w, w'\}) \,: \, w \snb{\Gamma} v,w' \snb{\Gamma} v,
      \mbox{ and } \{f_v(w), f_v(w')\} \in E' 
     \big\}.
    \end{align}
     The cardinality of the set in \eqref{eq: E'00} can be expressed as
   \begin{align} \label{eq: E'01}
   \big| \big\{ (v, \{w, w'\}) &\,: \, w \snb{\Gamma} v,w' \snb{\Gamma} v,
      \mbox{ and } \{f_v(w), f_v(w')\} \in E' 
     \big\} \big| \nonumber \\
    =  &  \big|\big\{ (v, \{\tilde w, \tilde w'\}) \,: \,  
     \tilde w \snb{\tilde \Gamma} v, \tilde w' \snb{\tilde \Gamma} v,  \mbox{ and }
      \{\tilde w, \tilde w'\} \in E' 
     \big\} \big|. 
   \end{align}    
   
    For any $(\tilde w, \tilde w') \in E(\tilde \Gamma)$, 
    the conditions $\tilde w \snb{\tilde \Gamma}  \tilde w '$ and $\tilde w \stackrel{\Gamma}{\nsim} \tilde w'$ imply that $f^{-1}_{\tilde w}(\tilde w') \neq \tilde w'$ and $f^{-1}_{\tilde w'}(\tilde w) \neq \tilde w$ (we remark that
    the converse of this assertion is not necessarily true in general
    ).} 
    Together with \eqref{eq: likelyMatched},
    \begin{align} \label{eq: E'size}
    |E'|\le  &
    \big| \big\{ \{ \tilde{w}, \tilde{w}'\} \in E(\tilde{\Gamma}) \, :\,  
        \tilde{w}, \tilde{w}' \in J_{\cal{M}},\, 
    f^{-1}_{\tilde{w}}( \tilde{w}')\neq \tilde{w}',\,
        \mbox{ and }  f^{-1}_{\tilde{w}'}(\tilde{w})\neq \tilde{w}\big\} \big|  
    \le 3 \frac{ \log^2 n}{ p} |J_{\cal{M}}^c|.
    \end{align}
    For every pair $\{ \tilde{w}, \tilde{w}'\} \in E(\tilde{\Gamma})$, \co{we claim} 
    there are at most $(1+o(1))p^2n$ common neighbors of $\tilde w$ and $\tilde w'$ in $\tilde\Gamma$. 
    Indeed, since $f_{\tilde{w}}$ is an isomorphism of $N_{\Gamma}(\tilde{w})$ and 
    $N_{\tilde{\Gamma}}(\tilde{w})$, the number of common neighbors of 
    $\{\tilde{w},\tilde{w}'\}$ in $\tilde{\Gamma}$ is the same as the number of common neighbors 
    of $\{ \tilde{w}, f^{-1}_{\tilde{w}}(\tilde{w}')\}$ in $\Gamma$. Since we conditioned on 
    $\Event_{typ}$, the latter is bounded above by $(1 + o(1)) p^2n$, by \eqref{eq: EtypCommonNeighbors2}. 
    \co{ Hence, combining \eqref{eq: E'00}, \eqref{eq: E'01}, and \eqref{eq: E'size} we get 
    \begin{align*}
      \big| \big\{ (v, \{w, w'&\}) \,: \, v\in I,  w,w' \in I_v, w \snb{\Gamma} w', \mbox{ and } f_v(w) \stackrel{\Gamma}{\nsim} f_v(w') \big\} \big| \\
        \le &
        \big|\big\{ (v, \{\tilde w, \tilde w'\}) \,: \,  
     \tilde w \snb{\tilde \Gamma} v, \tilde w' \snb{\tilde \Gamma} v,  \mbox{ and }
      \{\tilde w, \tilde w'\} \in E' 
     \big\} \big| 
        \le  3\frac{ \log^2 n}{p} |J_{\cal{M}}^c|(1+o_n(1))p^2n=o(|I|n^2p^3),
    \end{align*}
    where the last inequality holds due to \eqref{eq: IComparableToJMc} and the assumption $p^2n = \omega(\log^2(n))$. } 
    On the other hand, by \eqref{1-398710981709},
    $$
    \big| \big\{ (v,\{w,w'\}) \,:\, v\in I, w,w' \in I_v, \{w,w'\} \in E(\Gamma)\big\}\big|
    =\Omega(|I|n^2p^3).
    $$
    Consequently, there exists $v \in I$ with 
    \begin{align*}
    |\{ \{w,w'\} \,:\, w,w' \in I_v, \{w,w'\} \in E({\Gamma}) \mbox{ and } \{f_v(w),f_v(w')\} \in E(\Gamma)\}| \ge (1-o(1)) |E(\Gamma_{I_v})|. 
    \end{align*}
    \co{The} lemma follows since $ |I_v| \ge \frac{1}{11}pn$.
    \end{proof}

    \begin{lemma} \label{lem: IVEtooGood}
    Let $\delta\le 0.001$ be any fixed positive constant. Assuming that $p^2n=\omega(\log^2 n)$, and that $n$ is sufficiently large,
    condition on $\Event_{typ}\cap\{0< |J_{\cal{M}}^c| \le 
    \delta n\}$. Let $v$ be a \co{vertex} from the statement of Lemma \ref{lem: IVEgood}, and $E_v$ the set
    defined in that lemma. Then    
        $$
            \big| E_{v} \cup f_v(E_v) \big| \ge \frac{3}{2} |E_{v}|,
        $$
        where $f_v(E_v) := \big\{ \{f_v(w),\, f_v(w')\} \,:\, \{w,w'\} \in E_v\big\}$.
    \end{lemma}
\begin{proof}
    The proof is by contradiction. Assume that $\big| E_{v} \cup f_v(E_v) \big| < \frac{3}{2} |E_{v}|$.
    There are two steps in the argument. First, setting \co{ 
    $$ I_v= \{ w \snb{\Gamma} v \,:\,  f_v(w)\neq w,\, f_v(w) \in J_{\cal{M}} \}$$
    as in the previous lemma,}
    we derive a lower bound on the size of 
    $I_v \cap f_v(I_v)$ in terms of the size of $E_{v} \cap f_v(E_v)$, which is at least $\frac{1}{2}|E_v|$.
    Second, we show that if 
    $I_v \cap f_v(I_v)$ is sufficiently large 
    then there exist two disjoint induced subgraphs
    of $N_{\Gamma}(v)$ with vertex set cardinalities proportional to $pn$, which are close to being isomorphic. However, such scenario is not observable in a typical realization of $\Gamma$!
    
    To start, \co{ suppose $w \in I_v$ and there exists $w'$ such that $\{w,w'\} \in f_v(E_v)$. By definition of $E_v$, there exist $u,u' \in I_v$ satisfying $f_v(u)=w$ and $f_v(u')=w'$, which implies $w \in f_v(I_v)$. Hence, 
    }
    \begin{align*}
        |I_v \cap f_v(I_v)|
    \ge& |\{ w \in I_v\,:\, \exists w' \mbox{ such that } \{w,w'\} \in f_v(E_v) \}| \\
    \ge& |\{ w \in I_v\,:\, \exists w' \mbox{ such that } \{w,w'\} \in E_v\cap f_v(E_v) \}|. 
    \end{align*} 
    For each \co{$w\snb{\Gamma}v$}, 
    the degree of $w$ in $N_{{\Gamma}}(v)$ is at most $(1+o(1))p^2n$ 
    by \eqref{eq: EtypDegree} from $\Event_{typ}$. 
    Thus, each vertex in $I_v$ is an endpoint of at most 
    $(1+o(1))p^2n$ edges in $E_{v}\cap f_v(E_v)$. 
    Therefore, 
    $$
    |I_v \cap f_v(I_v)|\geq
    |\{ w \in I_v\,:\, \exists w' \mbox{ such that } \{w,w'\} \in E_v\cap f_v(E_v)  \}| \ge \frac{ |E_v\cap f_v(E_v)|}{(1+o(1))p^2n}
    \ge cpn
    $$
    for some universal constant $c>0$, where the last inequality follows 
    from Lemma~\ref{lem: IVEgood}.
    
    \co{For} the second part of the proof, we construct a bijection with the required properties.
    Observe that the above bound implies that there exists a pair of subsets 
    $I_1',\, I_2' \subset I_v$ such that  $|I_1'|=|I_2'|\ge cpn$ and such that
    $f_v$ restricted to $I_1'$ is a bijection onto $I_2'$. 
    
    \medskip
    {\bf Claim:} Given the sets $I_1',I_2'$ with above properties, there exist $I_1''\subset I_1'$ and $I_2'' \subset I_2'$ such that 
    $|I_1''|=|I_2''| \ge \frac{c}{4}pn$, 
    $I_1''\cap I_2'' = \emptyset$, and such that $f_v$ restricted to $I_1''$ is a bijection onto $I_2''$.
    
    We will verify the claim by relying on the condition 
    $f_v(w) \neq w$ for every $w \in I_v$. We will apply an inductive argument for \co{the} construction. 
    \co{For the base case}, there exist singletons $\{w_1\} \subset I_1'$ and $\{u_1\}\subset I_2'$  such that 
    $f_v(w_1)=u_1$ and $w_1\neq u_1$. 
    At $k$--th step of the induction ($k<|I_1'|/4$), we assume that
    we have found $k$--subsets $\{w_1,w_2, \dots, w_k\} \subset I_1'$  and $\{u_1,\dots, u_k\} \subset I_2'$ satisfying
    \begin{align*}
        \{w_1, \dots, w_k\} \cap \{u_1,\dots, u_k\} =& \emptyset, 
        \, \forall\; i \in [k], f_v(w_i)=u_i.
    \end{align*}
    Observe that, the cardinality of \co{both} of the sets
    \begin{align*}
    &\big\{(w,f_v(w)) \,:\, w \in I_1' \backslash \{w_1,\dots, w_k\},\, 
        \mbox{ and } f_v(w) \notin \{w_1,\dots, w_k\} \big\}\quad \mbox{ and }\\
    &\big\{(f_v^{-1}(u),u) \,:\, u \in I_2' \backslash \{u_1,\dots, u_k\},\, 
        \mbox{ and } f^{-1}_v(u) \notin \{u_1,\dots, u_k\} \big\}
    \end{align*}
    is at least $|I_1'| - 2k$. Since both sets are contained in $\{ (w,u)\,:\, w \in I_1', u=f_v(w)\}$,
    which has size $|I_1'|$, the intersection of the above two sets has size at least 
    $$  |I_1'|- 2k + |I_1'|-2k - |I_1'| = |I_1'|-4k > 0.$$
    Hence, there exists a pair $\{w,u\}$ with $w \in I_1'$ and $u\in I_2'$ satisfying 
    $\{w, u\} \cap \{w_1,\dots, w_k, u_1, \dots, u_k\} = \emptyset $ and 
    $ f_v(w)=u$. Notice that the $f_v(w)=u$ implies that $w\neq u$. Now,  
    $\{w_1,\dots, w_k,w_{k+1}=w\} \subset I_1'$ and $\{u_1,\dots, u_k,u_{k+1}=u\} \subset I_2'$ 
    are two disjoint subsets satisfying $f_v(w_i) = u_i$ for $i \in [k+1]$. Hence, the claim is proved by induction.

    Finally, we need to show that the existence of the sets $I_1''$ and $I_2''$ constructed above contradicts
    the description of the event $\Event_{typ}$.
    Observe that for any pair $\{w,w'\} \in E(\Gamma_{I_1''}) \cap E_v$, we have $\{f_v(w), f_v(w')\} \in E(\Gamma_{I_2''})$. 
    Since $I_1'' \subset V(N_{\Gamma}(v))\backslash \{v\}$, and
    $|I_1''|\ge \frac{c}{4}pn$, the condition
    \eqref{eq: EtypSubgraphsEdgeCounts} from $\Event_{typ}$ implies that 
    \begin{align} \label{eq: EGammaI1''}
    |E(\Gamma_{I_1''})| = (1 + o(1))\frac{1}{2}|I_1''|^2p \ge (1+o_n(1)) \frac{c}{8}n^2p^3. 
    \end{align}
    On the other hand, \co{ $I_1'' \subset I_v$ implies $E(\Gamma_{I_1''})\backslash E_v \subset E(\Gamma_{I_v}) \backslash E_v$. By Lemma \ref{lem: IVEgood}, the set $E_v \subset E_{\Gamma_{I_v}}$ satisfies 
    $ |E_v| \ge (1-o_n(1)) |E_{\Gamma_{I_v}}|$. Therefore, 
   \begin{align*} 
   |E(\Gamma_{I_1''}) \backslash E_v| \le  
    |E(\Gamma_{I_v}) \backslash E_v| = o_n( |E(\Gamma_{I_v})|). 
   \end{align*} 
    By the conditions \eqref{eq: EtypSubgraphsEdgeCounts} and \eqref{eq: EtypDegree} from $\Event_{typ}$, 
    $$
     |E(\Gamma_{I_v})| \le (1+o_n(1))\frac{1}{2}|I_v|^2p \le (1+o_n(1))\frac{1}{2}|V(N_{\Gamma}(v))|^2p 
     \le (1+o_n(1)) \frac{1}{2} n^2p^3
    $$
   where last equality holds in view of \eqref{eq: EtypDegree}. 
   Combined with \eqref{eq: EGammaI1''}, we obtain, 
   $$
    |E(\Gamma_{I_1''}) \backslash E_v| = o_n( E( \Gamma_{I_1''})).     
   $$
   
    We conclude that 
    \begin{align*}
        |\{ \{w,w'\} \in E(\Gamma_{I_1''}) \, :\, \{f_v(w), f_v(w')\} \in E(\Gamma_{I_2''})\}| 
        &\ge \big| \big\{ \{ w,w' \} \in E(\Gamma_{I_1''}) \cap E_v\big\}\big| \\
        &= ( 1- o(1)) |E(\Gamma_{I_1''})|
         \ge ( 1- o(1)) \frac{|I_1''|^2}{2} p.
    \end{align*}
    }
    But the last relation violates condition \eqref{eq: EtypLocalSubgraphMap} from $\Event_{typ}$, 
    which leads to contradiction. 
   \end{proof}

\bigskip

Now, everything is ready to complete the proof of the main result.
\begin{proof}[Proof of Theorem~\ref{309847109847}]
Fix any positive constant $\delta \le 0.001$.
We will assume that $p^2n=\omega(\log^{3+4\Cgl}n)$ and $2ep\leq \frac{\delta}{20C_3
}$, \co{where the constant $C_3$ is taken from Proposition~\ref{0983701984721-94871-}.}
Recall that $\Prob(\Event_{typ})\geq 1-n^{-\omega(1)}$, by Proposition~\ref{prop: GammaTypical}.
Next, applying the results we got in Steps I, II, and III, we obtain an estimate on the probability of $ \Event_3\big( \frac{1}{20}\delta\big) \cap \Event_{typ}$:
\begin{align*}
    \Prob\Big( \Event_3\Big( \frac{1}{20}\delta\Big) \cap \Event_{typ} \Big) \ge & 
    \Prob \Big( \Event_2\Big( \frac{1}{20C_3
    }\delta \Big) \cap \Event_{typ} \Big) - n^{-\omega(1)} &
    & \mbox{ by Proposition } \ref{0983701984721-94871-} \\
    \ge& \Prob \bigg( \Event_1 \bigg( \Big( \frac{1}{20C_3
    C_2
    } 
        \delta \Big)^3 \bigg) \cap \Event_{typ} \bigg) - n^{-\omega(1)} & & \mbox{ by Proposition } \ref{138413-4982174-29847} \\
    \ge& 1 - n^{-\omega(1)}. & & \mbox{ by Proposition }\ref{19741094720987098},
\end{align*}
\co{where the constant $C_2$ comes from Proposition~\ref{138413-4982174-29847}.}
Further, observe that conditioned on the event $\Event_3( \frac{1}{20}\delta)\cap \Event_{typ} $, 
$$
    \frac{1}{10} pn |J_{\cal M}^c | \le  \sum_{v\in J_{\cal M}^c} |\{w:\; (v,w) \in \cal{M}^c\}| \le (1+o(1))\frac{1}{20}\delta n^2p
    \leq \frac{1}{10}\delta n^2p,
$$
and thus 
$
    |J_{\cal M}^c | \le \delta n.
$
Now, we estimate $\Prob\{ \Gamma \neq \tilde{\Gamma}\}$ by intersecting the event $\{\Gamma \neq \tilde{\Gamma}\}$
with $\Event_{typ}$ and $\Event_3 \big( \frac{1}{20} \delta  \big)$:
\begin{align*}
    \Prob\{ \Gamma \neq \tilde{\Gamma} \} &\le n^{-\omega(1)}
    + \Prob \Big( \{\Gamma \neq \tilde{\Gamma}\} \cap  \Event_{typ} \cap \Event_3 \Big( \frac{1}{20}\delta \Big) \Big) \\
    &\le  n^{-\omega(1)}
    + \Prob \Big( \{ \Gamma \neq \tilde{\Gamma}\} \cap \Event_{typ} \cap \{ |J_{\cal{M}}^c|\le \delta n\} \Big) \\
    &\le n^{-\omega(1)}
    + \Prob \Big( \{0<|J_{\cal{M}}^c|\le \delta n \} \cap \Event_{typ}  \Big),
\end{align*}
where we used that $ \Event_{typ} \cap \{\Gamma \neq \tilde{\Gamma}\}$ is contained in  $\{|J_{\cal{M}}^c|>0\}$,
in view of Lemma \ref{lem: IVsturctureJ}.

At this stage we apply a counting argument using data structures similar
to the ones employed in the proof of Proposition~\ref{19741094720987098}.
Let $\cal{D}'$ be the collection of all data structures of the form
$$
    (v, \quad \bar{I}_v , \quad \bar{f}_v : \bar{I}_v \mapsto [n],\quad \bar{E}_v),
$$
where
\begin{itemize}
\item $v$ is any index in $[n]$,
\item $ \bar{I}_v$ is any subset of $[n]$
with  $ \frac{1}{11}pn \le |\bar{I}_v| \le 2pn$, 
\item $\bar{E}_v$ is any subset of unordered pairs of elements of $\bar{I}_v$
with $|\bar{E}_v| \geq r := \lfloor 0.001n^2p^3 \rfloor$, 
\item $\bar f_v$ is any injective mapping from $\bar{I}_v$
into $[n]$ such that $|\bar{E}_v \cup  \bar{f}_v(\bar{E}_v) | \ge \frac{3}{2} |\bar{E}_v|. $
\end{itemize}
Elements of $\cal{D}'$ are meant to give a partial description
of pairs of $1$--neighborhoods of $\Gamma$ and $\tilde\Gamma$:
given $D'=(v, \bar{I}_v ,\bar{f}_v,\bar{E}_v) \in \cal{D}'$,
$v$ is the center of the neighborhoods, 
$\bar{I}_v$ and $\bar{E}_v$ are possible realizations
of random sets $I_v$ and $E_v$ defined in Lemma \ref{lem: IVEgood} for $v$,
and $\bar{f}_v$ would correspond to a restriction of the isomorphism $f_v$ to $I_v$.
Any given data structure $D'$
may or may not give a correct description of $\Gamma,\tilde\Gamma$, depending on a realization of the graphs.

\HL{
We say that $\Gamma$ is $D'$--compatible for a data structure
$D'=(v, \bar{I}_v ,\bar{f}_v,\bar{E}_v) \in \cal{D}'$
if all of the following hold:}
\begin{itemize}
    \item $I_v=\bar{I}_v$,
    \item $ f_v(w)=\bar{f}_v(w)$ for all $w \in \bar{I}_v$,
    \item $E_v= \bar{E}_v$, 
\end{itemize}
where the random sets $I_v$ and $E_v$ were introduced in Lemma \ref{lem: IVEgood}. 

By Lemma \ref{lem: IVEgood} and Lemma \ref{lem: IVEtooGood}, for any realization
of $\Gamma$ from the event $\{0<|J_{\cal{M}}^c|\le \delta n\} \cap \Event_{typ}$, \co{there}
exists $D' \in \cal{D}' $ \HL{so that $\Gamma$ is $D'$--compatible.} So to complete the proof of the theorem, 
it remains to show that
\begin{align} \label{eq: IVDataStructure}
    \Prob \Big\{ \exists D' \in \cal{D}' \mbox{ such that } \HL{ \Gamma \mbox{ is $D'$--compatible}}\Big\}
= & n^{-\omega(1)}. 
\end{align}
For any $r\leq k\leq 2n^2p^2$, denote by $\cal{D}'_k$ the subset of structures of $\cal{D}'$
with the set $\bar{E}_v$ of size $k$.
A rough estimate of $|\cal{D}'_k|$ gives 
$$
    |\cal{D}'_k| \le n^{1+4pn} \Big( \frac{ e\cdot 2n^2p^2}{k} \Big)^k=
    \Big( \frac{ e\cdot 2n^2p^2}{k} \Big)^{(1+o(1))k}
    \leq \Big( \frac{ 2000e}{p} \Big)^{(1+o(1))k},
$$
where we used that $\log n=o(p^2n)$, and
where the dominating factor is a bound 
$${\lfloor |\bar{I}_v|^2/2\rfloor \choose k} \le 
{ \lfloor 2n^2p^2\rfloor  \choose k} \le \Big( \frac{ e\cdot 2n^2p^2}{k} \Big)^k$$
on the number of choices of $\bar{E}_v$ after $\bar{I}_v$ has been chosen.
On the other hand, for any $D' \in \cal{D}_k'$, 
$$
    \Prob \{ \HL{ \mbox{$\Gamma$ is $D'$--compatible}}  \}  \le \Prob \{  \bar{E}_v \cup \bar{f}_v(\bar{E}_{v}) \subset E(\Gamma) \}
    = p^{|\bar{E}_v \cup \bar{f}_v(\bar{E}_{v})|}
    \le  \exp \Big( - \log(1/p) \frac{3}{2} k \Big).
$$
Therefore, by the union bound argument, \eqref{eq: IVDataStructure} follows as
long as we assume that $\log\frac{ 2000e}{p}<1.4\log(1/p)$.
\end{proof}

\medskip

As a conclusion of this section, we verify that Theorem~\ref{309847109847} implies Theorem~\ref{19471048710498}.
\begin{proof}[Proof of Theorem~\ref{19471048710498}]
Let the sequence of parameters $(p_n)_{n=1}^\infty$ satisfy
$n^{-1/2}\log^{\max(C,2)}n\leq p_n\leq \min(c,0.0001)$
for all large $n$, where the constants $C,c>0$
are taken from the statement of Theorem~\ref{309847109847}.
For every $n$, let $\Gamma_n$ be a labeled $G(n,p_n)$ Erd\H os--R\'enyi graph on $\{1,2,\dots,n\}$, and let 
$\Event_n$ be the event
\begin{align*}
\Event_n:=\big\{
&\mbox{The unlabeled
graph on $n$ vertices corresponding to $\Gamma_n$}\\
&\mbox{is uniquely reconstructable from the multiset of its $1$--neighborhoods}\big\}.
\end{align*}
Our goal is to show that $\lim\limits_{n\to\infty}\Prob(\Event_n^c)=0$.
Define auxiliary events
\begin{align*}
\Event_n':=\big\{
&\mbox{For every $1\leq i\leq n$ and every vertex $j$ of $\Gamma_n$
adjacent to $i$,}\\
&\mbox{there is a vertex $k\in V(N_{\Gamma_n}(i))\setminus\{i,j\}$
not adjacent to $j$}
\big\}.
\end{align*}
The event $\Event_n'$ ensures the special role of the centers
of the $1$--neighborhoods of $\Gamma_n$, and was previously considered in \cite{GM}
(see \cite[Lemma~7]{GM}).
Specifically, conditioned on any realization of $\Gamma_n$ from
$\Event_n'$, for every vertex $1\leq i\leq n$,
{\it every} choice of a labeled graph $G$ and a vertex $v$ of $G$,
whenever $\phi$ is an isomorphism of $N_{G}(v)$ onto $N_{\Gamma_n}(i)$
then necessarily $\phi(v)=i$.
This implies that everywhere on
the intersection $\Event_n^c\cap \Event_n'$, there is a graph
$\tilde\Gamma_n$ not isomorphic to $\Gamma_n$ such that
for every $1\leq i\leq n$, the $1$--neighborhoods $N_{\Gamma_n}(i)$
and $N_{\tilde\Gamma_n}(i)$ are isomorphic with fixed point $i$.
We can further define
$\tilde\Gamma_n$ to be equal to $\Gamma_n$ everywhere on the complement of
$\Event_n^c\cap \Event_n'$, to obtain a sequence of random graphs $(\tilde\Gamma_n)_{n=1}^\infty$.
Applying Theorem~\ref{309847109847}, we get that $\tilde \Gamma_n=\Gamma_n$
with probability $1-n^{-\omega(1)}$, and hence
$$
\Prob(\Event_n^c\cap \Event_n')=n^{-\omega(1)}.
$$
Thus, to complete the proof it is sufficient to check that
$\Prob(\Event_n')=1-n^{-\omega(1)}$. At this
stage we can either rely on \cite[Lemma~7]{GM}
(with a slightly adjusted range of $p_n$ in that argument), or
apply our Proposition~\ref{prop: GammaTypical}. We will use the latter.
Observe that conditioned on any realization
of $\Gamma_n$ from
$(\Event_n')^c$, there is a pair of vertices $i\neq j$ such that
the set of common neighbors of $i$ and $j$ has size 
$|V(N_{\Gamma_n}(i))|-2$.
Combining this with the probability bound for the event $\Event_{typ}^c$
considered in Proposition~\ref{prop: GammaTypical}
(specifically, the estimates from \eqref{eq: EtypDegree} and~\eqref{eq: EtypCommonNeighbors2}) and the assumptions on $p_n$, we get
$$
\Prob\big(\Event_n'\big)=1-n^{-\omega(1)}.
$$
The result follows.
\end{proof}

\section{Proof of Theorem~\ref{3u14o21u4oi}}\label{s:entropy}

Our proof is based on using the notion of the Shannon entropy.
Recall that the Shannon entropy of a discrete random variable $X$
taking values in some set $T$ is defined as
$$
H_2(X):=-\sum_{t\in T}\Prob\{X=t\}\,\log_2\,\Prob\{X=t\}.
$$
It was shown in \cite{Choi} that
the Shannon entropy of an unlabeled Erd\H os--R\'enyi graph $G$ with parameters $p$ and $m$ satisfying
$\min(m p,m-m p)=\omega(\log m)$ is given by the formula
$$
H_2(G)={m\choose 2}\;\bigg(p\log_2\frac{1}{p}+(1-p)\log_2\frac{1}{1-p}\bigg) - \log_2 m!+ o(1).
$$
Thus, the entropy of the unlabeled graph differs from the entropy of the labeled Erd\H os--R\'enyi graph with the same parameters
by the term $- \log_2 m!+ o(1)$ which can be viewed as a correction obtained by the vertex permutations.
Note that for $mp=\omega(\log m)$ and assuming $p=o(1)$, we get
\begin{equation}\label{313241545151}
H_2(G)=(1+o(1))\frac{m^2p}{2}\,\log_2\frac{1}{p}
\end{equation}
as $m\to\infty$.

Let $\Gamma_n$ be the unlabeled $G(n,p_n)$ Erd\H os--R\'enyi graph, and
let $N_{\Gamma_n,1}$ be the (unlabeled) $1$--neighborhood of a uniform random vertex $v$ of $\Gamma_n$
(where $v$ is \co{assumed} to be independent from $\Gamma_n$).
\co{ Denoting  the set of neighbors of $v$ by $A$}, the size of $A$
is a Binomial($n-1,p_n$) random variable. \co{Conditioned on the size of $A$}, the induced subgraph 
of $N_{\Gamma_n,1}$ obtained from $N_{\Gamma_n,1}$ by removing $v$
is itself an unlabeled Erd\H os--R\'enyi graph
on $|A|$ vertices with parameter $p_n$.
Thus, using the chain rule for the entropy,
\begin{align*}
H_2(N_{\Gamma_n,1})
=\sum_{k=0}^{n-1}{n-1\choose k}p_n^{k}(1-p_n)^{n-1-k}
H_{2,k,p_n}
+H_2(\mbox{Binomial($n-1,p_n$)}),
\end{align*}
where for each $k$, $H_{2,k,p_n}$
is the entropy of an unlabeled Erd\H os--R\'enyi graph on $k$ vertices, with parameter $p_n$. 
In the range $\omega(n^{-1}\log n)=p_n=o(1)$, we simply bound each $H_{2,k,p_n}$
by the entropy of the corresponding labeled Erd\H os--R\'enyi graph to obtain
\begin{align}
H_2(N_{\Gamma_n,1})&=O\bigg(
\sum_{k=2}^{n-1}{n-1\choose k}p_n^{k}(1-p_n)^{n-1-k}\Big(k^2 p_n\log \frac{1}{p_n}\Big)
+H_2(\mbox{Binomial($n-1,p_n$)})\bigg)\nonumber\\
&=O\bigg(
n^2p_n^3\log \frac{1}{p_n}
+\log (p_nn)\bigg),\quad\mbox{for $\omega(n^{-1}\log n)=p_n=o(1)$.}\label{1-974109473-97498}
\end{align}

\bigskip

Now, everything is ready to prove Theorem~\ref{3u14o21u4oi}.
Note that in view of the above estimates the condition $\omega(n^{-1}\log n)=p_n=o(n^{-1/2})$
implies that $H_2(\Gamma_n)=\omega\big( n\,H_2(N_{\Gamma_n,1})\big)$.
Thus, it is enough to show that under the assumption $H_2(\Gamma_n)=\omega\big( n\,H_2(N_{\Gamma_n,1})\big)$,
the unlabeled $G(n,p_n)$ Erd\H os--R\'enyi graph $\Gamma_n$ is a.a.s non-reconstructable.
Denote by $\mathcal N_n$ the multiset of all unlabeled $1$--neighborhoods of $\Gamma_n$.
By the subadditivity of the entropy, the entropy of $\mathcal N_n$ is bounded above by $n\,H_2(N_{\Gamma_n,1})$,
so our assumption implies
$$
H_2(\Gamma_n)=\omega\big(H_2(\mathcal N_n)\big).
$$
For each $n$, we denote by $\Event^{(n)}$ the event that $\Gamma_n$
is reconstructable from $\mathcal N_{n}$. Note that $\Event^{(n)}$
is measurable with respect to $\Gamma_{n}$,
and that on the event $\Event^{(n)}$, $\Gamma_n$ is a function of $\mathcal N_{n}$,
so that, as long as $\Prob(\Event^{(n)})>0$, we have $H_2(\mathcal N_{n}\,|\,\Event^{(n)})
\geq H_2(\Gamma_n\,|\,\Event^{(n)})$.
To prove the assertion of the theorem, we will argue by contradiction.
Namely, we will assume that there is $\delta>0$ and a subsequence of events $\Event^{(n_k)}$
such that
$\Prob(\Event^{(n_k)})\geq \delta$ for all $k$.
For arbitrary discrete random variable $Y$ and arbitrary non-zero probability event $\Event$ we have,
\co{ 
$$
H_2(Y | \Event) \le \Prob( \Event)^{-1} H_2(Y),
$$
which is a specific instance of the general fact that the
conditional entropy of $Y$ never exceeds its [unconditional] entropy: 
$$
H_2(Y) \ge \Prob( \Event )\, H_2(Y | \Event) + \Prob( \Event^c)\, H_2(Y | \Event^c)
$$
(see, for example, \cite[Theorem~2.6.5]{inftheory}).}

Hence, we obtain from our assumptions
$$
H_2(\Gamma_{n_k}\,|\,\Event^{(n_k)})\leq \delta^{-1}H_2(\mathcal N_{n_k})=o(H_2(\Gamma_{n_k})).
$$

Thus, we arrive at the relation $H_2(\Gamma_{n_k}\,|\,\Event^{(n_k)})=o(H_2(\Gamma_{n_k}))$.
It is not difficult to check, however, that the last assertion
is false under our assumption on the probabilities of $\Event^{(n_k)}$.
Indeed, let $\Event_{n_k}'$ be the event that the graph $\Gamma_{n_k}$
has at most $\frac{n(n-1)p_n}{2}+\sqrt{p_n}\,n\log n$ edges.
By the chain rule for entropy, and in view of the restriction that $\Event_{n_k}',\Event^{(n_k)}$ are measurable with respect to $\Gamma_n$, the sum
\begin{align} \label{eq: entropy00}
&\Prob(\Event^{(n_k)})\,H_2(\Gamma_{n_k}\,|\,\Event^{(n_k)})
+\Prob((\Event^{(n_k)})^c\cap \Event_{n_k}')\,H_2(\Gamma_{n_k}\,|\,(\Event^{(n_k)})^c\cap \Event_{n_k}')\nonumber \\
&\phantom{\Prob(\Event^{(n_k)})\,H_2(\Gamma_{n_k}\,|\,\Event^{(n_k)})} +\Prob((\Event^{(n_k)})^c\cap (\Event_{n_k}')^c)\,
H_2(\Gamma_{n_k}\,|\,(\Event^{(n_k)})^c\cap (\Event_{n_k}')^c)\nonumber \\
=&
 H_2(\Gamma_{n_k}) -\Prob(\Event^{(n_k)})\log_2\frac{1}{\Prob(\Event^{(n_k)})}\\
&\phantom{H_2(\Gamma_{n_k})}-\Prob((\Event^{(n_k)})^c\cap \Event_{n_k}')\log_2\frac{1}{\Prob((\Event^{(n_k)})^c\cap \Event_{n_k}')} \nonumber \\
&\phantom{H_2(\Gamma_{n_k})}-\Prob((\Event^{(n_k)})^c\cap (\Event_{n_k}')^c)\log_2\frac{1}{\Prob((\Event^{(n_k)})^c\cap (\Event_{n_k}')^c)}. \nonumber 
\end{align}
Applying Bernstein's inequality \eqref{eq: BernStein}, we get
$$
\Prob((\Event_{n_k}')^c)\leq 2\exp\bigg(-\frac{(\sqrt{p_n}\,n\log n)^2/2}{p_n{n\choose 2}+\sqrt{p_n}\,n\log n}\bigg)=n^{-\omega(1)}.
$$
This, together with \eqref{eq: entropy00}, the condition $H_2(\Gamma_{n_k}\,|\,\Event^{(n_k)})=o(H_2(\Gamma_{n_k}))$,
and the trivial upper bound $H_2(\Gamma_{n_k}\,|\,(\Event^{(n_k)})^c\cap (\Event_{n_k}')^c)\leq n(n-1)/2$,
give
\begin{align*}
& \Prob((\Event^{(n_k)})^c\cap \Event_{n_k}')\,H_2(\Gamma_{n_k}\,|\,(\Event^{(n_k)})^c\cap \Event_{n_k}') \\
 \ge&\co{   
H_2(\Gamma_{n_k}) -O(1) - \Prob(\Event^{(n_k)})\,H_2(\Gamma_{n_k}\,|\,\Event^{(n_k)})
 - \Prob((\Event^{(n_k)})^c\cap (\Event_{n_k}')^c)\,
H_2(\Gamma_{n_k}\,|\,(\Event^{(n_k)})^c\cap (\Event_{n_k}')^c)} \\
\geq& \co{ H_2(\Gamma_{n_k}) -O(1) -  o( H_2(\Gamma_{n_k})) - n^{-\omega(1)} n(n-1)/2 }\\
\geq& (1-o(1))H_2(\Gamma_{n_k})-O(1),
\end{align*}
where $H_2(\Gamma_{n_k})=(1+o(1))\frac{n^2p_n}{2}\log_2\frac{1}{p_n}$.
However, since on the event $\Event_{n_k}'$ the cardinality of the range of possible realizations of $\Gamma_{n_k}$
is bounded above by $\exp\big((1+o(1))\big(\frac{n(n-1)p_n}{2}+\sqrt{p_n}\,n\log n\big)\log\frac{e}{p_n}\big)$,
we have
$$
H_2(\Gamma_{n_k}\,|\,(\Event^{(n_k)})^c\cap \Event_{n_k}')
\leq (1+o(1))\bigg(\frac{n(n-1)p_n}{2}+\sqrt{p_n}\,n\log n\bigg)\log_2\frac{1}{p_n}.
$$
Since $\Prob((\Event^{(n_k)})^c\cap \Event_{n_k}')\leq 1-\delta$, we finally get
$$
(1-\delta)(1+o(1))\big(\frac{n^2p_n}{2}+\sqrt{p_n}\,n\log n\big)\log_2\frac{1}{p_n}
\geq (1-o(1))\frac{n^2p_n}{2}\log_2\frac{1}{p_n}-O(1),
$$
which is clearly false.
This contradiction implies the assertion of the theorem.

\section{Further questions}\label{s:further}

Our work leaves unaddressed the question of reconstructability of unlabeled Erd\H os--R\'enyi graphs from $2$--neigh\-borhoods.
In \cite{GM}, it was shown that for any constant $\varepsilon>0$ and for the sequence $n^{-1+\varepsilon}
\leq p_n\leq n^{-3/4-\varepsilon}$, the unlabeled $G(n,p_n)$ random graph is not reconstructable with probability $1-o(1)$.
This result has recently been strengthened in \cite{JKRS},
where it was proved that reconstruction from $2$--neighborhoods
is impossible a.a.s in the range $\omega(n^{-5/4})= p_n\leq \frac{1}{3}n^{-3/4}\log^{1/4} n$.
Non-reconstructability results for $1$-- and $2$-- neighborhoods obtained in \cite{GM} rely on
comparison between the number of ``typical'' realizations of $n$--tuples of neighborhoods
and the number of ``typical'' realizations of a $G(n,p_n)$ graph (the word ``typical'' refers to certain graph statistics,
but we prefer to avoid technical details here). 
This counting argument for $2$--neighborhoods hits a natural threshold around $p\sim n^{-3/4}$
which suggests the following problem (stated earlier in \cite{GM}):
\begin{Problem}
Let $\varepsilon>0$ be any constant, and let the sequence $(p_n)_{n=1}^\infty$
satisfy
$n^{-3/4+\varepsilon}
\leq p_n$ for large $n$. For each $n$, let $\Gamma_n$ be unlabeled Erd\H os--R\'enyi $G(n,p_n)$ graph.
Is it true that $\Gamma_n$ is reconstructable from its $2$--neighborhoods asymptotically almost surely?
\end{Problem}

We refer to the recent paper \cite{JKRS} where a progress towards
solving the above problem was made and, in particular, where it was shown that
for some small positive $\delta$ and for $p_n\geq n^{-2/3-\delta}$, the graph $\Gamma_n$
is reconstructable with high probability from its $2$--neighborhoods.

\medskip

An interesting phenomenon regarding reconstructability of random graphs
emerges when comparing the Erd\H os--R\'enyi and uniform $d$--regular models.
Whereas
reconstruction of Erd\H os--R\'enyi graphs from $r$--neighborhoods for $r\geq 3$
is essentially based on fluctuations of vertex degrees,
in the corresponding $d$--regular case the only source of information
about the global graph structure is 
in the {\it arrangement of short cycles} within the neighborhoods.
For that reason, we expect that reconstruction/non-reconstruction thresholds
for $d$--regular graphs
from $1$-- and $2$--neighborhoods
essentially coincide with those for the corresponding Erd\H os--R\'enyi graphs,
but the threshold for reconstruction from $3$--neighborhoods for $d$--regular random graphs
should be much larger than in the Erd\H os--R\'enyi setting:
\begin{Problem}[Reconstruction from $1$-- and $2$--neighborhoods for random $d$--regular graphs]
Is the reconstructa\-bility/non-reconstructability threshold from $1$-- and
$2$--neighborhoods for random regular graphs
matches respective [conjectured or verified] thresholds for the Erd\H os--R\'enyi graphs with corresponding average degree?

More precisely, let $\varepsilon>0$, and
for each $n$, let $\Gamma_{n}$ be a uniform random $d_n$--regular unlabeled graph on $n$ vertices.
\begin{itemize}
    \item[(a)] Is $\Gamma_{n}$ reconstructable
    from its $1$--neighborhoods a.a.s in the regime $n^{1/2+\varepsilon}\leq d_n$?
    \item[(b)] Is $\Gamma_{n}$ reconstructable
    from its $2$--neighborhoods a.a.s for $n^{1/4+\varepsilon}\leq d_n$?
    \item[(c)] Is $\Gamma_{n}$ non-reconstructable a.a.s
    from its $1$--neighborhoods for 
    $n^{\varepsilon}\leq d_n\leq n^{1/2-\varepsilon}$    
    and from its $2$--neighborhoods as long as $n^{\varepsilon}\leq d_n\leq n^{1/4-\varepsilon}$?
\end{itemize}
\end{Problem}
\begin{Problem}[Reconstruction from $3$--neighborhoods for random $d$--regular graphs]
For each $n$, let $\Gamma_{n}$ be a uniform random $d_n$--regular unlabeled graph on $n$ vertices.
Is it true that there is a constant $c>0$ such that $\Gamma_{n}$ is {\bf not} reconstructable
    from its $3$--nighborhoods a.a.s. whenever $n^{\varepsilon}\leq d_n\leq n^{c-\varepsilon}$
    for all large $n$ and some $\varepsilon>0$?
\end{Problem}



\begin{thebibliography}{99}

\bibitem{AMRW}
{
R. Arratia, D. Martin, G. Reinert, and M. S. Waterman, Poisson
process approximation for sequence repeats, and sequencing by hybridization, J. of Comp. Bio. 3 (1996), 425--463.
}


\bibitem{BBN}
{
P. Balister, B. Bollob\'{a}s\ and\ B. Narayanan, Reconstructing random jigsaws, in {\it Multiplex and multilevel networks}, 31--50, Oxford Univ. Press, Oxford. MR3931403
}

\bibitem{Ber46} 
{
S. N. Bernstein, The Theory of Probabilities, Gostechizdat, Moscow, (1946)
}


\bibitem{Bondy}
{
J. A. Bondy, A graph reconstructor's manual, in {\it Surveys in combinatorics, 1991 (Guildford, 1991)}, 221--252, London Math. Soc. Lecture Note Ser., 166, Cambridge Univ. Press, Cambridge. MR1161466
}

\bibitem{BFM}
{
C. Bordenave, U. Feige\ and\ E. Mossel, Shotgun assembly of random jigsaw puzzles, Random Structures Algorithms {\bf 56} (2020), no.~4, 998--1015. MR4101351
}

\bibitem{Choi}
{
Y. Choi\ and\ W. Szpankowski, Compression of graphical structures:
fundamental limits, algorithms, and experiments, IEEE Trans. Inform. Theory {\bf 58} (2012), no.~2, 620--638. MR2917956
}

\bibitem{inftheory}
{
T. M. Cover\ and\ J. A. Thomas, {\it Elements of information theory}, second edition, Wiley-Interscience, Hoboken, NJ, 2006. MR2239987
}

\bibitem{Czajka}
{
T. Czajka\ and\ G. Pandurangan, Improved random graph isomorphism, J. Discrete Algorithms {\bf 6} (2008), no.~1, 85--92. MR2398207
}

\bibitem{DJM}
{
J. Ding, Y. Jiang, H. Ma,
Shotgun threshold for sparse Erd\H os-R\'enyi graphs,
arXiv:2208.09876
}

\bibitem{DL}
{
J. Ding, H. Liu,
Shotgun assembly threshold for lattice labeling model,
arXiv:2205.01327
}

\bibitem{GM}
{
J. Gaudio, E. Mossel,
Shotgun Assembly of Erd\H os--R\'enyi Random Graphs,
Electronic Communications in Probability {\bf 27} (2022), 1--14. 
}

\bibitem{GRS}
{
J. Gaudio, M. Z. R\'acz, A. Sridhar,
Local canonical labeling of Erd\H os--R\'enyi random graphs,
arXiv:2211.16454
}

\bibitem{JS}
{
P. Jacquet\ and\ W. Szpankowski, Entropy computations via analytic de-Poissonization, IEEE Trans. Inform. Theory {\bf 45} (1999), no.~4, 1072--1081. MR1686243
}

\bibitem{JKRS}
{
T. Johnston, G. Kronenberg, A. Roberts, A. Scott,
Shotgun assembly of random graphs,
arXiv:2211.14218
}

\bibitem{Kelly}
{
P. J. Kelly, A congruence theorem for trees, Pacific J. Math. {\bf 7} (1957), 961--968. MR0087949
}

\bibitem{M}
{
A. Martinsson, A linear threshold for uniqueness of solutions to random jigsaw puzzles, Combin. Probab. Comput. {\bf 28} (2019), no.~2, 287--302. MR3922781
}

\bibitem{MBT}
{
A. S. Motahari, G. Bresler\ and\ D. N. C. Tse, Information theory of DNA shotgun sequencing, IEEE Trans. Inform. Theory {\bf 59} (2013), no.~10, 6273--6289. MR3106829
}

\bibitem{MosselRoss}
{
E. Mossel\ and\ N. Ross, Shotgun assembly of labeled graphs, IEEE Trans. Network Sci. Eng. {\bf 6} (2019), no.~2, 145--157. MR3969756
}

\bibitem{MS}
{
E. Mossel, N. Sun,
Shotgun assembly of random regular graphs,
arXiv:1512.08473
}

\bibitem{NPS}
{
R. Nenadov, P. Pfister\ and\ A. Steger, Unique reconstruction threshold for random jigsaw puzzles, Chic. J. Theoret. Comput. Sci. {\bf 2017}, Art. 2, 16 pp. MR3672682
}

\bibitem{PRS}
{
M. Przykucki, A. Roberts, A. Scott,
Shotgun reconstruction in the hypercube,
Random Structures and Algorithms, 2021, https://doi.org/10.1002/rsa.21028
}

\bibitem{Ulam}
{
S. M. Ulam, {\it A collection of mathematical problems}, Interscience Tracts in Pure and Applied Mathematics, no. 8, Interscience Publishers, New York, 1960. MR0120127
}

\end{thebibliography}
\end{document}